\documentclass[12pt,reqno,a4paper]{amsbook}
\usepackage{amssymb}
\usepackage{amsmath}
\makeindex

\newcommand{\ba}{\begin{array}}
\newcommand{\eea}{\end{eqnarray}}
\newcommand{\ea}{\end{array}}

\newcommand{\vare}{\varepsilon}

\newcommand{\I}{\mathbb I}
\newcommand{\D}{\mathbb D}

\newcommand{\R}{\mathbb R}%

\numberwithin{section}{chapter}
\newtheorem{definition}{Definition}[section]
\newtheorem{theorem}[definition]{Theorem}
\newtheorem{lemma}[definition]{Lemma}
\newtheorem{proposition}[definition]{Proposition}
\newtheorem{corollary}[definition]{Corollary}
\newtheorem{example}[definition]{Example}
\newtheorem{remark}[definition]{Remark}

\newtheorem{observation}[definition]{Observation}

\usepackage[matrix,arrow,curve]{xy}
\usepackage[utf8x]{inputenc}
\usepackage{tikz}

\begin{document}
 
\vspace*{1cm}

\begin{center}
\begin{huge}
\textsc{Foliations with geometric structures:\\ \large{An approach through $h$-principle} }
\end{huge}
\end{center}

\vspace{1cm}

\begin{center}
 \textsc{ A Thesis \\ Submitted For The Degree Of\\
Doctor Of Philosophy\\
In Mathematics}
\end{center}
%\vspace{1cm}
\begin{center}
{\scshape By}

\vspace{.8cm}

 \textsc{Sauvik Mukherjee}

\vspace{.3cm}

{\scshape Thesis Supervisor: Prof. Mahuya Datta}
\end{center}

\vspace{.5cm}

\begin{figure}[h]
\begin{center}
\end{center}
\end{figure}
\vspace{.1cm}

\begin{center}
\begin{large}
\textsc{Stat Math Unit\\Indian Statistical Institute\\ Kolkata-700108\\India\\[1ex]}
\end{large}
\end{center}

\thispagestyle{empty}
\cleardoublepage

\newpage

\begin{center}
 \Large Acknowledgements
\end{center}
\thispagestyle{empty}
\mbox{}\vspace{1in}\\
I would like to thank my supervisor Prof. Mahuya Datta for her cooperation and guidance. I would also like to thank Prof. Goutam Mukherjee and Prof. Dishant Pancholi for their teaching and Prof. Amartya k Dutta, Prof. B.V. Rajarama Bhat and Prof. Swagato k Ray for their help in some administrative matters.

\newpage

\begin{center}\Large Preface\end{center}
\thispagestyle{empty}
\mbox{}\vspace{1in}\\
A foliation on a manifold $M$ can be informally thought of as a partition of $M$ into injectively immersed submanifolds, called leaves. In this thesis we study foliations whose leaves carry some specific geometric structures.

The thesis consists of two parts. In the first part we classify foliations on open manifolds whose leaves are either locally conformal symplectic or contact manifolds. These foliations can be described by some higher geometric structures - namely the Poisson and the Jacobi structures. In the second part of the thesis, we consider foliations on open contact manifolds whose leaves are contact submanifolds of the ambient space.

Theory of $h$-principle plays the central role in deriving the main results of the thesis. It is a theory rich in topological techniques to solve partial differential relations which arise in connection with topology and geometry. All the geometric structures mentioned above satisfy some differential conditions and that brings us into the realm of the $h$-principle theory.

\newpage

\thispagestyle{empty}
\begin{center}\Large{List of common symbols}\end{center}
\mbox{}\vspace{2cm}\\
\begin{tabular}{ll}
$\R^n$ & $n$-dimensional Euclidean space\\

$\I$ & unit interval [0,1]\\

$\D^n$ & unit disc in $\R^n$\\

$\mathbb S^n$ & unit sphere in $\R^n$\\

$L(V,W)$ & space of linear maps from a vector space $V$ to $W$\\

$C^k(M,N)$ & space of $C^k$-maps between manifolds $M$ and $N$\\

$\Gamma^k(X)$ & $C^k$-sections of a fibre bundle $X$\\

$TM$ & tangent bundle of a manifold $M$\\

$T^*M$ & cotangent bundle of a manifold $M$\\

$\wedge^k(E)$ & exterior bundle associated with a vector bundle $E$\\

$\Omega^k(M)$ & space of smooth $k$-forms on $M$\\

$\nu^k(M)$ & space of smooth multivector fields on $M$\\

\emph{Diff}$(M)$ & pseudogroup of local diffeomorphisms of $M$\\

$T\mathcal F$ & tangent space of a foliation $\mathcal F$\\

$\nu\mathcal F$ & normal bundle of a foliation $\mathcal F$\\

$GL_n(\R)$ & general linear group over $\R$\\

$U(n)$ & unitary group of order $n$\\

$BG$ & classifying space of the group $G$\\

$\Gamma_q$ & groupoid of germs of local diffeomorphisms of $\R^q$\\

$B\Gamma_q$ & classifying space of $\Gamma_q$-structures\\

$\Omega_q$ & universal $\Gamma_q$ structure on $B\Gamma_q$

\end{tabular}

\tableofcontents
\chapter{Introduction}

A foliation $\mathcal F$ on a manifold $M$ can be thought of as a partition of the manifold into injectively immersed submanifolds of $M$ called leaves of the foliation. The tangent spaces of the leaves combine together to define the tangent bundle of the foliation $\mathcal F$ which we denote by $T\mathcal F$. The quotient bundle $TM/T\mathcal F$ is referred as the normal bundle of the foliation and is denoted by $\nu\mathcal F$, which plays an important role in the study of foliations. The simplest type of foliations on a manifold are defined by submersions. In this case the level sets of the submersions define the leaves of regular foliations on a manifold. More generally, if a map $f:M\to N$ is transverse to a foliation $\mathcal F_N$ on $N$ then the inverse image of $\mathcal F_N$ under $f$ is a foliation on $M$. If $M$ is an open manifold, then it follows from Gromov-Phillips Transversality Theorem (\cite{gromov},\cite{phillips},\cite{phillips1}) that the homotopy classes of maps $M\to N$ transversal 
to $\mathcal F_N$ are in one to one correspondence with the homotopy classes of epimorphisms $TM\to \nu(\mathcal F_N)$.

Gromov-Phillips Theorem can be translated into the language of $h$-principle and can be deduced from a general theorem due to Gromov (\cite{gromov_pdr}). In the vocabulary of $h$-principle, a subset $\mathcal R$ of $J^r(M,N)$, the space of $r$-jets of maps from a manifold $M$  to $N$, is called an $r$-th order\emph{ partial differential relation} or simply a relation. If $\mathcal R$ is open then it is called an open relation.
A (continuous) section $\sigma:M\to J^r(M,N)$ of the $r$-jet bundle whose image is contained in $\mathcal R$ is referred as a section of $\mathcal R$. A \emph{solution} of $\mathcal R$ is a smooth map $f:M\to N$ whose $r$-jet extension $j^r_f:M\to J^r(M,N)$ is a section of $\mathcal R$. The space of solutions, $Sol(\mathcal R)$, has the $C^\infty$-compact open topology, whereas the space of sections of $\mathcal R$, $\Gamma(\mathcal R)$, has the $C^0$-compact open topology. The relation $\mathcal R$ is said to satisfy the \emph{parametric $h$-principle} if the $r$-jet map $j^r:Sol(\mathcal R)\to\Gamma(\mathcal R)$ is a weak homotopy equivalence. Thus the $h$-principle reduces a differential topological problem to a problem in algebraic topology.

The diffeomorphism group of $M$ acts on the space of maps $M\to N$ by pull-back operation. This extends to an action of \emph{Diff}$(M)$, the pseudogroup of local diffeomorphisms of $M$, on the space of $r$-jets. If $\mathcal R$ is invariant under this action then we say that $\mathcal R$ is \emph{Diff}$(M)$-invariant. Gromov proved in \cite{gromov} that if $M$ is an open manifold then every open, \emph{Diff}$(M)$-invariant relation on $M$ satisfies the parametric $h$-principle. We shall refer to this result as Open Invariant Theorem for future reference. Using the full strength of the hypothesis on $\mathcal R$, one first proves that $\mathcal R$ satisfies the parametric $h$-principle near any submanifold $K$ of positive codimension. A key point about an open manifold $M$ is that it has the homotopy type of a CW complex $K$ of dimension strictly less than that of $M$. Furthermore, $M$ admits deformations into arbitrary open neighbourhood of $K$. As a result, open manifolds exhibit tremendous amount of 
flexibility. This allows the $h$-principle to be lifted from an open neighbourhood of $K$ to all of $M$. Since transversality is a differential condition on the derivative of a function, these are solutions to a first order differential relation $\mathcal R_T$. Transversality being a stable property, the relation $\mathcal R_T$ is open. Furthermore, the relation is clearly \emph{Diff}$(M)$-invariant since the pull-back of a map $M\to N$ transverse to a foliation $\mathcal F_N$ on $N$ by a diffeomorphism of $M$ is also transverse to $\mathcal F_N$. Thus the Gromov-Phillips Theorem says that $\mathcal R_T$ satisfies the parametric $h$-principle.

The transversality theorem mentioned above plays a central role in the classification of foliations on open manifolds.
Formally, the codimension $q$ foliations on a manifold $M$ are defined by local submersions $f_i:U_i\to \R^q$ for some open covering $\mathcal U=\{U_i,i\in I\}$, such that there are diffeomorphisms $g_{ij}:f_i(U_i)\to f_j(U_j)$ satisfying the relations $g_{ij}f_i=f_j$ and cocycle conditions. The germs of the diffeomorphisms $g_{ij}$ at points $f_i(x)$, $x\in U_i$, define maps $\gamma_{ij}:U_i\cap U_j\to \Gamma_q$, where $\Gamma_q$ is the topological groupoid of germs of local diffeomorphisms of $\R^q$. For any topological groupoid $\Gamma$, there is a notion of $\Gamma$-structure (\cite{haefliger},\cite{haefliger1}).
Following Milnor's topological join construction (\cite{husemoller}) to define classifying space of principal $G$-bundles, one can construct a topological space $B\Gamma$ with universal $\Gamma$-structure $\Omega$ such that $[M,B\Gamma]$, the homotopy classes of maps $M\to B\Gamma$ classifies the $\Gamma$-structures up to homotopy (\cite{haefliger}, \cite{haefliger1}). In particular, when $\Gamma=\Gamma_q$, the derivative map $d:\Gamma_q\to GL_q(\R)$ induces a continuous map $Bd:B\Gamma_q\to BGL_q(\R)$ into the classifying space $BGL_q(\R)$ of real vector bundles of rank $q$.  If $\tilde{f}$ is a classifying map of a $\Gamma_q$ structure $\omega$, then $Bd\circ\tilde{f}$ classifies the normal bundle $\nu(\omega)$ associated to $\omega$. Furthermore, there is a vector bundle $\nu\Omega_q$ over $B\Gamma_q$ which is `universal' for the bundles $\nu(\omega)$ as $\omega$ runs over all $\Gamma_q$ structures on $M$. Haefliger cocycles of a foliation $\mathcal F$ on $M$ naturally give rise to a $\Gamma_q$-structure 
$\omega_{\mathcal F}$ on $M$. It is a general fact that $\nu(\omega_{\mathcal F})$ is isomorphic to the normal bundle of the foliation $\mathcal F$. Hence, $\nu(\omega_{\mathcal F})$ admits an embedding into the tangent bundle $TM$ and consequently, the classifying map $\tilde{f}$ can be covered by an epimorphism $F:TM\to \nu\Omega_q$. Haefliger observed, that any $\Gamma_q$-structure on $M$ can indeed be defined as the inverse image of a foliation by an embedding $e:M\to (N,\mathcal F_N)$ into a foliated manifold $N$. The $\Gamma_q$ structure is a foliation if and only if $e$ is transverse to $\mathcal F_N$. Thus, he reduced the homotopy classification of foliations on open manifolds to Gromov-Phillips Theorem and showed that the `integrable' homotopy classes of codimension $q$ foliations on an open manifold $M$ are in one-one correspondence with the homotopy classes of epimorphism $F:TM\to \nu\Omega_q$ (\cite{haefliger1}). In particular, it shows that if a map $f:M\to BGL(q)$ classifying the normal bundle 
of a codimension $q$ distribution $D$ on $M$ lifts to a map $\tilde{f}:M\to B\Gamma_q$, then the distribution $D$ is homotopic to one which is integrable, provided $M$ is open.
Soon after the work of Haefliger, Thurston extended the classification of foliations to closed manifolds thereby completing the classification problem (\cite{thurston},\cite{thurston1}). Thurston showed that the `concordant classes' of foliations are in one to one correspondence with the homotopy classes of $\Gamma_q$ structures $\mathcal H$ together with the `concordance classes' of bundle monomorphisms $\nu(\mathcal H)\to TM$. The proof of these results are very much involved and beyond the scope of our study.

In the thesis, we study foliations whose leaves carry some specific geometric structures. In particular we are interested in foliations whose leaves are symplectic, locally conformal symplectic or contact manifolds.
In his seminal thesis, Gromov had shown that the obstruction to the existence of a contact or a symplectic form on open manifolds is purely topological. Gromov obtained these results as applications to Open Invariant Theorem mentioned above. In a recent article (\cite{fernandes}), Fernandes and Frejlich proved that a foliation with a leafwise non-degenerate 2-form is homotopic through such pairs to a foliation with a leafwise symplectic form. Symplectic foliations on a manifold $M$ can be explained in terms of regular Poisson structures on the manifold (\cite{vaisman}). Recall that a Poisson structure $\pi$ is a bivector field satisfying the condition $[\pi,\pi]=0$, where the bracket denotes the Schouten bracket of multivector fields (\cite{vaisman}). The bivector field $\pi$ induces a vector bundle morphism $\pi^\#:T^*M\to TM$ by $\pi^\#(\alpha)(\beta)=\pi(\alpha,\beta)$ for all $\alpha,\beta\in T^*_xM$, $x\in M$.
The characteristic distribution $\mathcal D=\text{ Image }\pi^\#$ is, in general, a singular distribution which, however, integrates to a foliation. The restriction of the Poisson structure to a leaf of the foliation has the maximum rank and so we obtain a symplectic form on the leaf by dualizing $\pi$. Thus, the characteristic foliation is a (singular) symplectic foliation.  A Poisson bivector field $\pi$ is said to be \emph{regular} if the rank of $\pi^\#$ is constant. In this case the characteristic foliation is a regular symplectic foliation on $M$. On the other hand, given a regular symplectic foliation $\mathcal F$ on $M$ one can associate a Poisson bivector field $\pi$ having $\mathcal F$ as its characteristic foliation. Since a symplectic form on a manifold corresponds to a non-degenerate Poisson structure, Gromov's result on the existence of symplectic form is equivalent to saying that a non-degenerate bivector field on an open manifold is homotopic to a non-degenerate Poisson structure. In the same 
light, the result of Fernandes and Frejlich \cite{fernandes} can be translated into the statement that a regular bivector field $\pi_0$ is homotopic to a Poisson bivector field, provided the manifold is \emph{open} and the characteristic distribution of $\pi_0$ is integrable. However, this can not be done without deforming the underlying characteristic foliation Im$\pi_0^{\#}$. It would be pertinent to recall a result of Bertelson which preceded \cite{fernandes}. She showed that a leafwise non-degenerate 2-form on a foliation need not be homotopic to a leafwise symplectic form on the same foliation even if $M$ is open (\cite{bertelson}) - in order to keep the underlying foliation constant during homotopy, one needs to impose some additional `open-ness' condition on the foliation itself.

Poisson structures have further generalisations to Jacobi structures which are given by pairs $(\Lambda,E)$ consisting of a bivector field $\Lambda$ and a vector field $E$ on $M$ satisfying the following conditions:
\[[\Lambda,\Lambda]=2E\wedge\Lambda,\ \ \ \ \ [\Lambda,E]=0.\]
If $E=0$ then clearly $\Lambda$ is a Poisson structure on $M$. A Jacobi structure, as in the case of Poisson, is associated with an integrable singular distribution namely, $\mathcal D=\text{Im\,}\Lambda^{\#}+\langle E\rangle$, where $\langle E\rangle$ denotes the distribution generated by the vector field $E$. The leaves of $\mathcal D$ inherit the structures of locally conformal symplectic or contact manifolds according as the dimension of the leaf is even or odd (\cite{kirillov}). In particular, if the characteristic distribution $\mathcal D$ is regular then we obtain either a locally conformal symplectic foliation or a contact foliation on $M$. Motivated by a comment in \cite{fernandes}, we extend the work of Fernandes and Frejlich to give a homotopy classification of contact and locally conformal symplectic foliations. We prove that if an open manifold admits a foliation with a leafwise non-degenerate 2-form then it admits a locally conformal symplectic foliation with its foliated Lee class defined by a 
given cohomology class $\xi \in H_{deR}^{1}(M)$. In the same footing, we show that if there is a foliation on an open manifold with a leafwise almost contact structure then the manifold must admit a contact foliation. We also interprete these results in terms of regular Jacobi structures.

In the second part of the thesis, following the steps of Haefliger we study foliations on open manifolds $M$ in the presence of a contact form $\alpha$ such that the leaves of the foliations are contact submanifolds of $(M,\alpha)$. We first classify those foliations which are obtained by means of maps into a foliated manifold, as in Gromov-Phillips Theorem.
To state it explicitly, let $Tr_\alpha(M,\mathcal F_N)$ denote the space of maps $f:M\to N$ which are transversal to a given foliation $\mathcal F_N$ on $N$ and for which the inverse foliations $f^*\mathcal F_N$ are contact foliations on $M$. Since the contactness property of 1-forms is a stable property, the space $Tr_\alpha(M,\mathcal F_N)$ is realised as the space of solutions to some first order open differential relation $\mathcal R_\alpha$. The space $Tr_\alpha(M,\mathcal F_N)$ is clearly not invariant under \emph{Diff}$(M)$, though it is invariant under the action of contact diffeomorphisms of $M$. This suffices for the $h$-principle of $\mathcal R_\alpha$ near a core $K$ of $M$. In order to lift the $h$-principle to all of $M$, we can not use the ordinary deformations of $M$ into Op\,$K$ - since the relation is not invariant under \emph{Diff}$(M)$ it would not give a homotopy within $Tr_\alpha(M,\mathcal F_N)$. We would have liked to get deformations of $M$ into $Op\,K$ which would keep the contact 
form invariant. We can, however, only show that if $M$ is an open manifold, then there exists a regular homotopy $\varphi_t$ of isocontact immersions into itself such that $\varphi_0=id_M$ and $\varphi_1(M)$ is contained in an arbitrary small neighbourhood of $K$. In fact, we prove a weaker version of Gray's Stability  Theorem for contact forms on open contact manifolds which is one of the main results of the thesis. It may be recalled that a similar result for open symplectic manifolds was earlier obtained by Ginzburg in \cite{ginzburg}. Now coming back to contact set-up, since the composition of an $f\in Tr_\alpha(M,\mathcal F_N)$ with a contact immersion $\varphi$ of $M$ is again an element of $Tr_\alpha(M,\mathcal F_N)$, we can lift the $h$-principle near $K$ to a global $h$-principle on $M$ using the homotopy $\varphi_t$. More generally, we prove an extension of Open Invariant Theorem of Gromov on open contact manifolds $(M,\alpha)$. A similar result was obtained for open symplectic manifolds in \cite{
datta-rabiul}. Proceeding as in Haefliger, we then prove that the 'integrable' homotopy classes of contact foliations are in one-to-one correspondence with the homotopy classes of epimorphisms $F:TM\to \nu\Gamma_q$ such that $\ker F\cap \ker\alpha$ is a symplectic subbundle of $\ker\alpha$ relative to the symplectic form defined by $d\alpha$.

The thesis is organised as follows. We discuss the preliminaries in Chapter 2. This consists of five parts - In the first two sections we recall the preliminaries of symplectic and contact manifolds and review the basic definitions and examples of foliations. In the third section we introduce foliations with geometric structures and review the basic theory of Poisson and Jacobi structures. In the last two section we discuss the language of $h$-principle and some major results including Haefliger's classification theorem which serves as a background of the problems treated in the thesis. In Chapter 3, we give a classification of contact and locally conformal symplectic foliations and then interpret these results in terms of regular Jacobi structures. Chapter 4 is again divided into several sections. In Section 1 we recall a homotopy classification of submersions with symplectic fibres on open symplectic manifolds (\cite{datta-rabiul}) and note that a generalisation of this result leads to homotopy 
classification of symplectic foliations on open symplectic manifolds. In Section 2 we prove a `stability theorem' for contact forms on open contact manifolds. In section 3 we obtain an extension of Open Invariant Theorem of Gromov in the contact set-up. In sections 4 and 5 we prove a contact version of Gromov-Phillips Theorem and discuss some of its special cases. In the final section we obtain a homotopy classification of contact foliations on open contact manifolds.

\chapter{Preliminaries}
\section{Preliminaries of symplectic and contact manifolds}

In this section we review various geometric structures on manifolds which are defined by differential forms. These are already standard in the Mathematics literature and can be found in \cite{mcduff-salamon} and \cite{geiges1}.
\subsection{Symplectic manifolds}
\begin{definition}
{\em An antisymmetric bilinear form $\omega$ on a vector space $V$ defines a linear map $\tilde{\omega}:V\to V^*$ given by $\tilde{\omega}(v)(v')=\omega(v,v')$ for all $v,v'\in V$. The dimension of the image of $\tilde{\omega}$ (which is an even integer) is called the \emph{rank} of $\omega$.  A 2-form $\omega$ is said to be \emph{non-degenerate} if  $\tilde{\omega}$ is an isomorphism; equivalently, if $\omega(v,w)=0$ for all  $w\in V$ implies that $v=0$. A vector space $V$ is called a \emph{symplectic vector space} if there exists a nondegenerate 2-form $\omega$ on it. }
\end{definition}
Since rank of a 2-form is an even integer, a symplectic vector space is even dimensional. If $\dim V=2n$, then $\omega$ is non-degenerate if and only if $\omega^n\neq 0$.

The \emph{symplectic complement} of a subspace $W$ in a symplectic vector space $(V,\omega)$, denoted by $W^{\perp_\omega}$, is defined as \[W^{\perp_{\omega}}=\{v\in V:\omega(v,w)=0 \text{ for all }\ w\in W\}\]
A subspace $W$ of a symplectic vector space $(V,\omega)$ is said to be \emph{symplectic} if the restriction of $\omega$ to $W$ is symplectic. The symplectic complement of a symplectic subspace $W$ is also a symplectic subspace of $V$  and $V=W\oplus W^{\perp_\omega}$.
\begin{definition}
\em{A 2-form $\omega$ on a manifold  $M$ is said to be an \emph{almost symplectic form} \index{almost symplectic form} if its restrictions to the tangent spaces $T_xM$, $x\in M$, are non-degenerate. An almost symplectic form which is also closed is called a \emph{symplectic form}\index{symplectic form} on the manifold. Manifolds equipped with such forms are called almost symplectic and symplectic manifolds respectively.}
\end{definition}

\begin{example}\label{ex:symp}\end{example}
\begin{enumerate}\item The Euclidean space $\mathbb{R}^{2n}$ has a canonical symplectic form given by $\omega_0 = \Sigma_idx_i\wedge dy_i$, where $(x_1,\dots,x_n,y_1,\dots,y_n)$ is the canonical coordinate system on $\R^{2n}$.
\item All oriented surfaces are symplectic manifolds.
\item The 2-sphere $\mathbb{S}^2$ is a symplectic manifold but $\mathbb{S}^{2n}$ are not for $n>1$.
\item The 6 dimensional sphere $\mathbb{S}^6$ is an example of an almost symplectic manifold which is not symplectic.
\item The total space of the cotangent bundle has a canonical symplectic form which is exact.\end{enumerate}
\begin{definition}
\label{D:symplectic vector bundle}
\em{ Let $p:E\to B$ be a vector bundle over a topological space $B$. Let $\wedge^k(E^*)$ denote the $k$-th exterior bundle associated with the dual $E^*$. A section $\omega$ of $\wedge^2(E^*)$ is called a symplectic form on $E$ if $\omega_b$ is a symplectic form on the fiber $E_b$ for all $b\in B$. The pair $(E,\omega)$ is then called a symplectic vector bundle.}
\end{definition}
Clearly, the tangent bundle of a symplectic manifold is a symplectic vector bundle.

\begin{definition}{\em Let $(M,\omega)$ and $(N,\omega')$ be two symplectic manifolds. A diffeomorphism $f:M\to N$  is said to be a \emph{symplectomorphism} if it pulls back the form $\omega'$ onto $\omega$.}
\end{definition}

The following theorem implies that there is no local invariant for symplectic manifolds.
\begin{theorem}(Darboux)
 \label{T:darboux}
Any symplectic manifold $(M^{2n},\omega)$ is locally symplectomorphic to the Euclidean manifold $(\R^{2n},\omega_0)$, where $\omega_0$ is the standard symplectic form defined as in Example~\ref{ex:symp}.
 \end{theorem}

\begin{definition}{\em Two symplectic forms $\omega_0,\omega_1$ on a manifold $M$ are said to be \emph{isotopic} if there is an isotopy $\delta_t$, $t\in [0,1]$, such that $\delta_1^*\omega_0=\omega_1$}. \end{definition}
Therefore, if $\omega_0$ and $\omega_1$ are isotopic they can be joined by a path $\omega_t$ in the space of symplectic forms such that the cohomology class of $\omega_t$ is independent of $t$. Explicitly, one can take $\omega_t=\delta_t^*\omega$ for $t\in\I$. The following theorem due to Moser says that the converse of this is true on a closed manifold.
\begin{theorem}(Moser's Stability Theorem \cite{moser})
 \label{T:moser's stability}
 Let $M$ be a closed manifold (that is, compact and without boundary) and let $\omega_t,t\in \I=[0,1]$ be a family of symplectic forms belonging to the same de Rham cohomology class. Then there exists an isotopy $\{\phi_t\}_{t\in \I}$ of $M$ such that $\phi_0=id_M$ and $\phi_t^*\omega_t=\omega_0$.
 \end{theorem}

A version of Moser's stability theorem for open manifolds was proved by Ginzburg in \cite{ginzburg}. Here we give a version due to Eliashberg.
\begin{theorem}(\cite{eliashberg})
 \label{T:equidimensional-symplectic-immersion}
 Let $(\tilde{M},\tilde{\omega})$ be a symplectic manifold without boundary and let $M$ be an equidimensional submanifold of $\tilde{M}$ with boundary. Suppose that $\omega_t,\ t\in \I$, is a family of symplectic forms on $M$ representing the same cohomology class. If $\tilde{\omega}|_M=\omega_0$, then there exists a regular homotopy $f_t:M\to \tilde{M}$ (that is, a homotopy of immersions) such that $f_0$ is the inclusion $M\to \tilde{M}$ and $f_t^*\tilde{\omega}=\omega_t,\ t\in \I$.
\end{theorem}

We shall obtain a contact analogue of this result in Chapter 4.

\subsection{Locally Conformal Symplectic manifolds}
\begin{definition}{\em A non-degenerate 2-form $\omega$ on a manifold $M$ is said to be \emph{conformal symplectic} \index{conformal symplectic} if there is a nowhere vanishing $C^\infty$ function $f$ on $M$ such that $f\omega$ is a symplectic form.}\end{definition}
\begin{definition}{\em A \emph{locally conformal symplectic} structure on a manifold $M$ is given by a pair $(\omega,\theta)$, where $\omega$ is a non-degenerate 2-form and $\theta$ is a closed 1-form on $M$ satisfying the relation
\begin{equation}d\omega+\theta\wedge\omega=0.\label{def_lcs} \index{locally conformal symplectic} \end{equation}
The form $\theta$ is called the \emph{Lee form} of $\omega$. \index{Lee form}
If $\dim M\geq 4$ then $\omega\wedge -:\Omega^1(M)\to \Omega^3(M)$ is injective because of the non-degeneracy of $\omega$. In this case, $\theta$ is uniquely determined by the relation (\ref{def_lcs}).
}\end{definition}
If $\omega$ is a locally conformal symplectic form, then there is an open covering $\{U_i\}_{i\in I}$ of $M$ such that $d\omega=df_i\wedge \omega$ on $U_i$ for some smooth functions $f_i$ defined on $U_i$. This implies that $d(e^{-f_i}\omega)=0$, that is, $\omega$ is conformal symplectic on each $U_i$. This can be taken as the alternative definition of locally conformal symplectic structure if $\dim M\geq 4$.

\noindent\textbf{Lichnerowicz cohomology.} A closed 1-form $\theta$ on a manifold $M$ defines a coboundary operator $d_\theta:\Omega^*(M)\to \Omega^{*+1}(M)$ by \[d_\theta=d+\theta\wedge\ \index{$d_\theta$},\]
where $d$ is the exterior differential operator on differential forms. Indeed, it is easy to verify that $d_\theta^2\alpha=d\theta\wedge \alpha$ for any differential form $\alpha$ and therefore $d_\theta^2=0$ if and only if $\theta$ is closed. The resulting cohomology is called the Lichnerowicz cohomology which depends only on the cohomology class of $\theta$. A locally conformal symplectic form with Lee form $\theta$ is therefore a $d_\theta$-closed non-degenerate 2-form on $M$.

\subsection{Contact manifolds}
A hyperplane distribution $\xi$ on a manifold $M$ can be locally written as $\xi=\ker \alpha$ for some local 1-form $\alpha$ on $M$. The form $\alpha$ is only unique upto multiplication by a nowhere vanishing function. If $\xi$ is coorientable i.e. when the quotient bundle $TM/\xi$ is a trivial line bundle, then $\xi$ is obtained as the kernel of a global 1-form on $M$ given by the following composition
\[TM\stackrel{q}{\to}TM/\xi\cong M\times\R\stackrel{p_1}{\to}\R,\]
where $q$ is the quotient map and $p_1$ is the projection onto the first factor.

\begin{definition}{\em
Let $M$ be a $2n+1$ dimensional manifold. A hyperplane distribution $\xi$ is called a \emph{contact structure} \index{contact structure} if $\alpha \wedge (d\alpha)^n$ is nowhere vanishing for any local 1-form $\alpha$ defining $\xi$. A global 1-form $\alpha$ for which $\alpha \wedge (d\alpha)^n$ is nowhere vanishing is called a \emph{contact form} \index{contact form} on $M$. The distribution $\ker\alpha$ is then called the \emph{contact distribution of} $\alpha$.}\label{contact_form}
\end{definition}
\begin{example}\end{example}
\begin{enumerate}\item
Every odd dimensional Euclidean space $\R^{2n+1}$ has a canonical contact form given by $\alpha=dz+\sum_{i=1}^nx_i\,dy_i$, where $(x_1,\dots,x_n,y_1,\dots,y_n,z)$ is the canonical coordinate system on $\R^{2n+1}$.
\item
Every even dimensional Euclidean space $\R^{2n}$ has a canonical 1-form $\lambda=\sum_{i=1}^n(x_idy_i-y_idx_i)$ which is called the Liouville form of $\R^{2n}$, where $(x_1,\dots,x_n$, $y_1,\dots,y_n)$ is the canonical coordinate system on $\R^{2n}$. The restriction of $\lambda$ on the unit sphere in $\R^{2n}$ defines a contact form.
\item For any manifold $M$, the total space of the vector bundle $T^*M\times\R\to M$ has a canonical contact form.
\end{enumerate}
If $\alpha$ is a contact form then
\[d'\alpha=d\alpha|_{\ker\alpha}\index{$d'\alpha$}\]
defines a symplectic structure on the contact distribution $\xi=\ker\alpha$. Also, there is a global vector field $R_\alpha$ on $M$ defined by the relations
\begin{equation}\alpha(R_\alpha)=1,\ \ \ i_{R_\alpha}.d\alpha=0,\label{reeb} \index{Reeb vector field}
\end{equation}
where $i_X$ denotes the interior multiplication by the vector field $X$. Thus, $TM$ has the following decomposition:
\begin{equation}TM=\ker\alpha \oplus \ker\,d\alpha,\label{decomposition}\end{equation}
where $\ker\alpha$ is a symplectic vector bundle and $\ker\,d\alpha$ is the 1-dimensional subbundle generated by $R_\alpha$. The vector field $R_\alpha$ is called the \emph{Reeb vector field} of the contact form $\alpha$.

A contact form $\alpha$ also defines a canonical isomorphism $\phi:TM\to T^*M$ between the tangent and the cotangent bundles of $M$ given by
\begin{equation}\phi(X)=i_X d\alpha+\alpha(X)\alpha, \text{ for } X\in TM.\label{tgt_cotgt}\end{equation}
It is easy to see that the Reeb vector field $R_\alpha$ corresponds to the 1-form $\alpha$ under $\phi$.

\begin{definition} {\em Let $(N,\xi)$ be a contact manifold. A monomorphisn $F:TM\to (TN,\xi)$ is called \textit{contact} if $F$ is transversal to $\xi$ and $F^{-1}(\xi)$ is a contact structure on $M$. A smooth map $f:M\to (N,\xi)$ is called \textit{contact} if its differential $df$ is contact.

If $M$ is also a contact manifold with a contact structure $\xi_0$, then a monomorphism $F:TM\to TN$ is said to be \textit{isocontact} if $\xi_0=F^{-1}\xi$ and $F:\xi_0\to\xi$ is conformal symplectic with respect to the conformal symplectic structures on $\xi_0$ and $\xi$. A smooth map $f:M\to N$ is said to be \textit{isocontact} if $df$ is isocontact.

A diffeomorphism $f:(M,\xi)\to (N,\xi')$ is said to be a \emph{contactomorphism} \index{contactomorphism} if $df$ is isocontact.
\index{isocontact map}}
\end{definition}
If $\xi=\ker\alpha$ for a globally defined 1-form $\alpha$ on $N$, then $f$ is contact if $f^*\alpha$ is a contact form on $M$.
Furthermore, if $\xi_0=\ker\alpha_0$ then $f$ is isocontact if $f^*\alpha=\varphi \alpha_0$ for some nowhere vanishing function $\varphi:M\to\R$.

\begin{definition}{\em A vector field $X$ on a contact manifold $(M,\alpha)$ is called a \emph{contact vector field} if it satisfies the relaion $\mathcal L_X\alpha=f\alpha$ for some smooth function $f$ on $M$, where $\mathcal L_X$ denotes the Lie derivation operator with respect to a vector field $X$.}\label{D:contact_vector_field} \index{contact vector field}\end{definition}
Every smooth function $H$ on a contact manifold $(M,\alpha)$ gives a contact vector field $X_H=X_0+\bar{X}_H$ defined as follows:
\begin{equation}X_0=HR_\alpha \ \ \ \text{ and }\ \ \ \bar{X}_H\in \Gamma(\xi) \text{ such that  }i_{\bar{X}_H}d\alpha|_\xi = -dH|_\xi,\label{contact_hamiltonian} \index{contact hamiltonian}\end{equation}
where $\xi=\ker\alpha$; equivalently,
\begin{equation}\alpha(X_H)=H\ \ \text{ and }\ \ i_{X_H}d\alpha=-dH+dH(R_{\alpha})\alpha.\label{contact_hamiltonian1} \end{equation}
The vector field $X_H$ is called the \emph{contact Hamiltonian vector field} of $H$.

If $\phi_t$ is a local flow of a contact vector field $X$, then
\[\frac{d}{dt}\phi_t^*\alpha = \phi_t^*(i_X.d\alpha+d(\alpha(X)))=\phi_t^*(f\alpha)=(f\circ\phi_t)\phi_t^*\alpha.\]
Therefore, $\phi_t^*\alpha=\lambda_t\alpha$, where $\lambda_t=e^{\int f\circ\phi_t \,dt}$. Thus the flow of a contact vector field preserves the contact structure.

\begin{theorem}Every contact form $\alpha$ on a manifold $M$ of dimension $2n+1$ can be locally represented as $dz-\sum_{i=1}^np_i\,dq_i$, where $(z,q_1,\dots,q_n,p_1,\dots,p_n)$ is a local coordinate system on $M$.\end{theorem}

\begin{theorem}(Gray's Stability Theorem (\cite{gray})
 \label{T:gray stability}
If $\xi_t,\ t\in \I$ is a smooth family of contact structures on a closed manifold $M$, then there exists an isotopy $\psi_t,\ t\in \I$, of $M$ such that \[d\psi_t(\xi_0)=\xi_t\ for\ all\ t\in \I\]
 \end{theorem}
Next we shall give some examples of compact domains $U$ with piecewise smooth boundary in a contact manifold which contracts into itself by isocontact embeddings. We shall first recall the formal definition of such domains (\cite{eliashberg}).
\begin{definition}
\em{Let $U$ be a compact domain with piecewise smooth boundary in a contact manifold $(M,\alpha)$; $U$ is called \emph{contactly contractible} if there exists a contact vector field $X$ which is inward transversal to the boundary of $U$ and is such that its flow $\psi_t$ satisfies the following property:
\[\psi_t^*\alpha =h_t\alpha, \text{ where } h_t\to 0\ as\ t\to +\infty.\]
}
\end{definition}
\begin{example}  \label{L:contactly contractible domains}\end{example}
\begin{enumerate}\item The Euclidean ball in $(\mathbb{R}^{2n+1},dz-\Sigma_1^n(x_jdy_j-y_jdx_j))$ centered at the origin;
 \item the semi-ball centred at the origin i.e, one half of the Euclidean ball cut by a hyperplane;
 \item the image of a contactly contractible domain under a $C^1$-small diffeomorphism.\end{enumerate}

\begin{remark}{\em
In Chapter 4, we shall see an extension of \ref{T:gray stability} for non-closed contact manifold, which is one of the main results in the thesis (see Theorem~\ref{T:equidimensional_contact immersion}).}
\end{remark}
We end this section with the concept of a contact submanifold.
\begin{definition} {\em A submanifold $N$ of a contact manifold $(M,\xi)$ is said to be a \emph{contact submanifold} if the inclusion map $i:N\to M$ is a contact map.}\label{contact_submanifold} \index{contact submanifold}\end{definition}

\begin{lemma} A submanifold $N$ of a contact manifold $(M,\alpha)$ is a contact submanifold if and only if $TN$ is transversal to $\xi|_N$ and $TN\cap\xi|_N$ is a symplectic subbundle of $(\xi,d'\alpha)$.\label{L:contact_submanifold}\end{lemma}

\newpage
\section{Preliminaries of  foliations}
In this section we recall definitions of foliations and some primary examples for which our main reference is \cite{moerdijk}. We also review the notion of $\Gamma_q$-structures and its relations with foliations following \cite{haefliger}.
\subsection{Foliations}
Foliations on $n$-dimensional manifolds are modelled on the product structure $\R^q\times\R^{n-q}$ of $\R^n$ for some $q>0$.
We will call a diffeomorphism $f:\R^q\times\R^{n-q}\to \R^q\times\R^{n-q}$ \emph{admissible} if there are smooth functions $g: \R^q \to \R^q$ and $h:\R^q\times\R^{n-q}\to \R^{n-q}$ such that
\[f(x,y)=(g(x),h(x,y)) \text{ for all } (x,y)\in \R^q\times\R^{n-q}.\]

\begin{definition}\label{D:foliation atlas}
\em{ A codimension $q$ \emph{foliation atlas} on a manifold $M$ is defined by an atlas $\{U_i,\phi_i\}_{i\in I}$, where $\{U_i\}$ is an open cover of $M$ and
\[\phi_i:U_i\to \phi_i(U_i)\subset \R^{q}\times\R^{n-q}\]
are homeomorphisms such that the transition maps
\[\phi_j\phi_i^{-1}:\phi_i(U_i\cap U_j)\to \phi_j(U_i\cap U_j)\] are admissible maps. A codimension $q$ \emph{foliation} \index{foliation} on a manifold is a maximal foliation atlas on it.

For any foliation chart $(U_i,\phi_i)$, the sets $\phi_i^{-1}(x\times\R^{n-q})$ are called \emph{plaques}. Since the transition maps are admissible, the plaques through a point $p\in U_i\cap U_j$ defined by $\phi_i$ and $\phi_j$ coincide on the open set $U_i\cap U_j$. We define an equivalence relation on $M$ as follows: Two points $p$ and $q$ in $M$ are equivalent if there is a sequence of points $p=p_0,p_1,\dots,p_k=q$ such that any two consecutive points $p_i$ and $p_{i+1}$ lie on a plaque. The equivalence classes of this relation are called \emph{leaves} of the foliation. These are injectively immersed submanifolds of $M$.}\end{definition}

\subsection{Foliations as involutive distribution}
The tangent spaces of the plaques (or leaves) of a foliation $\mathcal F$ define a subbundle $T\mathcal F$ of $TM$, called the \emph{tangent bundle of $\mathcal F$},\index{$T\mathcal F$} which is clearly an involutive distribution. A subbundle $D$ of $TM$ is said to be \emph{involutive} if the space of sections of $D$ is closed under the Lie bracket of vector fields, that is, if $X,Y\in \Gamma(D)$ then so is $[X,Y]=XY-YX$. Conversely, if $D$ is an involutive distribution on a manifold $M$, then
Frobenius Theorem (\cite{warner}) says that $D$ is integrable; that is, through any point $x\in M$ there exists a maximal integral submanifold of $D$. The integral submanifolds of $D$ are the leaves of some foliation $\mathcal F$ on $M$.

\subsection{Foliations as Haefliger Cocycle}
A foliation $\mathcal F$ on a manifold can also be defined by the following data:
\begin{enumerate}\item An open covering $\{U_i, i\in I\}$ of $M$
\item submersions $s_i:U_i\to \R^q$ for each $i\in I$
\item local diffeomorphisms $h_{ij}:s_i(U_i\cap U_j)\to s_j(U_i\cap U_j)$ for all $i,j\in I$ for which $U_i\cap U_j\neq\emptyset$
\end{enumerate}
satisfying the commutativity relations
\[h_{ij}s_i=s_j \text{ on } U_i\cap U_j \text{ for all }(i,j) \]
and the cocycle conditions
\[h_{jk}h_{ij}=h_{ik} \text{ on }s_i(U_i\cap U_j\cap U_k).\]
The diffeomorphisms $\{h_{ij}\}$ are referred as \emph{Haefliger cocycles}\index{Haefliger cocycles}.

Since $s_i$'s are submersions, $s_i^{-1}(x)$ are submanifolds of $U_i$ of codimension $q$. Furthermore, since
\[s_i^{-1}(x)=s_j^{-1}(h_{ij}(x)) \ \text{ for all }\ x\in s_i(U_i\cap U_j),\]
the sets $s_i^{-1}(x)$ patch up to define a decomposition of $M$ into immersed submanifolds of codimension $q$. These submanifolds are the leaves of a foliation $\mathcal F$ on $M$. The tangent distribution $T\mathcal F$ is given by the local data $\ker ds_i$, $i\in I$.

On the other hand, if the foliation data is given by $\{U_i,\phi_i\}$ as in Definition~\ref{D:foliation atlas} then $s_i:U_i\to \R^q$ defined by $s_i=p_1\circ \phi_i$ are submersions, where $p_1:\R^q\times\R^{n-q}\to \R^q$ is the projection onto the first factor. Since $\phi_j\phi_i^{-1}$ is an admissible map, $h_{ij}:s_i(U_i\cap U_j)\to s_j(U_i\cap U_j)$ given by $h_{ij}(s_i(x))=s_j(x)$ is well-defined on $s_i(U_i\cap U_j)$. Furthermore, $\{h_{ij}\}$ satisfy the cocycle conditions.

\begin{definition}{\em Let $\mathcal F$ be a foliation on a manifold $M$. The quotient bundle $TM/T\mathcal F$ is defined as the \emph{normal bundle of the foliation $\mathcal F$} and is denoted by $\nu\mathcal F$.\index{$\nu(\mathcal F)$}}\end{definition}
If a foliation is given by the Haefliger data $\{U_i,s_i,h_{ij}\}$ then note that $(ds_i)_x:T_xM\to \mathbb{R}^q$ are surjective linear maps and $\ker(ds_i)_x=T_x\mathcal{F}$ for all $x\in U_i$. Therefore, $s_i$ induces an isomorphism $\tilde{s}_i:\nu(\mathcal F)|_{U_i}\to U_i\times\mathbb{R}^q$ given by
\[\tilde{s}_i(v+T_x\mathcal F)=(ds_i)_x(v)\ \text{ for all }v\in T_xM.\]
Noting that $(ds_j)_x\circ (ds_i)_x^{-1}$ is well defined for all $x\in U_i\cap U_j$, the transition maps of the normal bundle of $\mathcal F$ are given as follows:
\[\tilde{s}_j(x)\tilde{s}_i(x)^{-1}=ds_j\circ (ds_i)_x^{-1}=(dh_{ij})_{s_i(x)},\] where the second equality follows from the relation $h_{ij}s_i=s_j$.

\begin{definition} A smooth map $f:(M,\mathcal F)\to (M',\mathcal F')$ between foliated manifolds is said to be a \emph{foliation preserving map} if the derivative map of $f$ take $T\mathcal F$ into $T\mathcal F'$.
\end{definition}

\subsection{Maps transversal to a foliation}
The simplest type of foliations on manifolds are defined by submersions. Indeed, if $f:M\to N$ is a submersion then the fibres $f^{-1}(x)$ define (the leaves of) a foliation on the manifold $M$. In this case the leaves turn out to be embedded submanifolds of $M$. Now let $N$ itself be equipped with a foliation $\mathcal{F}_N$ of codimension $q$. In general, the inverse images of the leaves of a foliation on $N$ under a smooth map $f:M\to N$ need not give a foliation on $M$. We require some additional condition on the maps and this brings us to the notion of maps transversal to a foliation.

Let $N$ be a manifold with a foliation $\mathcal F_N$ and let $q:TN\to \nu(\mathcal F_N)$ denote the quotient map.
A smooth map $f:M\to N$ is said to be \emph{transversal to the foliation} $\mathcal F_N$ if  $q\circ df:TM\to \nu(\mathcal F_N)$ is an epimorphism; in other words,
\[df_x(T_xM) +(T\mathcal F_N)_{f(x)}=T_{f(x)}N \mbox{ \ for all \ }x\in M\]
If $\mathcal{F}_N$ is represented by the Haefliger data $\{U_i,s_i,h_{ij}\}$, then $\{f^{-1}(U_i),s_i\circ f, h_{ij}\}$ gives a Haefliger structure on $M$. The associated foliation is referred as the \emph{inverse image foliation of $\mathcal F_N$ under} $f$ and is denoted by $f^*\mathcal F_N$. The leaves of $f^*\mathcal F_N$ are the preimages of the leaves of $\mathcal F_N$ under $f$. Hence codimension of $f^*\mathcal F_N$ is the same as that of $\mathcal F_N$.

\subsection{$\Gamma_q$ structures\label{classifying space}} In this section we review some basic facts about $\Gamma$-structures for a topological groupoid $\Gamma$ following \cite{haefliger}. We also recall the connection between foliations on manifolds and $\Gamma_q$ structures, where $\Gamma_q$ is the groupoid of germs of local diffeomorphisms of $\R^q$\index{$\Gamma_q$}). For preliminaries of topological groupoid we refer to \cite{moerdijk}.
\begin{definition}\label{GS}{\em
Let $X$ be a topological space with an open covering $\mathcal{U}=\{U_i\}_{i\in I}$ and let $\Gamma$ be a topological groupoid over a space $B$. A 1-cocycle on $X$ over $\mathcal U$ with values in $\Gamma$ is a collection of continuous maps \[\gamma_{ij}:U_i \cap U_j\rightarrow \Gamma\] such that \[\gamma_{ik}(x)=\gamma_{ij}(x)\gamma_{jk}(x),\ \text{ for all }\ x\in U_i \cap U_j \cap U_k. \]
The above conditions imply that $\gamma_{ii}$ has its image in the space of units of $\Gamma$ which can be identified with $B$ via the unit map $1:B\to \Gamma$. We call two 1-cocycles $(\{U_i\}_{i\in \I},\gamma_{ij})$ and $(\{\tilde{U}_k\}_{k\in K},\tilde{\gamma}_{kl})$ equivalent if for each $i\in I$ and $k\in K$, there are continuous maps \[\delta_{ik}:U_i \cap \tilde{U}_k\rightarrow \Gamma\] such that
\[\delta_{ik}(x)\tilde{\gamma}_{kl}(x)=\delta_{il}(x)\ \text{for}\ x\in U_i \cap \tilde{U}_k \cap \tilde{U}_l\]
\[\gamma_{ji}(x)\delta_{ik}(x)=\delta_{ij}(x)\ \text{for}\ x\in U_i \cap U_j \cap \tilde{U}_k.\]
An equivalence class of a 1-cocycle is called a $\Gamma$-\emph{structure}\index{$\Gamma$-structure}. These structures have also been referred as Haefliger structures in the later literature.}
\end{definition}
For a continuous map $f:Y\rightarrow X$ and a $\Gamma$-structure $\Sigma=(\{U_i\}_{i\in I},\gamma_{ij})$ on $X$, the \emph{pullback $\Gamma$-structure} $f^*\Sigma$ is defined by the covering $\{f^{-1}U_i\}_{i \in I}$ together with the cocycles $\gamma_{ij}\circ f$.

If $f,g:Y\to X$ are homotopic maps and $\Sigma$ is a $\Gamma$-structure on $X$ then the pull-back structures $f^*\Sigma$ and $g^*\Sigma$ are not the same. They are homotopic in the following sense.
\begin{definition}{\em
Two $\Gamma$-structures $\Sigma_0$ and $\Sigma_1$ on a topological space $X$ are called \emph{homotopic} if there exists a $\Gamma$-structure $\Sigma$ on $X\times I$, such that $i_0^*\Sigma=\Sigma_0$ and $i_1^*\Sigma=\Sigma_1$, where $i_0:X\to X\times I$ and $i_1:X\to  X\times I$ are canonical injections defined by $i_t(x)=(x,t)$ for $t=0,1$.}\end{definition}

\begin{definition}{\em
Let $\Gamma$ be a topological groupoid with space of units $B$, source map $\mathbf{s}$ and target map $\mathbf{t}$. Consider the infinite sequences \[(t_0,x_0,t_1,x_1,...)\] with $t_i \in [0,1],\ x_i \in \Gamma$ such that all but finitely many $t_i$'s are zero and $\mathbf{t}(x_i)=\mathbf{t}(x_j)$ for all $i,j$. Two such sequences \[(t_0,x_0,t_1,x_1,...)\] and \[(t'_0,x'_0,t'_1,x'_1,...)\] are called equivalent if $t_i=t'_i$ for all $i$ and $x_i=x'_i$ for all $i$ with  $t_i\neq 0$. Denote the set of all equivalence classes by $E\Gamma$. The topology on $E\Gamma$ is defined to be the weakest topology such that the following set maps are continuous:
\[t_i:E\Gamma \rightarrow [0,1]\ \text{ given by }\ (t_0,x_0,t_1,x_1,...)\mapsto t_i \]
\[x_i: t_i^{-1}(0,1] \rightarrow \Gamma \ \text{ given by }\ (t_0,x_0,t_1,x_1,...)\mapsto x_i.\]
There is also a `$\Gamma$-action' on $E\Gamma$ as follows: Two elements $(t_0,x_0,t_1,x_1,...)$ and $(t'_0,x'_0,t'_1,x'_1,...)$ in $E\Gamma$ are said to be $\Gamma$-equivalent if $t_i=t'_i$ for all $i$,  and  if there exists a $\gamma\in \Gamma$ such that $x_i=\gamma x'_i$ for all $i$ with $t_i\neq 0$. The set of equivalence classes with quotient topology is called the \emph{classifying space of} $\Gamma$, and is denoted by $B\Gamma$\index{$B\Gamma$}.}
\end{definition}

Let $p: E\Gamma \rightarrow B\Gamma$ denote the quotient map. The maps $t_i:E\Gamma \rightarrow [0,1]$ project down to maps $u_i:B\Gamma \rightarrow [0,1]$ such that $u_i \circ p=t_i$. The classifying space $B\Gamma$ has a natural $\Gamma$-structure $\Omega=(\{V_i\}_{i\in I},\gamma_{ij})$, where $V_i=u_i^{-1}(0,1]$ and $\gamma_{ij}:V_i \cap V_j \rightarrow \Gamma$ is given by
\[(t_0,x_0,t_1,x_1,...)\mapsto x_i x_j^{-1}\]
We shall refer to this $\Gamma$ structure as the \emph{universal $\Gamma$-structure}\index{universal $\Gamma$-structure}.

For any two topological groupoids $\Gamma_1,\Gamma_2$ and for a groupoid homomorphism $f:\Gamma_1\rightarrow \Gamma_2$ there exists a continuous map \[Bf:B\Gamma_1\rightarrow B\Gamma_2,\]
defined by the functorial construction.
\begin{definition}{\em (Numerable $\Gamma$-structure)
Let $X$ be a topological space. An open covering $\mathcal{U}=\{U_i\}_{i\in I}$ of $X$ is called \emph{numerable} if it admits a locally finite partition of unity $\{u_i\}_{i\in I}$, such that $u_i^{-1}(0,1]\subset U_i$. If a $\Gamma$-structure can be represented by a 1-cocycle whose covering is numerable then the $\Gamma$-structure is called \emph{numerable}. }
\end{definition}
It can be shown that every $\Gamma$-structure on a paracompact space is numerable.
\begin{definition}{\em Let $X$ be a topological space. Two numerable $\Gamma$-structures are called \emph{numerably homotopic} if there exists a homotopy of numerable $\Gamma$-structures joining them.}
\end{definition}
Haefliger proved that the homotopy classes of numerable $\Gamma$-structures on a topological space $X$ are in one-to-one correspondence with the homotopy classes of continuous maps $X\to B\Gamma$.
\begin{theorem}(\cite{haefliger1})
\label{CMT} Let $\Gamma$ be a topological groupoid and $\Omega$ be the universal $\Gamma$ structure on $B\Gamma$. Then
\begin{enumerate}
\item $\Omega$ is numerable.
\item If $\Sigma$ is a numerable $\Gamma$-structure on a topological space $X$, then there exists a continuous map $f:X\rightarrow B\Gamma$ such that $f^*\Omega$ is homotopic to $\Sigma$.
\item If $f_0,f_1:X\rightarrow B\Gamma$ are two continuous functions, then $f_0^*\Omega$ is numerably homotopic to $f_1^*\Omega$ if and only if $f_0$ is homotopic to $f_1$.
\end{enumerate}
\end{theorem}
\subsection{$\Gamma_q$-structures and their normal bundles}
We now specialise to the groupoid $\Gamma_q$ of germs of local diffeomorphisms of $\mathbb{R}^{q}$. The source map $\mathbf s:\Gamma_q\to \R^q$ and the target map $\mathbf t:\Gamma_q\to \R^q$ are defined as follows: If $\phi\in\Gamma_q$ represents a germ at $x$, then
\[{\mathbf s}(\phi)=x\ \ \text{ and }\ \ {\mathbf t}(\phi)=\phi(x)\]
The units of $\Gamma_q$ consists of the germs of the identity map at points of $\R^q$.
$\Gamma_q$ is topologised as follows: For a local diffeomorphism $f:U\rightarrow f(U)$, where $U$ is an open set in $\mathbb{R}^q$, define $U(f)$ as the set of germs of $f$ at different points of $U$. The collection of all such $U(f)$ forms a basis of some topology on $\Gamma_q$ which makes it a topological groupoid. The derivative map gives a groupoid homomorphism \[\bar{d}:\Gamma_q \rightarrow GL_q(\mathbb{R})\]
which takes the germ of a local diffeomorphism $\phi$ of $\R^q$ at $x$ onto $d\phi_x$.
Thus, to each $\Gamma_q$-structure $\omega$ on a topological space $M$ there is an associated (isomorphism class of) $q$-dimensional vector bundle $\nu(\omega)$ over $M$ which is called the \emph{normal bundle of} $\omega$. In fact, if $\omega$ is defined by the cocycles $\gamma_{ij}$ then the cocycles $\bar{d}\circ \gamma_{ij}$ define the vector bundle  $\nu(\omega)$. Moreover, two equivalent cocycles in $\Gamma_q$ have their normal bundles isomorphic. Thus the normal bundle of a $\Gamma_q$ structure is the isomorphism class of the normal bundle of any representative cocycle. If two $\Gamma_q$ structures $\Sigma_0$ and $\Sigma_1$ are homotopic then there exists a $\Gamma_q$ structure $\Sigma$ on $X\times I$ such that $i_0^*\Sigma=\Sigma_0$ and $i_1^*\Sigma=\Sigma_1$, where $i_0:X\to X\times \{0\}\hookrightarrow X\times I$ and $i_1:X\to X\times \{1\}\hookrightarrow X\times I$ are canonical injective maps. Then $\nu(i_0^*\Sigma_0)\cong i_0^*\nu(\Sigma)\cong i_1^*\nu(\Sigma)\cong \nu(i_1^*\Sigma_1)$. Hence, 
normal bundles of homotopic $\Gamma_q$ structures are isomorphic.

In particular, we have a vector bundle $\nu\Omega_q$ on $B\Gamma_q$ associated with the universal $\Gamma_q$-structure $\Omega_q$ \index{$\Omega_q$} on $B\Gamma_q$.
\begin{proposition}If a continuous map $f:X\to B\Gamma_q$ classifies a $\Gamma_q$-structure $\omega$ on a topological space $X$, then $Bd\circ f$ classifies the vector bundle $\nu(\omega)$. In particular, $\nu\Omega_q\cong Bd^*E(GL_q(\R))$ and hence $\nu(\omega)\cong f^*\nu\Omega_q$.
\end{proposition}

\subsection{$\Gamma_q$-structures vs. foliations}

If a foliation  $\mathcal F$ on a manifold $M$ is represented by the Haefliger data $\{U_i,s_i,h_{ij}\}$, then we can define a $\Gamma_q$ structure on $M$ by $\{U_i,g_{ij}\}$, where \[g_{ij}(x) = \text{ the germ of } h_{ij} \text{ at } s_i(x) \text{ for }x\in U_i\cap U_j.\]
In particular, $g_{ii}(x)$ is the germ of the identity map of $\R^q$ at $s_i(x)$ and hence $g_{ii}$ takes values in the units of $\Gamma_q$. If we identify the units of $\Gamma_q$ with $\R^q$, then $g_{ii}$ may be identified with $s_i$ for all $i$. Thus, one arrives at a $\Gamma_q$-structure $\omega_{\mathcal F}$ represented by 1-cocycles $(U_i,g_{ij})$ such that \[g_{ii}:U_i\rightarrow \mathbb{R}^q\subset \Gamma_q\] are submersions for all $i$. The functions $\tau_{ij}:U_i\cap U_j\to GL_q(\R)$ defined by $\tau_{ij}(x)=(\bar{d}\circ g_{ij})(x)$ for $x\in U_i\cap U_j$, define the normal bundle of $\omega_{\mathcal F}$. Furthermore, since $\tau_{ij}(x)=dh_{ij}(s_i(x))$, $\nu(\omega_{\mathcal F})$ is isomorphic to the quotient bundle $\nu(\mathcal F)$. Thus a foliation on a manifold $M$ defines a $\Gamma_q$-structure whose normal bundle is embedded in $TM$.

As we have noted above, foliations do not behave well under the pullback operation, unless the maps are transversal to foliations. However, in view of the relation between foliations and $\Gamma_q$ structures, it follows that the inverse image of a foliation by any map gives a $\Gamma_q$-structure. The following result due to Haefliger says that any $\Gamma_q$ structure is of this type.
\begin{theorem}(\cite{haefliger1})
\label{HL}
Let $\Sigma$ be a $\Gamma_{q}$-structure on a manifold $M$. Then there exists a manifold $N$, a closed embedding $s:M \hookrightarrow N$ and a $\Gamma_{q}$-foliation $\mathcal{F}_N$ on $N$ such that $s^*(\mathcal{F}_N)=\Sigma$ and $s$ is a cofibration.
\end{theorem}

\newpage
\section{Foliations with geometric structures\label{forms_foliations}}
\subsection{Foliated de Rham cohomology}
Let $\Omega^r(M)$ denote the space of differential $r$-forms on a manifold $M$. For any foliation $\mathcal F$ on a manifold $M$, let $I^r(\mathcal{F})$ denote the subspace of $\Omega^r(M)$ consisting of all $r$-forms which vanish on the $r$-tuple of vectors from $T\mathcal F$. In other words, $I^r(\mathcal F)$ consists of all forms whose pull-back to the leaves of $\mathcal F$ are zero. Define \[\Omega^r(M,\mathcal{F})=\frac{\Omega^r(M)}{I^r(\mathcal{F})}\]
and let $q:\Omega^r(M)\to \Omega^r(M,\mathcal{F})$ be the quotient map. Since the leaves are integral submanifolds of $M$, the exterior differential operator $d$ maps $I^r(\mathcal{F})$ into $I^{r+1}(\mathcal{F})$ for all $r>0$, and thus we obtain a coboundary operator $d_{\mathcal{F}}:\Omega^r(M,\mathcal{F})\to \Omega^{r+1}(M,\mathcal{F})$ \index{$d_{\mathcal{F}}$} defined by $d_{\mathcal F}(\omega+ I^r(\mathcal{F}))=d\omega + I^{r+1}(\mathcal{F})$ so that the following diagram commutes:
\[
 \xymatrix@=2pc@R=2pc{
 \Omega^r(M) \ar@{->}[r]^-{d}\ar@{->}[d]_-{q} &  \Omega^{r+1}(M)\ar@{->}[d]^-{q}\\
 \Omega^r(M,\mathcal{F})\ar@{->}[r]_-{d_{\mathcal{F}}} & \Omega^{r+1}(M,\mathcal{F})
 }
\]
The cohomology groups of the cochain complex $(\Omega^r(M,\mathcal F),d_{\mathcal F})$ are called \emph{foliated de-Rham cohomology} \index{foliated de-Rham cohomology} groups of $(M,\mathcal F)$ and are denoted by $H^r(M,\mathcal{F})$, $r\geq 0$.

\begin{definition} {\em Let $M$ be a manifold with a foliation $\mathcal F$. A differential form $\omega$ on $M$ will be called $\mathcal F$-leafwise closed (resp. leafwise exact or leafwise symplectic) if the pull-back of $\omega$ to the leaves of $\mathcal F$ are closed forms (resp. exact forms, symplectic forms).}
\end{definition}

Let $T^*\mathcal F$ denote the dual bundle of $T\mathcal F$. The space  $\Omega^r(M,\mathcal{F})$ can be identified with the space of sections of the exterior bundle $\wedge^r(T^*\mathcal F)$ by the correspondence $\omega+I^r(\mathcal{F})\mapsto \omega|_{\Lambda^r(T\mathcal F)}$, $\omega\in \Omega^r(M)$. The induced coboundary map $\Gamma(\wedge^r(T^*\mathcal F))\to \Gamma(\wedge^{r+1}(T^*\mathcal F))$ will also be denoted by the same symbol $d_{\mathcal F}$. The sections of $\wedge^r(T^*\mathcal F)$ will be referred as \emph{tangential $r$-forms} or \emph{foliated $r$-forms}, or simply, $r$-forms on $\mathcal F$.\index{foliated $r$-forms} on $(M,\mathcal F)$.

\begin{definition} {\em Let $\mathcal F$ be a foliation on a manifold $M$. A foliated $k$-form $\alpha$ is said to be a \emph{foliated closed} or $d_{\mathcal F}$-\emph{closed} if $d_\mathcal F\alpha=0$. It is \emph{foliated exact} or $d_{\mathcal F}$-\emph{exact} if there exists a foliated $(k-1)$ form $\tau$ on $(M,\mathcal F)$ such that $\alpha=d_{\mathcal F}\tau$.}
\end{definition}

\begin{definition}{\em Let $\mathcal F$ be an even-dimensional foliation on a manifold $M$. A smooth section $\omega$ of $\wedge^2(T^*\mathcal F)$ will be called a \emph{symplectic form} on $\mathcal F$ if the following conditions are  satisfied:
\begin{enumerate}
\item $\omega$ is non-degenerate (i.e., $\omega_x$ is non-degenerate on the tangent space $T_x\mathcal F$ for all $x\in M$,)  and
\item $\omega$ is $d_{\mathcal F}$-closed. \end{enumerate}
The pair $(\mathcal F,\omega)$ will be called a \emph{symplectic foliation} on $M$. \index{symplectic foliation}
}\end{definition}

\begin{definition}{\em Let $\mathcal F$ be an even-dimensional foliation on a manifold $M$. A smooth section $\omega$ of $\wedge^2(T^*\mathcal F)$ will be called a \emph{locally conformal symplectic form} on $\mathcal F$ if the following conditions are  satisfied:
\begin{enumerate}
\item $\omega$ is non-degenerate  and
\item there exists a $d_\mathcal F$-closed foliated 1-form $\theta$ satisfying the relation $d_{\mathcal F}\omega+\theta\wedge \omega=0$.\end{enumerate}
The foliated deRham cohomology class of $\theta$ will be referred as the (foliated) \emph{Lee class} of $\omega$. The pair $(\mathcal F,\omega)$ will be called a \emph{locally conformal symplectic foliation} on $M$.}
\end{definition}

\begin{definition}{\em Let $\mathcal F$ be a foliation of dimension $2k+1$ on a manifold $M$. A foliated 1-form $\alpha$ (that is, a section of $T^*\mathcal F$) is said to be a \emph{contact form} on $\mathcal F$ if $\alpha\wedge (d_{\mathcal F}\alpha)^k$ is nowhere vanishing. The pair $(\mathcal F,\alpha)$ will be referred as a \emph{contact foliation} \index{contact foliation} on $M$.

A pair $(\alpha,\beta)$ consisting of a foliated 1-form $\alpha$ and a foliated 2-form $\beta$ is said to be an \emph{almost contact structure} on $(M,\mathcal F)$ if $\alpha\wedge \beta^k$ is nowhere vanishing. The triple $(\mathcal F,\alpha,\beta)$ will be called an \emph{almost contact foliation} on $M$.
}\end{definition}

\subsection{Poisson and Jacobi manifolds} We shall now consider some higher geometric structures which are given by multi-vector fields in contrast with the ones described in the previous section, which were defined by differential forms. These geometric structures are intimately related with foliations for which the leaves are equipped with locally conformal symplectic or contact forms.

\begin{definition} {\em Let $M$ be a smooth manifold. A (smooth) section of the vector bundle $\wedge^p(TM)$ will be called a \emph{$p$-vector field}. The space of $p$-vector fields for all $p\geq 0$ will be referred as the space of \emph{multi-vector fields}\index{multi-vector fields}.}
\end{definition}

If $X,Y$ are two vector fields on $M$ written locally as $X=\sum_i a_i\frac{\partial}{\partial x_i}$ and $Y=\sum_i b_i\frac{\partial}{\partial x_i}$ then the formula for the Lie bracket of $X$ and $Y$ is given as follows:
\[[X,Y]=\sum_i a_i(\sum_j\frac{\partial b_j}{\partial x_i}\frac{\partial}{\partial x_j})-\sum_i b_i(\sum_j\frac{\partial a_j}{\partial x_i}\frac{\partial}{\partial x_j}).\]
If we use the notation $\zeta_i$ for $\frac{\partial}{\partial x_i}$ then the vector fields $X$ and $Y$ could be thought of as functions of $x_i$'s and $\zeta_i$'s which are linear with respect to $\zeta_i$'s. So the formula for the lie bracket turns out to be \[[X,Y]=\sum_i\frac{\partial X}{\partial \zeta_i}\frac{\partial Y}{\partial x_i}-\sum_i\frac{\partial Y}{\partial \zeta_i}\frac{\partial X}{\partial x_i}\] Now let $X=\Sigma_{i_1<\dots<i_p}X_{i_1,\dots,i_p}\zeta_{i_1}\wedge\dots\wedge\zeta_{i_p}$ and $Y=\Sigma_{i_1<\dots<i_q}Y_{i_1,\dots,i_q}\zeta_{i_1}\wedge\dots\wedge\zeta_{i_q}$ be $p$ and $q$ vector fields respectively. Define the bracket of $X$ and $Y$, in analogy with the formula for the Lie bracket of vector fields as
\begin{equation}
\label{schouten bracket}
 [X,Y]=\sum_i\frac{\partial X}{\partial \zeta_i}\frac{\partial Y}{\partial x_i}-(-1)^{(p-1)(q-1)}\sum_i\frac{\partial Y}{\partial \zeta_i}\frac{\partial X}{\partial x_i}
\index{Schouten-Nijenhuis bracket}\end{equation}

\begin{theorem}(\cite{vaisman})
\label{schouten-nijenhuis}
 The formula (\ref{schouten bracket}) satisfies the following
\begin{enumerate}
 \item Let $X$ and $Y$ be $p$ and $q$ vector fields respectively, then \[[X,Y]=-(-1)^{(p-1)(q-1)}[Y,X]\]
\item Let $X,Y$ and $Z$ be $p,q$ and $r$ vector fields respectively, then \[[X,Y\wedge Z]=[X,Y]\wedge Z+(-1)^{(p-1)q}Y\wedge [X,Z]\] \[[X\wedge Y,Z]=X\wedge [Y,Z]+(-1)^{(r-1)q}[X,Z]\wedge Y\]
\item \[(-1)^{(p-1)(r-1)}[X,[Y,Z]]+(-1)^{(q-1)(p-1)}[Y,[Z,X]]\]\[+(-1)^{(r-1)(q-1)}[Z,[X,Y]]=0\]
\item If $X$ is a vector field and $f$ is a real valued function on $M$ then
\[[X,Y]=\mathcal{L}_XY \ \ and \ \ [X,f]=X(f)\]
\end{enumerate}
\end{theorem}
\begin{definition}
 \em{The bracket in (\ref{schouten bracket}) is called the Schouten-Nijenhuis bracket.}
\end{definition}
The second assertion in Theorem~\ref{schouten-nijenhuis} implies that the definition of the Schouten-Nijenhuis bracket given by (\ref{schouten bracket}) is independent of the choice of local coordinates.
\begin{definition}{\em A bivector field $\pi$ on $M$ is called a \emph{Poisson bivector field} if it satisfies the relation $[\pi,\pi]=0$, where [\ ,\ ] is the Schouten-Nijenhuis bracket (\cite{vaisman}).}\index{Poisson bivector field}\end{definition}
A Poisson structure on a smooth manifold $M$ can also be defined by a $\mathbb{R}$-bilinear antisymmetric operation \[\{,\}:C^{\infty}(M,\mathbb{R})\times C^{\infty}(M,\mathbb{R})\to C^{\infty}(M,\mathbb{R})\]
which satisfies the Jacobi identity:
\[\{f,\{g,h\}\}+\{g,\{h,f\}\}+\{h,\{f,g\}\}=0 \ \text{for all }f,g,h\in C^\infty(M);\]
and the Leibnitz identity for each $f\in C^\infty(M)$:
\[\{f,gh\}=\{f,g\}h+g\{f,h\}\ \ \ \text{ for all } g,h\in C^\infty(M).\]
The relation between a \emph{Poisson bracket} $\{\ ,\ \}$ \index{Poisson bracket} and the associated Poisson bi-vector field is given as follows: For any two functions $f,g\in C^\infty(M)$
\[\{f,g\}=\pi(df,dg).\]

\begin{example}
{\em Let $M$ be a $2n$-dimensional manifold with a symplectic form $\omega$. The non-degeneracy condition implies that $b:TM\to T^*M$, given by $b(X)=i_X\omega$ is a vector bundle isomorphism, where $i_X$ denotes the interior multiplication by $X\in TM$. Then $M$ has a Poisson structure defined by
\[\pi(\alpha,\beta)=\omega(b^{-1}(\alpha),b^{-1}(\beta)), \ \text{ for all }\alpha,\beta\in T^*_xM, x\in M.\]}\label{ex_symplectic}\end{example}

In \cite{kirillov}, Kirillov further generalised the Poisson bracket. The underlying motivation was to understand the geometric properties of all manifolds $M$ which admit a local lie algebra structure on $C^{\infty}(M)$.
\begin{definition} {\em A \emph{local lie algebra} structure on $C^{\infty}(M)$ is an antisymmetric $\R$ bilinear map
\[\{,\}:C^{\infty}(M)\times C^{\infty}(M)\to C^{\infty}(M)\] such that
\begin{enumerate}\item the bracket satisfies the Jacobi identity namely, \[\{f,\{g,h\}\}+\{g,\{h,f\}\}+\{h,\{f,g\}\}=0 \ \text{for all }f,g,h\in C^\infty(M);\]
\item supp\,$(\{f,g\})\subset \text{supp\,}(f)\cap \text{supp\,}(g)$ for $f,g\in C^{\infty}(M)$ (that is, $\{\ ,\ \}$ is local).
 \end{enumerate} }\end{definition}
The bracket defined above is called a \emph{Jacobi bracket}\index{Jacobi bracket}.

\begin{definition} {\em A \emph{Jacobi structure} on a smooth manifold $M$ is given by a pair $(\Lambda,E)$, where $\Lambda$ is a bivector field and $E$ is a vector field on $M$, satisfying the following two conditions:
\begin{equation}[\Lambda,\Lambda]  =  2E\wedge \Lambda,\ \ \ \ \ \ [E,\Lambda] =  0.\label{jacobi_def}\index{Jacobi structure}\end{equation}
If $E=0$ then $\Lambda$ is a Poisson bivector field on $M$.}\end{definition}

The notion of a local Lie algebra structure on $C^\infty(M)$ is equivalent to that of a Jacobi structure on $M$ (\cite{kirillov}).
If $(\Lambda, E)$ is a Jacobi pair, then we can define the associated Jacobi bracket by the following relation:
\begin{equation}\{f,g\}=\Lambda(df,dg)+fE(g)-gE(f),\ \text{for}\ f,g\in C^{\infty}(M) \label{eqn:jacobi}\index{Jacobi bracket}\end{equation}
Taking $E=0$ we get the relation between the Poisson bracket and the Poisson bivector field.

\begin{example} {\em Every locally conformal symplectic manifold (in short, an l.c.s manifold)
is a Jacobi manifold, where the Jacobi pair is given by \[\Lambda(\alpha,\beta)=\omega(b^{-1}(\alpha),b^{-1}(\beta))\ \ \text{ and }\ \ \ E=b^{-1}(\theta),\]
where $b:TM\to T^*M$ is defined as in Example~\ref{ex_symplectic}.}\end{example}

\begin{example} {\em Every manifold with a contact form is a Jacobi manifold. If $\alpha$ is a contact form on $M$, then recall that there is an isomorphism $\phi:TM\to T^*M$ defined by $\phi(X)=i_X d\alpha+\alpha(X)\alpha$ for all $X\in TM$. A Jacobi pair on $(M,\alpha)$ can be defined as follows:
\[\Lambda(\beta,\beta')=d\alpha(\phi^{-1}(\beta),\phi^{-1}(\beta')),\ \ \ \mbox{and }\ \ \ \ E=\phi^{-1}(\alpha),\]
where $\beta,\beta'$ are 1-forms on $M$.
The bivector field $\Lambda$ defines a bundle homomorphism ${\Lambda}^\#:T^*M\to TM$ by
\[{\Lambda}^\#(\alpha)(\beta)=\Lambda(\alpha,\beta),\]
where $\alpha,\beta\in T^*_xM$, $x\in M$. The image of the vector bundle morphism $\Lambda^\#:T^*M\to TM$ is $\ker\alpha$ and $\ker\Lambda^\#$ is spanned by $R_\alpha$. The contact Hamiltonian vector field $X_H$ can then be expressed as $X_H=HR_\alpha+\Lambda^{\#}(dH)$. }\end{example}

Let $(M,\Lambda,E)$ be a Jacobi manifold. The Jacobi pair $(\Lambda, E)$ defines a distribution $\mathcal D$, called the \emph{characteristics distribution} \index{characteristics distribution} of the Jacobi pair, as follows:
\begin{equation}{\mathcal D}_x={\Lambda}^\#(T^*_xM)+\langle E_x\rangle, x\in M,\label{distribution_jacobi}\end{equation}
where $\langle E_x\rangle$ denotes the subspace of $T_xM$ spanned by the vector $E_x$.
\begin{remark}{\em In general, $\mathcal D$ is only a singular distribution; however, it is completely integrable in the sense of Sussman (\cite{vaisman}).}\end{remark}

\begin{definition}{\em A Jacobi pair $(\Lambda,E)$ is called \emph{regular} if $x\mapsto \dim\mathcal D_x$ is a locally constant function on $M$. It is said to be a \emph{non-degenerate} Jacobi structure if $\mathcal D$ equals $TM$.}\end{definition}
Every $C^\infty$ function $f$ on a Jacobi manifold $(M,\Lambda,E)$ defines a vector field $X_f$ by $X_f=\Lambda^{\#}(df)$ so that $\Lambda(df,dg)=X_f(g)$. Then we have the following relations (\cite{kirillov}):
\begin{equation}\begin{array}{rcl}
\ [E, X_f] & =  & X_{Ef}\\
\ [X_f,X_g]  & =  & X_{\{f,g\}}-fX_{Eg}+gX_{Ef}-\{f,g\}E
\label{jacobi_bracket_hamiltonian}\end{array}\end{equation}
where $[,]$ is the usual Lie bracket of vector fields. If $\mathcal D$ is regular then the characteristic distribution $\mathcal D$ is spanned by the vector fields $E$ and $X_f$, $f\in C^\infty(M)$. Thus, it follows easily from the relations in (\ref{jacobi_bracket_hamiltonian}) that $\mathcal D$ is involutive and therefore, integrable.
\begin{lemma} A Jacobi structure $(\Lambda, E)$ restricts to a non-degenerate Jacobi structure on the leaves of its characteristic distribution.\label{L:jacobi_leaf}\end{lemma}
\begin{proof} Let $f,g$ be two smooth functions on a leaf $L$ of the characteristic distribution $\mathcal D$. The induced Jacobi bracket on a leaf $L$ is given as follows:
\[\{f,g\}(x)=\{\tilde{f},\tilde{g}\}(x) \text{ for all }x\in L\]
where $\tilde{f}$ and $\tilde{g}$ are arbitrary extensions of $f$ and $g$ respectively on some open neighbourhood of $L$. Since $E(x)\in \mathcal D_x$, $E\tilde{f}(x)$ depends only on the values of $f$ on the leaf $L$ through $x$. Also, $X_{\tilde{f}}\tilde{g}(x)=d\tilde{g}_x(X_{\tilde{f}}(x))=dg_x(X_{\tilde{f}}(x))$ since $X_{\tilde{f}}(x)\in \mathcal D_x=T_xL$. This shows that the value of $X_{\tilde{f}}\tilde{g}(x)$ is independent of the extension of $\tilde{g}$. Similarly, it is also independent of the choice of the extension $\tilde{f}$. Thus $\{f,g\}$ is well-defined by (~\ref{eqn:jacobi}). It follows through a routine calculation that the above defines a Jacobi bracket. The non-degeneracy of the bracket on $L$ is immediate from the definition of $\{f,g\}$.
\end{proof}
\begin{theorem}(\cite{kirillov}) Every non-degenerate Jacobi manifold is either locally conformal symplectic or a contact manifold.\label{T:jacobi_leaf1}\end{theorem}
\begin{proof}
First suppose that $M$ is of even dimension. Since $(\Lambda,E)$ is non-degenerate and the rank of $\Lambda^\#$ is even, $\Lambda^{\#}:T^*M\to TM$ must be an isomorphism. Define a 2-form $\omega$ and a 1-form $\theta$ on $M$ as follows:
\begin{center}$\begin{array}{rcll}\omega(\Lambda^\#(\alpha),\Lambda^\#(\beta)) & = & \Lambda(\alpha,\beta) & \text{for all } \alpha,\beta\in T^*_xM, x\in M,\\
\Lambda^{\#}\circ\theta & = & E. & \end{array}$\end{center}
We shall show that $\omega$ is a locally conformal symplectic form with Lee form $\theta$; in other words,
we need to show that $\theta$ is closed and $d\omega+\theta\wedge\omega=0$. Since the vector fields $X_f=\Lambda^\#(df)$, $f\in C^\infty(M)$, generate $TM$, it is enough to verify any relation on $X_f$'s only.

In the following we shall use the notation $\Sigma_{\circlearrowright}$ for cyclic sum over $f,g,h$.
First note that \[(\theta \wedge \omega)(X_f,X_g,X_h)=-\Sigma_{\circlearrowright}Ef.X_g(h)\] since $\theta(X_f)=-Ef$. Next, we have
\[
\begin{array}{rcl}
 d\omega(X_f,X_g,X_h)&=&\Sigma_{\circlearrowright}X_f \omega(X_g,X_h)-\Sigma_{\circlearrowright}\omega([X_f,X_g],X_h)\\
 &=&\Sigma_{\circlearrowright}X_fX_g(h)-\Sigma_{\circlearrowright}[X_f,X_g]h\\
&=&\Sigma_{\circlearrowright}X_fX_g(h)-\Sigma_{\circlearrowright}X_fX_g(h)-\Sigma_{\circlearrowright}X_gX_h(f)\\
&=&-\Sigma_{\circlearrowright}X_gX_h(f)\\
&=&-\Sigma_{\circlearrowright}X_g[\{h,f\}-hEf+fEh]\\
&=&-\Sigma_{\circlearrowright}X_g\{h,f\}+\Sigma_{\circlearrowright}X_g(hEf)-\Sigma_{\circlearrowright}X_g(fEh)\\
&=&-\Sigma_{\circlearrowright}[\{g,\{h,f\}-gE\{h,f\}+\{h,f\}Eg]+\Sigma_{\circlearrowright}hX_g(Ef)\\
& & +\Sigma_{\circlearrowright}EfX_g(h)-\Sigma_{\circlearrowright}fX_g(Eh)-\Sigma_{\circlearrowright}EhX_g(f)\\
&=&\Sigma_{\circlearrowright}gE\{h,f\}-\Sigma_{\circlearrowright}\{h,f\}Eg+\Sigma_{\circlearrowright}hX_g(Ef)+\Sigma_{\circlearrowright}EfX_g(h)\\
& &-\Sigma_{\circlearrowright}fX_g(Eh)-\Sigma_{\circlearrowright}EhX_g(f)\\
\end{array}
\]
The second summand in the last expression will cancel the fourth summand, as it will follow from the identity below:
\[
\begin{array}{rcl}
 -\Sigma_{\circlearrowright}\{h,f\}Eg &=&-\Sigma_{\circlearrowright}[hEf-fEh+X_h(f)]Eg\\
&=&-\Sigma_{\circlearrowright}hEfEg+\Sigma_{\circlearrowright}fEhEg-\Sigma_{\circlearrowright}X_hf.Eg\\
&=& -\Sigma_{\circlearrowright}X_h(f).Eg\\
\end{array}
\]
Furthermore, the first summand can be written as
\[
 \begin{array}{rcl}
  \Sigma_{\circlearrowright}gE\{h,f\}&=& \Sigma_{\circlearrowright}gE[hEf-fEh+X_hf]\\
&=&\Sigma_{\circlearrowright}g[E(hEf)-E(fEh)+E(X_hf)]\\
&=&\Sigma_{\circlearrowright}g[hEEf+EfEh-fEEh-EhEf+E(X_hf)]\\
&=&\Sigma_{\circlearrowright}gEX_hf\\
 \end{array}
\]
Thus we get
\[
 \begin{array}{rcl}
  d\omega(X_f,X_g,X_h)&=&\Sigma_{\circlearrowright}gEX_hf+\Sigma_{\circlearrowright}hX_g(Ef)-\Sigma_{\circlearrowright}fX_g(Eh)-\Sigma_{\circlearrowright}Eh.X_gf\\
&=& -\Sigma_{\circlearrowright}g[X_h,E]f+\Sigma_{\circlearrowright}hX_g(Ef)-\Sigma_{\circlearrowright}Eh.X_gf\\
&=& \Sigma_{\circlearrowright}gX_{Eh}f+\Sigma_{\circlearrowright}hX_g(Ef)-\Sigma_{\circlearrowright}Eh.X_gf\\
&=&-\Sigma_{\circlearrowright}Eh.X_gf\\
&=&\Sigma_{\circlearrowright}Eh.X_fg\\
&=& -\theta \wedge \omega(X_f,X_g,X_h)\\
 \end{array}
\]
To show that $\theta$ is closed, we observe that
\[
 \begin{array}{rcl}
  d\theta(X_f,X_g)&=&X_f\theta(X_g)-X_g\theta(X_f)-\theta([X_f,X_g])\\
&=& -X_fEg+X_gEf-\theta(X_{\{f,g\}}-fX_{Eg}+gX_{Ef}-\{f,g\}E)\\
&=& -X_fEg+X_gEf+E(\{f,g\})-fEEg+gEEf\\
 \end{array}
\]
and
\[
 \begin{array}{rcl}
  E(\{f,g\})&=&E(fEg-gEf+X_fg)\\
&=&fEEg+EgEf-gEEf-EgEf+E(X_fg)\\
&=&fEEg-gEEf+E(X_fg)\\
 \end{array}
\]
Combining the above relations we get
\[
 \begin{array}{rcl}
  d\theta(X_f,X_g) &=& -X_fEg+X_gEf+fEEg-gEEf+E(X_fg)-fEEg+gEEf\\
&=& -[X_f,E]g+X_gEf\\
&=& X_{Ef}g+X_g(Ef)\\
&=&0\\
 \end{array}
\]
Thus, we have proved that $M$ is a locally conformal symplectic manifold when $M$ is even dimensional. If $\dim M$ is odd then $E\notin \text{Im}(\Lambda^{\#})$. In this case, we can define a 1-form $\alpha$ by \[\alpha(E)=1,\ \alpha(X_f)=0,\text{ for all }f\in C^\infty(M).\] Then,
\[
 \begin{array}{rcl}
d\alpha(X_f,X_g)&=&X_f\alpha(X_g)-X_g\alpha(X_f)-\alpha([X_f,X_g])\\
&=&-\alpha([X_f,X_g])\\
&=&-\alpha(X_{\{f,g\}}-fX_{Eg}+gX_{Ef}-\{f,g\}E)\\
&=&\{f,g\}\\
&=&fEg-gEf+X_fg\\
 \end{array}
\]
To show that $d\alpha$ is non-degenerate on Im\,$\Lambda^\#$, suppose that
$d\alpha(X_f,X_g)=0  \text{ for all }g\in C^\infty(M)$; that is,
\[fEg-gEf+X_fg=0 \text{ for all }g\in C^\infty(M).\]
In particular, if we take $g=1$ in the above we get $Ef=0$. Hence,
\[X_fg+fEg=(X_f+fE)g=0 \text{ for all } g\in C^\infty(M),\]
which can only happen if $f=0$, as $X_f$ and $E$ are linearly independent. Thus, $X_f=0$ proving that $d\alpha|_{\text{Im\,}\Lambda^{\#}}$ is nondegenerate. Finally we observe that
\[
 \begin{array}{rcl}
  d\alpha(E,X_f)&=&E\alpha(X_f)-X_f\alpha(E)-\alpha([E,X_f])\\
&=&0
 \end{array}
\]
This proves that $\alpha$ is a contact form with Reeb vector field $E$.
\end{proof}
Combining Lemma~\ref{L:jacobi_leaf} and Theorem~\ref{T:jacobi_leaf1} we obtain the following theorem.
\begin{theorem} The characteristic foliation of a regular Jacobi structure is either a locally conformal symplectic foliation or a contact foliation. Conversely, a locally conformal symplectic foliation or a contact foliation defines a regular Jacobi structure.\label{T:jacobi_foliation}
\end{theorem}
Results of this section will be used in Chapter 3.
\newpage
\section{Preliminaries of $h$-principle}
In this section we recall some preliminaries of $h$-principle following \cite{eliashberg}, \cite{geiges} and \cite{gromov_pdr}.
The theory of $h$-principle addresses questions related to partial differential equations or more general relations which appear in topology and geometry. As Gromov mentions in the foreword of his book `Partial Differential Relations'(\cite{gromov_pdr}), these equations or relations are mostly underdetermined, in contrast with those which arise in Physics. As a result, there are plenty of solutions to these equations/relations and one can hope to classify the solution space using homotopy theory.  The $r$-jet bundle associated with sections of a fibration $X\to M$ has the structure of an affine bundle over $X$. An $r$-th order partial differential relation for smooth sections of $X$ determines a subset $\mathcal R$ in the $r$-jet space $X^{(r)}$. The theory of $h$-principle studies to what extent the topological and geometric properties of this set $\mathcal R$ govern the solution space.

\subsection{Jet bundles}(\cite{golubitsky})
An ordered tuple $\alpha=(\alpha_1,\alpha_2,\dots,\alpha_m)$ of non-negative integers will be called a multi-index. For any $x=(x_1,x_2,\dots,x_n)\in\R^n$ and any multi-index $\alpha$, the notation $x^\alpha$ will represent the monomial $x_1^{\alpha_1}\dots x_m^{\alpha_m}$ and $\partial^\alpha$ will stand for the operator \[\frac{\partial^{|\alpha|}}{{\partial x_1}^{\alpha_1}{\partial x_2}^{\alpha_2}\dots{\partial x_m}^{\alpha_m}},\] where $|\alpha|=\alpha_1+\alpha_2+\dots+\alpha_m$.
 Two smooth maps $f,g:\R^m\to \R^n$ are said to be \emph{$k$-equivalent} at $x\in M$ if $f(x)=g(x)=y$ and $\partial^{\alpha}f(x)=\partial^{\alpha}g(x)$, for every multi-index $\alpha$ with $|\alpha|\leq k$. The equivalence class of $(f,x)$ is called the $k$-jet of $f$ at $x$ and is denoted by $j^k_f(x)$. Thus a $k$-jet $j^k_f(x)$ can be represented by a polynomial $\sum_{|\alpha|\leq k}\partial^\alpha f(x) x^\alpha$.
Let $B^k_{n,m}$ be the vector space of polynomials of degree at most $k$ in $m$ variables and values in $\mathbb{R}^n$. Then the space of $k$-jets of maps $\R^m\to \R^n$, denoted by $J^k(\R^m,\R^n)$, can be identified with the set $\R^m\times\R^n\times B^k_{m,n}$.

\begin{definition}
\label{D:Jet bundles}
\em{Let $M,N$ be $C^{\infty}$-manifolds. Two $C^\infty$ maps $f,g:M\to N$ are said to be $k$-equivalent at $x\in M$ if $f(x)=g(x)=y$ and with respect to some local coordinates around $x$ and $y$, $\partial^{\alpha}f(x)=\partial^{\alpha}g(x)$, for every multi-index $\alpha$ with $|\alpha|\leq k$. Using the chain rule one can see that the partial derivative condition does not depend on the choice of the local coordinate system around $x$ and $y$. As before, the equivalence class of an $f$ defined on an open neighbourhood of $x$ will be called the $k$-jet of $f$ at $x$ and will be denoted by $j^k_f(x)$. The set of $k$-jets of germs of all functions from $M$ to $N$ will be denoted by $J^k(M,N)$ and will be called the \emph{$k$-jet bundle} associated with the function space $C^\infty(M,N)$.}\end{definition}
\begin{remark}{\em In particular, $J^0(M,N)=M\times N$. The 1-jet bundle $J^1(M,N)$ can be identified with \emph{Hom}$(TM,TN)$ consisting of all linear maps $T_xM\to T_yN$, $x\in M$ and $y\in N$, under the correspondence
\[j^1_f(x)\mapsto (x,f(x),df_x),\]
where $f:U\to N$ is a smooth map defined on an open set $U$ containing $x$.}\label{1-jet}\end{remark}
If $(U,\phi)$ and $(V,\psi)$ are two charts of $M$ and $N$ respectively, then there is an obvious bijection
\[T_{U,V}:J^k(U,V)\to \R^m\times \R^n\times B^k_{n,m}.\]
The jet bundle $J^k(M,N)$ is topologised by declaring the sets $J^k(U,V)$ open. A manifold structure is given by declaring $T_{U,V}$ as charts.

We can generalise the notion of jet bundle to sections of a smooth fibration $p:X\to M$ as well.
\begin{definition}{\em Let $X^{(k)}_x$ denote the set of all $k$-jets of germs of smooth sections of $p$ defined on an open neighbourhood of $x\in M$. The $k$-th \emph{jet bundle of sections} of $X$ is defined as follows:
\[X^{(k)}=\cup_{x\in M}X^{(k)}_x. \index{$X^{(k)}$}\]\index{jet bundle}
Clearly, $X^{(0)}=X$. }
\end{definition}
\begin{remark}{\em If $X$ is a trivial fibration over a manifold $M$ with fibre $N$, then the sections of $X$ are in one-to-one correspondence with the maps from $M$ to $N$. Therefore, we can identify the jet space $X^{(k)}$ with $J^k(M,N)$.}\end{remark}

If $f$ and $g$ are two local sections of a fibration $p:X\to M$ which represent the same $k$-jet at a point $x\in M$, then they also represent the same $l$-jet at  $x$ for any $l\leq k$. Therefore, we have natural projection maps:
\[p_l^k:X^{(k)}\to X^{(l)}\ \ \ \text{for }l\leq k.\]
Set $p^{(k)}=p\circ p^k_0:X^{(k)}\to M$. If $g$ is a section of $p$ then $x\mapsto j^k_g(x)$ defines a section of $p^{(k)}$. We shall denote this section by $j^k_g$ or $j^kg$.

\begin{theorem}(\cite{golubitsky})
Let $p:X\to M$ be a smooth fibration over a manifold $M$ of dimension $m$. Suppose that the dimension of the fibre is $n$. Then
\begin{enumerate}\item $X^{(k)}$ is a smooth manifold of dimension $m+n+dim(B^k_{n,m})$;
\item $p^{(k)}:X^{(k)}\to M$ is a fibration;
\item for any smooth section $g:M\to X$, $j^kg:M\to X^{(k)}$ is smooth.\end{enumerate}
\end{theorem}

\subsection{Weak and fine topologies}

Let $p:X\to M$ be a smooth fibration.  We shall denote the space of $C^k$-sections of $X$ by $\Gamma^k(X)$ for $0\leq k\leq \infty$\index{$\Gamma^k(X)$}.
\begin{definition}{\em (\cite{hirsch}) The \emph{weak $C^0$ topology} on $\Gamma^0(X)$ is the usual compact open topology. If $k$ is finite, then the \emph{weak $C^k$-topology} (or the $C^k$ \emph{compact open topology}) on $\Gamma^\infty(X)$ is the topology induced by the $k$-jet map $j^k:\Gamma^\infty(X)\to \Gamma^0(X^{(k)})$, where $\Gamma^0(X^{(k)})$ has the $C^0$ compact open topology. The \emph{weak $C^\infty$ topology} (or the $C^\infty$ compact open topology) is the union of the weak $C^k$ topologies for $k$ finite.}\label{weak topology}\index{weak $C^k$ topology}\end{definition}

We shall now describe the fine topologies on the function spaces.
For any set $C\subset X^{(k)}$ define a subset $B(C)$ of $\Gamma^{\infty}(X)$ as follows:
\[B(C)=\{f\in \Gamma^{\infty}(X):j^k_f(M)\subset C\}. \] Then observe that $B(C)\cap B(D)= B(C\cap D)$.
\begin{definition}
\label{D:fine topology}
 \em{(\cite{golubitsky}) The collection $\{B(U): U \text{ open in }X^{(k)}\}$ forms a basis of some topology on $\Gamma^{\infty}(X)$, which we call the \emph{fine $C^k$-topology}. The \emph{fine $C^{\infty}$-topology} on $\Gamma^{\infty}(X)$ is the inverse limit of these $C^k$-topologies. The maps $p^k_{k-1}:X^{(k)}\to X^{(k-1)}$ define a spectrum with respect to the fine topologies.
}\index{fine topology}
\end{definition}
\begin{remark} {\em The fine $C^{k}$-topology on $\Gamma^{\infty}(X)$ is induced from the fine $C^0$-topology on $\Gamma^0(X^{(k)})$ by the $k$-jet map
\[j^k:\Gamma^{\infty}(X)\to \Gamma^0(X^{(k)}), \ \ f\mapsto j^kf.\]  }\end{remark}
The fine $C^k$ topology is, in general, finer than the weak $C^k$ topology. However, if $M$ is compact then these are equal.
For a better understanding of the fine $C^k$-topologies we describe a basis of the neighborhood system of an $f\in \Gamma^\infty(X)$. Let us first fix a metric on $X^{(k)}$. For any smooth section $f$ of $X$ and a positive smooth function $\delta:M\to \R_+$ define
\[\mathcal{N}^k_{\delta}(f)=\{g\in \Gamma^{\infty}(X): \text{dist}(j^k_f(x),j^k_g(x))<\delta(x) \text{ for all }x\in M\}.\]
The sets $\mathcal{N}^k_{\delta}(f)$ form a neighbourhood basis of $f$ in the fine $C^k$-topology.
\begin{remark} {\em If $\mathcal R$ is an open subset of $X^{(k)}$ then the space of sections of $X^{(k)}$ with images contained in $\mathcal R$ is an open subset of $\Gamma^0(X^{(k)})$ in the fine $C^0$-topology. Consequently, $(j^k)^{-1}(\Gamma(\mathcal R))$ is an open subspace of $\Gamma^\infty(X)$ in the fine $C^\infty$ topology.}
\end{remark}

\subsection{Holonomic Approximation Theorem}
\begin{definition}{\em A section of the jet-bundle $p^{(k)}:X^{(k)}\to M$ is said to be a \emph{holonomic} section if it is the $r$-jet map of some section $f:M\to X$.}\index{holonomic}\end{definition}

We now recall the Holonomic Approximation Theorem from \cite{eliashberg}. Throughout the thesis, for any subset $A\subset M$, $Op\,A$ \index{$Op\,A$} will denote an unspecified open neighbourhood of $A$ (which may change in course of an argument).
\begin{theorem} Let $A$ be a polyhedron (possibly non-compact) in $M$ of positive codimension. Let $\sigma$ be any section of the $k$-jet bundle $X^{(k)}$ over $Op\,A$.  Given any positive functions $\varepsilon$ and $\delta$ on $M$ there exist a diffeotopy $\delta_t:M\to M$ and a holonomic section $\sigma':Op\,\delta_1(A)\to X^{(k)}$ such that
\begin{enumerate}\item $\delta_t(A)\subset domain\,(\sigma)$ for all $t$;
\item $dist(x,\delta_t(x))<\delta(x)$ for all $x\in M$ and $t\in [0,1]$ and
\item $dist(\sigma(x),\sigma'(x))<\varepsilon(x)$ for all $x\in Op\,(\delta_1(A))$.
\end{enumerate}$($Any diffeotopy $\delta_t$ satisfying (2) will be referred as $\delta$-small diffeotopy.$)$
\label{T:holonomic_approximation}\index{Holonomic Approximation Theorem}\end{theorem}
We now mention the main steps in the proof of the theorem when $\dim M=2$ and $\dim A=1$. To start with $A$ is covered by small coordinate neighbourhoods $\{U_i\}$ of $M$. If the intersection $U_i\cap U_j$ is non-empty then we choose a small hypersurface $S_{i,j}$ in $U_i\cap U_j$ transversal to $A$.
The map $\sigma$ is then approximated by holonomic sections $\sigma_i$ on open sets $U_i$. On the intersection $U_i\cap U_j$, the two holonomic approximations do not match, in general. However, the set of points in $U_i\cap U_j$ where $\sigma_i$ is not equal to $\sigma_j$ can be made to lie in an arbitrary small neighbourhood of $S_{i,j}$. Let $S$ be the union of the transversals $S_{i,j}$. The main task is to modify the local holonomic sections $\sigma_i$ on $U_i$ to get a holonomic section $\sigma'$ defined on the subset $U\setminus S$, where $U$ is an open neighbourhood of $A$. It can also be ensured that $\sigma'$ lies sufficiently $C^0$-close to $\sigma$.

The next step is to get a small isotopy which would move $A$ outside the set $S$ where $\sigma'$ is already defined.
Indeed, if the transversals $S_{i,j}$ are small enough then there exist diffeotopies $\delta_t, t\in\I$ which have the following properties:
\begin{enumerate}\item $\delta_0$ is the identity map of $U$;
\item $\delta_t$ is identity outside a small neighbourhood of $S$;
\item $\delta_1$ maps $Op\,A$ into $U\setminus S$.
\end{enumerate}

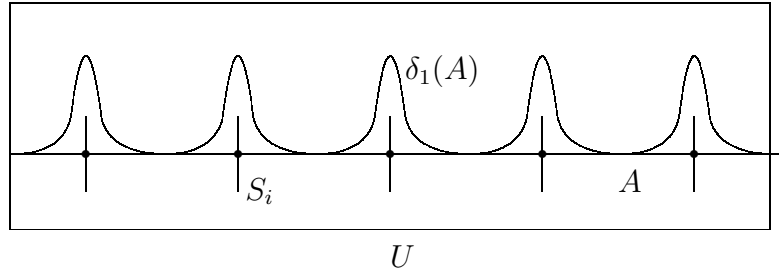
\begin{picture}(100,120)(-100,5)\setlength{\unitlength}{1cm}
\linethickness{.075mm}
\multiput(-1,1.5)(2,0){5}
{\line(0,1){1}}
\multiput(-2,2)(0,2){2}
{\line(1,0){10}}
\put(-2,1)
{\line(1,0){10}}

\multiput(-2,1)(10,0){2}
{\line(0,1){3}}

\multiput(-1,2)(2,0){5}{\circle*{.1}}

\multiput(-1,0)(2,0){5}{\qbezier(-1,2)(-.3,2)(-.2,2.5)
\qbezier(-.2,2.5)(-.1,3.3)(0,3.3)}
\multiput(-1,0)(2,0){5}{\qbezier(0,3.3)(.1,3.3)(.2,2.5)
\qbezier(.2,2.5)(.3,2)(1.1,2)}
\put(6,1.5){$A$}
\put(3.2,3){$\delta_1(A)$}
\put(1.1,1.4){$S_i$}
\put(3,.5){$U$}
\end{picture}\\
In the above diagram, $A$ is denoted by the horizontal line and the rectangle represents a neighbourhood $U$ of $A$ in $M$. The small vertical segments represent the set $S$. The intersection of $S$ with $A$ is shown by bullets. The curve in $U\setminus S$ represents the locus of $\delta_1(A)$.
\begin{remark}{\em The diffeotopies characterized by (1)-(3) above are referred as \emph{sharply moving diffeotopies} by Gromov (\cite{gromov_pdr}). It will appear once again in Definition~\ref{D:sharp_diffeotopy}
}\end{remark}

\subsection{Language of $h$-principle}
Let $p:X\to M$ be a smooth fibration.
\begin{definition}{\em A subset $\mathcal{R} \subset X^{(k)}$ of the $k$-jet space is called a \emph{partial differential relation of order} $k$ \index{partial differential relation} (or simply a \emph{relation}\index{relation}). If $\mathcal{R}$ is an open subset of the jet space then we call it an \emph{open relation}.
}\end{definition}
A $C^k$ section $f:M\rightarrow X$ is said to be a \emph{solution} of $\mathcal{R}$ if the image of its $k$-jet extension $j^k_f:M\rightarrow X^{(k)}$ lies in $\mathcal{R}$.
We denote by $\Gamma(\mathcal{R})$\index{$\Gamma(\mathcal{R})$} the space of sections of the $k$-jet bundle $X^{(k)}\to M$ having images in $\mathcal{R}$. The space of $C^\infty$ solutions of $\mathcal{R}$ is denoted by $Sol(\mathcal{R})$\index{$Sol(\mathcal{R})$}.

The $k$-jet map $j^k$ maps $Sol(\mathcal R)$ to $\Gamma(\mathcal R)$:
\[j^k:Sol(\mathcal R)\to \Gamma(\mathcal R)\]
and the image of $Sol(\mathcal R)$ under $j^k$ consists of all holonomic sections of $\mathcal R$.
The function spaces $Sol(\mathcal R)$ and $\Gamma(\mathcal R)$ will be endowed with the weak $C^{\infty}$ topology and the weak $C^0$ topology respectively.
\begin{definition}{\em A differential relation $\mathcal{R}$ is said to satisfy the \emph{$h$-principle} if every element $\sigma_{0} \in \Gamma(\mathcal{R})$ admits a homotopy $\sigma_{t}\in \Gamma(\mathcal{R})$ such that $\sigma_{1}$ is holonomic. We shall also say, in this case, that the solutions of $\mathcal R$ satisfies the $h$-principle.

The relation $\mathcal R$ satisfies the \emph{parametric $h$-principle} \index{parametric $h$-principle} if the $k$-jet map $j^{k}:Sol(\mathcal{R})\rightarrow \Gamma(\mathcal{R})$ is a weak homotopy equivalence.
}\index{$h$-principle}\end{definition}
\begin{remark}{\em We shall often talk about (parametric) $h$-principle for certain function spaces without referring to the relations of which they are solutions.} \end{remark}
Since $j^k$ is an injective map, $Sol(\mathcal R)$ can be identified with the holonomic sections of $\mathcal R$. Thus, if $\mathcal R$ satisfies the parametric $h$-principle, then it follows from the homotopy exact sequence of pairs that $\pi_i(\Gamma(\mathcal R),Sol(\mathcal R))=0$ for all integers $i\geq 0$. In other words, every continuous map $F_0:({\mathbb D}^{i},{\mathbb S}^{i-1})\to (\Gamma(\mathcal R),Sol(\mathcal R))$, $i\geq 1$, admits a homotopy $F_t$ such that $F_1$ takes all of ${\mathbb D}^{i}$ into $Sol(\mathcal R)$.

\begin{remark}{\em The space $\Gamma(\mathcal R)$ is referred as the space of formal solutions of $\mathcal R$. Finding a formal solution is a purely (algebraic) topological problem which can be addressed with the obstruction theory. Finding a solution of $\mathcal R$ is, on the other hand, a differential topological problem. Thus, the $h$-principle reduces a differential topological problem to a problem in algebraic topology.}\end{remark}

Let $Z$ be any topological space. Any continuous map $F:Z\to \Gamma(X)$ will be referred as a \emph{parametrized section }of $X$ with parameter space $Z$.
\begin{definition}{\em Let $M_0$ be a submanifold of $M$. We shall say that a relation $\mathcal R$ satisfies the \emph{$h$-principle near $M_0$} (or on  $Op(M_0)$) if given a section $F:U\to \mathcal{R}|_U$ defined on an open neighbourhood $U$ of $M_0$, there exists an open neighbourhood $\tilde{U}\subset U$ of $M_0$ such that $F|_{\tilde{U}}$ is  homotopic to a holonomic section $\tilde{F}:\tilde{U}\to \mathcal{R}$ in $\Gamma(\mathcal R)$.

Parametric $h$-principle is said to hold for $\mathcal R$ near $M_0$ if given any open set containing $\mathcal R$ and a
parametrized section $F_0:{\mathbb D}^i\to \Gamma(\mathcal{R}|_U)$ such that $F_0(z)$ is holonomic on $U$ for all $z\in {\mathbb S}^{i-1}$, there exists an open set $\tilde{U}$, $M_0\subset \tilde{U}\subset U$, and a homotopy $F_t:{\mathbb D}^i\to \Gamma(\mathcal{R}|_{\tilde{U}})$ satisfying the following conditions:
\begin{enumerate}\item $F_t(z)=F_0(z)$ for all $z\in {\mathbb S}^{i-1}$ and
\item $F_1$ maps ${\mathbb D}^i$ into $Sol(\mathcal R|_{\tilde{U}})$.
\end{enumerate}}
\end{definition}

\subsection{Open relations on open manifolds}
We shall here apply the Holonomic Approximation Theorem to open relations on open manifolds.
\begin{definition}{\em A manifold is said to be \emph{closed} if it is compact and without boundary. A manifold is \emph{open} if it is not closed.}
\end{definition}

\begin{remark}{\em
Every open manifold admits a Morse function $f$ without a local maxima. The codimension of the Morse complex of such a function is, therefore, strictly positive (\cite{milnor_morse},\cite{milnor}). The gradient flow of $f$ brings the manifold into an arbitrary small neighbourhood of the Morse complex. In fact, one can get a polyhedron $K\subset M$ such that codim\,$K>0$, and an isotopy $\phi_t:M\to M$, $t\in[0,1]$, such that $K$ remains pointwise fixed and $\phi_1$ takes $M$ into an
arbitrarily small neighborhood $U$ of $K$. The polyhedron $K$ is called a \emph{core} of $M$.}\label{core}\index{core}\end{remark}

\begin{proposition} Let $p:X\to M$ be a smooth vector bundle over an open manifold $M$. Let $\mathcal R$ be an open subset of the jet space $X^{(k)}$. Then given any section $\sigma$ of $\mathcal R$ there exist a core $K$ of $M$ and a holonomic section $\sigma':Op\,K\to X$ such that the linear homotopy between $\sigma$ and $\sigma'$ lies completely within $\Gamma(\mathcal R)$ over $Op\,K$.  \label{h-principle}\end{proposition}

\begin{proof} We fix a metric on $X^{(k)}$. Since $\mathcal R$ is an open subset of $X^{(k)}$, the space of sections of $\mathcal R$ is an open subset of $\Gamma(X^{(k)})$ in the fine $C^0$-topology. Therefore, given a section $\sigma$ of $\mathcal R$, there exists a positive function $\varepsilon$ satisfying the following condition:
\[\tau\in\Gamma(X^{(k)}) \text{ and } dist\,(\sigma(x),\tau(x))<\varepsilon(x)\ \ \ \Rightarrow \tau \text{ is a section of } \mathcal R\]
Consider a core $A$ of $M$ and a $\delta$-tubular neighbourhood of $A$ for some positive $\delta$. By the Holonomic Approximation Theorem (Theorem~\ref{T:holonomic_approximation}) there exist a diffeotopy $\delta_t$ and a holonomic section $\sigma'$ such that
\begin{enumerate}\item $dist\,(x,\delta_t(x))<\delta(x)$ for all $x\in M$ and $t\in [0,1]$ and
\item $dist\,(\sigma(x),\sigma'(x))<\varepsilon(x)$ for all $x\in U_\rho$,
\end{enumerate}
where $U_\rho$ is a $\rho$-tubular neighbourhood of $K=\delta_1(A)$ for some real number $\rho>0$. Now take a smooth map $\chi :M\rightarrow [0,1]$ satisfying the following conditions:
\[\chi  \equiv  1, \text{ on }  U_{\rho/2}\ \ \ \ \text{and}\ \ \ \text{supp\,}\chi \subset U_{\rho},\]
Define a homotopy $\sigma_t$ as follows:
\[\sigma_t = \sigma +t\chi (\sigma'-\sigma), \ \ t\in [0,1].\]
Then
\begin{enumerate}
\item $\sigma_0=\sigma$ and each $\sigma_t$ is globally defined;
\item $\sigma_t=\sigma$ outside $U_\rho$ for each $t$;
\item $\sigma_1$ is holonomic on $U_{\rho/2}$. \end{enumerate}
Moreover, since the above homotopy between $\sigma$ and $\sigma'$ is linear, $\sigma_t$ lies in the $\varepsilon$-neighbourhood of $\sigma$ for each $t$. Hence the homotopy $\sigma_t$ lies completely within $\mathcal R$ by the choice of $\varepsilon$. This completes the proof of the proposition since $K=\delta_1(A)$ is also a core of $M$.\end{proof}
\begin{remark}\end{remark}\begin{enumerate}\item[(a)] Note that the core $K$ can not be fixed a priori in the statement of the proposition.
\item[(b)] The theorem, in fact, remains valid in the general set up where $X$ is a smooth fibration \cite{gromov_pdr}. In this case, however, the linearity condition on the homotopy $\sigma_t$ has to be dropped for obvious reason.\end{enumerate}

\subsection{Open Diff invariant relations and $h$-principle}
The set of all diffeomorphisms on a manifold is a group under composition of maps. Let \emph{Diff}$(M)$ \index{\emph{Diff}$(M)$} denote the set of all local diffeomorphisms of $M$, i.e, all diffeomorphisms $f:U\to V$ where $U$, $V$ are open subsets of $M$. The composition of maps in \emph{Diff}$(M)$ is not defined for every pair of local diffeomorphisms. However, if $f,g$ are local diffeomorphisms of $M$, then $g\circ f\in Diff(M)$ if and only if domain of $g$ is equal to the codomain of $f$. A subset $\mathcal D$ of \emph{Diff}$(M)$ is called a \emph{pseudogroup} if the following conditions are satisfied (\cite{geiges}):

\begin{enumerate}
 \item If $f\in \mathcal D$ and $V$ is an open subset of the domain of $f$, then $f|_V:V\to f(V)$ is in $\mathcal D$.
 \item If the domain $U$ of $f$ has the decomposition $U=\cup_i U_i$ and if $f|_{U_i}:U_i\to f(U_i)\in \mathcal D$ for all $i$, then $f\in \mathcal D$.
 \item For any open set $U$, the identity map $id_U\in \mathcal D$.
 \item For any $f\in \mathcal D$, $f^{-1}\in \mathcal D$.
 \item If $f_1,f_2\in \mathcal D$ are such that $f_1\circ f_2$ is well defined then $f_1\circ f_2\in \mathcal D$.
\end{enumerate}
\begin{example}\end{example}
\begin{enumerate}\item \emph{Diff}$(M)$ has all the above properties.
%\item Let $(M,g)$ be a Riemannian manifold with the Riemannian metric $g$. Then the set of local diffeomorphisms of $M$ which preserve the metric $g$ is a pseudogroup.
\item The set of all local symplectomorphisms of a symplectic manifold $(M,\omega)$ preserving the symplectic form $\omega$ is a pseudogroup.
\item The set of all local contactomorphisms $\varphi$ of a contact manifold $(M,\xi)$ is a pseudogroup.
\end{enumerate}

\begin{definition}(\cite{geiges})
\em{A fibration $p:X\to M$ is said to be \emph{natural} if there exists a map $\Phi:Diff(M)\to Diff(X)$ having the following properties:
\begin{enumerate}
\item For $f\in Diff(M)$ with domain $U$ and target $V$, $\Phi(f):p^{-1}(U)\to p^{-1}(V)$ is such that $p\circ \Phi(f)=f\circ p$.
\begin{equation}
\xymatrix@=2pc@R=2pc{
p^{-1}(U) \ar@{->}[r]^-{\Phi(f)}\ar@{->}[d] & p^{-1}(V)\ar@{->}[d]\\
U \ar@{->}[r]_-{f} & V
}\label{D:extension}
\end{equation}

\item $\Phi(id_U)=id_{p^{-1}(U)}$.
\item If $f,g\in Diff(M)$ are composable, then $\Phi(f\circ g)=\Phi(f)\circ \Phi(g)$.
\item For any open set $U$ in $M$, $\Phi:Diff(U)\to Diff(p^{-1}(U))$ is continuous with respect to the $C^{\infty}$ compact open topologies.
\end{enumerate}
The map $\Phi$ satisfying (1) - (4) above is called a \emph{continuous extension} of \emph{Diff}$(M)$.
}
\end{definition}

\begin{example}\label{ex:extension}
\end{example}
\begin{enumerate}\item Let $X=M\times N$ be the trivial bundle over a manifold $M$ with fibre $N$ which is also a manifold. The group of diffeomorphisms of $M$ has a natural action on the space $C^\infty(M,N)$ given by $\delta\mapsto \delta^*f=f\circ\delta$, where $M, N$ are smooth manifolds. This gives an extension of \emph{Diff}$(M)$ to \emph{Diff}$(X)$ as follows: If $\delta:U\to V$ belongs to \emph{Diff}$(M)$ then
\[\Phi(\delta):(\text{id}_U, f)\mapsto (\text{id}_V,f\circ\delta^{-1}),\ \ \ f\in C^\infty(U,N).\]

\item All exterior bundles are natural. The pull-back operation on forms by maps define an extension of \emph{Diff}$(M)$ to \emph{Diff}$(\wedge^k(T^*M))$ for all $k\geq 1$: If $\delta:U\to V$ is a local diffeomorphism of $M$ then
    \[\Phi(\delta):\omega_x\mapsto (d\delta^{-1})_{\delta(x)}^*\omega_x,\ \ \ \omega_x\in \wedge^k(T_x^*M), x\in U.\]

\end{enumerate}
Any continuous extension $\Phi$ defines an `action' of \emph{Diff}$(M)$ on the space of local sections of $X$.
Furthermore, $\Phi$ naturally gives an extension of \emph{Diff}$(M)$ to \emph{Diff}$(X^{(k)})$ which we shall denote by $\Phi^k:Diff(M)\to Diff(X^{(k)})$. For any $f:U\to V$ in \emph{Diff}$(M)$, 
\begin{equation}\Phi^k(f)(j^k_x\sigma)=j^k_{f(x)}(\Phi(f)\circ \sigma \circ f^{-1}), \ \ x\in U\label{D:extension_jet}\end{equation}
This gives an `action' of \emph{Diff}$(M)$ on the jet space $X^{(k)}$.
For brevity, we shall denote the $k$-jet $\Phi^k(f)(j^k_x\sigma)$ by $f^*(j^k_x\sigma)$.
\begin{definition}(\cite{geiges})
 \em{Let $X\to M$ be a natural fibration with an extension $\Phi$. A relation $\mathcal{R}\subset X^{(k)}$ is said to be $\mathcal D$-invariant (for some pseudogroup $\mathcal D$)\index{$\mathcal D$-invariant} if $\Phi^k(f)$ maps $\mathcal{R}$ into itself for all $f\in\mathcal D$. We also say, in this case, that $\mathcal R$ is invariant under the action of $\mathcal D$.}
\end{definition}

\begin{example}\label{ex:invariant_relation}\end{example}
\begin{enumerate}\item Let $\mathcal R$ denote the relation consisting of 1-jets of germs of local immersions of a manifold $M$ into another manifold $N$. Then $\mathcal R$ can be identified with the subset of \emph{Hom}$(TM,TN)$ consisting of all injective linear maps. Hence, $\mathcal R$ is open. Also, it is easy to see that $\mathcal R$ is invariant under the natural action of \emph{Diff}$(M)$ (see Example~\ref{ex:extension}(1)). Similarly the relation consisting of 1-jets of germs of local submersions is also open and \emph{Diff}$(M)$-invariant.
\item Let $\mathcal R$ denote the set of 1-jets of germs of 1-forms $\alpha$ on a manifold $M$ such that $d\alpha$ is non-degenerate. Since non-degeneracy is an open condition, it can be shown that $\mathcal R$ is open (see Lemma~\ref{lemma_lcs}). Furthermore, it easy to see that if $\omega$ is a symplectic form then so is $f^*\omega$ for any diffeomorphism $f$ of $M$. Hence, $\mathcal R$ is clearly invariant under the natural action of \emph{Diff}$(M)$ on $(T^*M)^{(1)}$ (see Example~\ref{ex:extension}(2)).
\item Let $\mathcal R$ be the set of 1-jets of germs of contact forms on an odd-dimensional manifold $M$. The defining condition of contact forms (Definition~\ref{contact_form}) is an open condition; therefore,  $\mathcal R$ is an open relation. Moreover, if $\alpha$ is a contact form then $f^*\alpha$ is also contact for any diffeomorphism $f$ of $M$. Thus $\mathcal R$ is invariant under the natural action of \emph{Diff}$(M)$ on $(T^*M)^{(1)}$ (see Example~\ref{ex:extension}(2)).
\end{enumerate}

The following result, due to Gromov, is the first general result in the theory of $h$-principle. We shall refer to this result as Open Invariant Theorem for future reference\index{Open Invariant Theorem}.
\begin{theorem}(\cite{gromov}) Every open, Diff$(M)$ invariant relation $\mathcal R$ on an open manifold $M$ satisfies the parametric $h$-principle.\label{open-invariant}
\end{theorem}
\begin{proof} We give a very brief outline of the proof of ordinary $h$-principle. Since $M$ is an open manifold, it has a core $K$ which is by definition a polyhedron of positive codimension. Hence by Proposition~\ref{h-principle} and part (b) of the previous remark, any section $\sigma_0$ of $\mathcal R$ can be homotoped to a holonomic section $\sigma_1$ on an open neighbourhood $U$ of $K$ such that the homotopy $\sigma_t$ lies in $\Gamma(\mathcal R|_U)$. Now, $K$ being a core of the open manifold $M$, there exists an isotopy $\delta_t$ of $M$ such that $\delta_1$ maps $M$ into $U$. Since $\mathcal R$ is invariant under the action of \emph{Diff}$(M)$, the sections $\delta_1^*(\sigma_t)$, $t\in \I$, lie in $\mathcal R$; moreover, $\delta_1^*\sigma_1$ is holonomic. On the other hand, the homotopy $\delta_t^*\sigma_0$ also lies in $\mathcal R$. The concatenation of the two homotopies defines a homotopy between $\sigma_0$ and $\delta_1^*\sigma_1$ which is a holonomic section of $\mathcal R$. Thus $\mathcal R$ 
satisfies the $h$-principle.  \end{proof}
For a detailed proof of the above result we refer to \cite{haefliger2}.

\subsection{Open, non-Diff invariant relations and $h$-principle}

If a relation is invariant under the action of a smaller pseudogroup of diffeomorphism, say $\mathcal D$, then also we may expect $h$-principle to hold, provided $\mathcal D$ has some additional properties.
\begin{definition}  {\em (\cite{gromov_pdr}) Let $M_{0}$ be a submanifold of $M$ of positive codimension and let $\mathcal{D}$ be a pseudogroup of local diffeomorphisms of $M$. We say that $M_0$ \emph{is sharply movable} by $\mathcal{D}$, if given any hypersurface $S$ in an open set $U$ in $M_0$ and any $\varepsilon>0$, there is an isotopy $\delta_{t}$, $t\in\I$, in $\mathcal{D}$ and a positive real number $r$ such that the following conditions hold:
\begin{enumerate}\item[$(i)$] $\delta_{0}|_U=id_{U}$,
\item[$(ii)$] $\delta_{t}$ fixes all points outside the $\varepsilon$-neighbourhood of $S$,
\item[$(iii)$] $dist(\delta_{1}(x),M_{0})\geq r$ for all $x\in S$ and for some $r>0$,\end{enumerate}
where $dist$ denotes the distance with respect to any fixed metric on $M$.}\label{D:sharp_diffeotopy}\index{sharp diffeotopy}\end{definition}
The diffeotopy $\delta_t$ will be referred as a \emph{sharply moving diffeotopy}.
A pseudogroup $\mathcal D$ is said to have the \emph{sharply moving property} if every submanifold $M_0$ of positive codimension is sharply movable by $\mathcal D$.
\begin{example}\end{example}
\begin{enumerate}\item Let $M$ be a smooth manifold. A diffeomorphism $f:M\times\R\to M\times\R$ is called a fibre-preserving diffeomorphism if $\pi\circ f=\pi$, where $\pi:M\times\R\to M$ is the projection onto the first factor. Then the set $\mathcal{D}(M\times\R,\pi)$ consisting of fiber preserving diffeomorphisms of $M\times\R$ forms a subgroup of \emph{Diff}$(M\times\R)$. It is also easy to see that $\mathcal{D}(M\times\R,\pi)$ sharply moves $M=M\times \{0\}$ in $M\times\R$.
\item Symplectomorphisms of a symplectic manifold $(M,\omega)$ have the sharply moving property (\cite{gromov_pdr}).
\item Contactomorphisms of a contact manifold $(M,\alpha)$ also have the sharply moving property. (We refer to Theorem~\ref{CT} for a proof of this fact.)
\label{ex:sharp_diffeotopy}\end{enumerate}
We end this section with the following result due to Gromov (\cite{gromov_pdr}).
\begin{theorem}
\label{T:gromov-invariant}
Let $p:X\to M$ be a smooth fibration and $\mathcal{R}\subset X^{(r)}$ an open relation which is invariant under the action of a pseudogroup $\mathcal{D}$. If $\mathcal{D}$ sharply moves a submanifold $M_{0}$ in $M$ of positive codimension then the parametric h-principle holds for $\mathcal{R}$ on $Op\,(M_{0})$.
\end{theorem}
\begin{proof}Let $\sigma_0$ be a section of $\mathcal R$ on $Op\,M_0$. We apply the Holonomic Approximation theorem to $\sigma_0$ as in Proposition~\ref{h-principle}, and obtain a homotopy $\sigma_t$ in $\Gamma(\mathcal R)$ defined over $Op\,(\delta_1(M_0))$. However, this time we take the diffeotopies $\delta_t$ from $\mathcal D$. This can be done because $\mathcal D$ has the sharply moving property. Since $\mathcal R$ is invariant under the action of $\mathcal D$, we can bring the homotopy onto an open neighbourhood of $M_0$ by the action of $\mathcal D$. Indeed, the two homotopies $\delta_t^*\sigma_0$ and $\delta_1^*\sigma_t$ lie in $\Gamma(\mathcal R)$ over $Op\,(M_0)$. The concatenation of these two gives a path between $\sigma_0$ and the holonomic section $\delta_1^*\sigma_1$ within $\Gamma(\mathcal R|_{Op\,M_0})$ proving the $h$-principle.
\end{proof}
\newpage
\section{Some examples of $h$-principle}
\subsection{Early evidence of $h$-principle}
Early evidences of $h$-principle can be found in the work of Nash on isometric immersions (\cite{nash1}, \cite{nash2}) and in the work of Smale and Hirsch (\cite{hirsch},\cite{smale1},\cite{smale2}), Phillips (\cite{phillips},\cite{phillips1}) and Feit (\cite{feit}). The general framework of $h$-principle developed by Gromov unified these works and gave many new results. We state some of these results here which will be referred in the later chapters.
\begin{theorem}(Smale-Hirsch Immersion theorem \cite{hirsch})
Let $M$ and $N$ be smooth manifolds with $\dim M<\dim N$. Then the space of smooth immersions $M\to N$ is weak homotopy equivalent to the space of bundle monomorphisms $TM\to TN$.\end{theorem}

\begin{theorem}(Phillips Submersion theorem \cite{phillips})
\label{Phillips Submersion theorem}
Let $M$ be an open manifold such that $\dim M\geq \dim N$. Then the space of smooth submersions $M\to N$ is weak homotopy equivalent to the space of bundle epimorphisms $TM\to TN$.\end{theorem}

\begin{theorem}$($Gromov-Phillips Theorem \cite{gromov}, \cite{phillips1}$)$
\index{Gromov-Phillips Theorem}
\label{T:Gromov Phillip}
Let $M$ be an open manifold and $N$ a foliated manifold with a foliation $\mathcal F_N$. Let $\pi:TN\to\nu(\mathcal F_N)$ denote the projection onto the normal bundle of $\mathcal F_N$. Then the space of smooth maps $f:M\to (N,\mathcal F_N)$ transversal to $\mathcal F_N$ has the same weak homotopy type as the space of all bundle homomorphisms $F:TM\to TN$ such that $\pi\circ F:TM\to \nu(\mathcal F_N)$ is an epimorphism.
\end{theorem}
\begin{proof} Smooth maps $f:M\to N$ transversal to the foliation $\mathcal F_N$ are solutions of a first order relation $\mathcal{R}^{\pitchfork}$ on $M$ defined as follows:
\[\mathcal{R}^{\pitchfork}=\{(x,y,F)\in J^1(M,N)| \pi\circ F:T_xM\to \nu_y(\mathcal{F}_N) \text{ is an epimorphism}\} \]
The relation $\mathcal{R}^{\pitchfork}$ is open, as the set of all surjective linear maps $T_xM\to \nu_y\mathcal{F}_N$ is an open subset of \emph{Hom}$(TM,\nu(\mathcal F_N))$. Furthermore, $\mathcal{R}^{\pitchfork}$ is invariant under the action of \emph{Diff}$(M)$. Indeed, if $\delta:U\to V$ is in \emph{Diff}$(M)$ and $f:U\to N$ is transversal to $\mathcal F_N$ then clearly, $f\circ\delta$ is transversal to $\mathcal F_N$. Hence, by Theorem~\ref{open-invariant}, $\mathcal{R}^{\pitchfork}$ satisfies the parametric $h$-principle provided $M$ is open. Observe that a section of $\mathcal{R}^{\pitchfork}$ can be realised as a bundle morphism $F:TM\to TN$ such that $\pi\circ F:TM\to \nu(\mathcal{F_N})$ is an epimorphism. This completes the proof.\end{proof}

\subsection{$h$-principle in symplectic and contact geometry}
We have already noted in Example~\ref{ex:extension} that the diffeomorphism group of a manifold $M$ has a natural action on the space of differential forms on the manifold and the space of symplectic forms (resp. the space of contact forms) is invariant under this action (Example~\ref{ex:invariant_relation}). Furthermore, the non-degeneracy conditions on symplectic and contact forms are open conditions.
The following results in the symplectic and contact geometry were obtained as applications of Open Invariant Theorem (Theorem~\ref{open-invariant}).
\begin{theorem}(\cite{gromov}) Let $M$ be an open manifold and $\zeta$ be a fixed de Rham cohomology class in $H^2(M)$. Then the space of symplectic forms in the cohomology class $\zeta$ has the same homotopy type as the space of almost symplectic forms on $M$.\label{gromov_symplectic}
\end{theorem}
In Corollary~\ref{conformal_symp}, we shall obtain a similar classification for locally conformal symplectic forms on open manifolds.
\begin{definition}{\em Let $M$ be a manifold of dimension $2n+1$. An almost contact structure on $M$ is a pair $(\alpha,\beta)\in \Omega^1(M)\times \Omega^2(M)$ such that $\alpha \wedge \beta^n$ is a nowhere vanishing form on $M$.}\index{almost contact structure}\end{definition}

\begin{theorem}(\cite{gromov})
The space of contact forms on an open manifold has the same weak homotopy type as the space of almost contact structures on it.\label{gromov_contact}
\end{theorem}
The above results show that the obstruction to the existence of a symplectic form (resp. a contact form) on an open manifold is purely topological. The results are not true for closed manifolds.

\subsection{Homotopy classification of foliations\label{haefliger_map}}

Let $M$ be a smooth manifold and $Fol^q(M)$ be the set of all codimension $q$ foliations on $M$\index{$Fol^q(M)$}. Recall the classifying space $B\Gamma_q$ and the universal $\Gamma_q$ structure $\Omega_q$ on it (see Subsection \ref{classifying space}). If $\mathcal F\in Fol^q(M)$ and $f:M\to B\Gamma_{q}$ is a classifying map of $\mathcal F$, then  $f^*\Omega_q= \mathcal F$ as $\Gamma_q$-structure. We define a vector bundle epimorphisms $TM\to \nu\Omega_q$ by the following diagram (see \cite{haefliger1})
\begin{equation}
 \xymatrix@=2pc@R=2pc{
TM \ar@{->}[r]^-{\pi}\ar@{->}[rd] & \nu \mathcal{F}\cong f^*(\nu \Omega_q) \ar@{->}[r]^-{\bar{f}}\ar@{->}[d] & \nu \Omega_q \ar@{->}[d]\\
& M \ar@{->}[r]_-{f} & B\Gamma_{q}
}\label{F:H(foliation)}
\end{equation}
where $TM\to \nu(\mathcal F)$ is the quotient map and $(\bar{f},f)$ defines the pull-back diagram. The morphism $\bar{f}\circ \pi$ is defined uniquely only up to homotopy. Thus, there is a function
\[H':Fol^q(M)\to \pi_0(\mathcal E(TM,\nu\Omega_q)),\]
where $\mathcal E(TM,\nu\Omega_q))$ is the space of all vector bundle epimorphism $F:TM\to \nu \Omega_q$ and $\pi_0(\mathcal E(TM,\nu\Omega_q))$ is the set of its components.
\begin{definition} {\em Two foliations $\mathcal F_0$ and $\mathcal F_1$ on a manifold $M$ are said to be \emph{integrably homotopic} if there exists a foliation $\tilde{\mathcal F}$ on $M\times\R$ which is transversal to the trivial foliation of $M\times\R$ by leaves $M\times\{t\}$ ($t\in [0,1]$) and that the induced foliations on  $M\times\{0\}$ and $M\times\{1\}$ coincide with $\mathcal F_0$ and $\mathcal F_1$ respectively.}\label{D:integrably_homotopic}
\end{definition}
If $\mathcal F_0$ and $\mathcal F_1$ are integrably homotopic as in Definition~\ref{D:integrably_homotopic} and if $F:M\times[0,1]\to B\Gamma_q$ is a classifying map of $\tilde{\mathcal F}$ then we have a diagram similar to (\ref{F:H(foliation)}) given as follows:
\[
 \xymatrix@=2pc@R=2pc{
T(M\times[0,1]) \ar@{->}[r]^-{\bar{\pi}}\ar@{->}[rd] & \nu \tilde{F} \ar@{->}[r]^-{\bar{F}}\ar@{->}[d] & \nu \Omega_q \ar@{->}[d]\\
& M\times [0,1] \ar@{->}[r]_-{F} & B\Gamma_{q}
}
\]
Let $i_t:M\to M\times\{t\}\hookrightarrow M\times\R$ denote the canonical injective map of $M$ into $M\times\{t\}$ and $f_t:M\to B\Gamma_q$ be defined as $f_t(x)=F(x,t)$ for $(x,t)\in M\times[0,1]$. Then $\bar{F}\circ \bar{\pi}\circ di_t:TM\to \nu(\Omega_q)$ defines a homotopy between $\bar{f}_0\circ\pi$ and $\bar{f}_1\circ\pi$, where $(\bar{f}_i,f_i):\nu(\mathcal F_i)\to \nu\Omega_q$, $i=0,1$, denote the pull-back diagrams. Thus, we get $H'(\mathcal F_0)= H'(\mathcal F_1)$. Hence, $H'$ induces a function
\[H:\pi_0(Fol^q(M))\longrightarrow \pi_0(\mathcal E(TM,\nu\Omega_q)),\]
where $\pi_0(Fol^q(M))$ denotes the integrable homotopy classes of codimension $q$ foliations on $M$.
We shall refer to $H$ as the \emph{Haefliger map}\index{Haefliger map}.
\begin{theorem}(\cite{haefliger1})
\label{HCF}
If $M$ is an open manifold, then the Haefliger map induces a bijection between the sets $\pi_0(Fol^{q}(M))$ and $\pi_0(\mathcal E(TM,\nu\Omega_q))$.\index{Haefliger's Classification Theorem}
\end{theorem}
Let $Tr(M,\mathcal{F}_N)$ \index{$Tr(M,\mathcal{F}_N)$} be the space of smooth maps $f:M\to (N,\mathcal F_N)$ into a foliated manifold $(N,\mathcal F_N)$ and $\mathcal E(TM,\nu(\mathcal F_N))$ \index{$\mathcal E(TM,\nu(\mathcal F_N)$} denote the space of epimorphisms $F:TM\to \nu(\mathcal F_N)$. Then we have a commutative diagram
\[\xymatrix@=2pc@R=2pc{
\pi_0(Tr(M,\mathcal{F}_N))\ar@{->}[r]^-{P}\ar@{->}[d]_-{\cong}^-{\pi_0(\pi \circ d)} & \pi_0(Fol^{q}(M))\ar@{->}[d]^-{H}\\
\pi_0(\mathcal E(TM,\nu \mathcal{F}_N))\ar@{->}[r] & \pi_0(\mathcal E(TM,\nu\Omega_q))
}\]
in which the left vertical arrow is a bijection by Gromov-Phillips Theorem (Theorem~\ref{T:Gromov Phillip}). The function $P$ is induced by the natural map which takes an $f\in Tr(M,\mathcal F_N)$ onto the inverse foliation $f^*\mathcal F_N$. On the other hand, there is a reverse path from $Fol^{q}(M)$ to $Tr(M,\mathcal{F}_N)$ for some foliated manifold $(N,\mathcal F_N)$ as suggested in Theorem~\ref{HL}. These two observations reduce the classification of foliations to Gromov-Phillips Theorem.

\begin{corollary}(\cite{haefliger1}) Let $M$ be an open manifold of dimension $n$ and let $\tau:M\to BGL(n)$ be a classifying map of the tangent bundle $TM$. There is a one-to-one correspondence between the integrable homotopy classes of foliations on $M$ and the homotopy classes of lifts of $\tau$ in $B\Gamma_q\times BGL(n-q)$. In particular, a codimension $q$ distribution $D$ on $M$ is homotopic to a foliation if the classifying map of $TM/D$ lifts to $B\Gamma_q$.\label{C:haefliger}
\end{corollary}

In \cite{thurston}, Thurston  generalized Haefliger's result to closed manifolds. He viewed a $\Gamma_q$ structure on a manifold $M$ as a triple $\Sigma=(\nu,Z,\mathcal{F})$, where $\nu$ is a $q$-dimensional vector bundle on $M$ with a section $Z$ and $\mathcal{F}$ is a foliation of codimension $q$ on a neighbourhood $U$ of $Z(M)$ which is transversal to the fibers of $\nu$. If $\mathcal{G}$ is a foliation of codimension $q$ then we can associate a $\Gamma_q$ structure $\Sigma(\mathcal{G})$ on $M$ to it by taking $\nu=\nu(\mathcal{G})$, $Z=M\hookrightarrow \nu(\mathcal F)$ and $\mathcal{F}=(\exp|_\nu)^*\mathcal{G}$. The vector bundle $\nu$ in this case embeds in $TM$. In this setting Thurston proved the following.
\begin{theorem}(\cite{thurston})
Let $\Sigma=(\nu,Z,\mathcal{F})$ be a $\Gamma_q$ structure on a manifold $M$ with $q>1$. Then for any vector bundle monomorphism $i:\nu \to TM$, there exists a codimension $q$ foliation on $M$ whose induced $\Gamma_q$ structure is homotopic to $\Sigma$.
\end{theorem}

\chapter{Regular Jacobi structures on open manifolds}

In this chapter we shall prove that locally conformal symplectic foliations and contact foliations on open manifolds satisfy the $h$-principle. We also interpret these results in terms of regular Jacobi strucres. For basic definitions of foliations with geometric structures, we refer to Section~\ref{forms_foliations}.

\section{Background of the problem - $h$-principle in Poisson geometry}

In a recent article (\cite{fernandes}), Fernandes and Frejlich have proved the following $h$-principle for symplectic foliations.
\begin{theorem} Let $M$ be an open manifold equipped with a foliation $\mathcal{F}_0$ and a 2-form $\Omega_0$ on $\mathcal{F}_0$ which is nondegenerate. Then $(\mathcal{F}_0,\Omega_0)$ can be homotoped through such pairs to a pair $(\mathcal{F}_1,\Omega_1)$ such that $\Omega_1$ is a symplectic form on $\mathcal F$.\label{T:Fernandes_Poisson}
\end{theorem}
In the statement of Theorem~\ref{T:Fernandes_Poisson}, we can not replace $\mathcal F_0$ by an arbitrary distribution, since it need not be homotopic to any integrable distribution at all (See Corollary~\ref{C:haefliger}). However, we can replace $\mathcal F_0$ by a distribution which is homotopic to a foliation. Taking this into account, Fernandes and Frejlich interpreted the above theorem in terms of regular Poisson structures as follows. %However, an $h$-principle of general Poisson structure is still open.

\begin{theorem} $($\cite{fernandes}$)$ Every regular bivector field $\pi_0$ on an open manifold can be homotoped to a regular Poisson bivector provided the distribution \text{Im\,}$\pi_0^\#$ is homotopic to an integrable one.\label{hprinciple_fernandes}\end{theorem}
\noindent
Since a symplectic form on a manifold corresponds to a non-degenerate Poisson structure, the above result may be seen as a generalisation of Thereom~\ref{gromov_symplectic} due to Gromov.
The authors further remarked in \cite{fernandes} that there should be analogues of Theorem~\ref{hprinciple_fernandes} for Jacobi manifolds, in other words, for locally conformal symplectic foliations and contact foliations on open manifolds. The results of this chapter are inspired by this remark.

In this connection we also recall a result of M. Bertelson. She observed that symplectic forms on a given foliation $\mathcal F$ may not satisfy $h$-principle, even if the leaves of $\mathcal F$ are open manifolds (\cite{bertelson1}). However, she proved $h$-principle with some `strong open-ness' condition on $\mathcal F$ (\cite{bertelson}). Following Bertelson we shall refer to such foliated manifolds $(M,\mathcal F)$ as \emph{open foliated manifolds}.
Bertelson, in fact, obtained an $h$-principle for general relations on open foliated manifolds $(M,\mathcal F)$ which can be stated as follows:
\begin{theorem}(\cite{bertelson})
If $(M,\mathcal{F})$ is an open foliated manifold, then any relation $\mathcal R$ which is open and invariant under foliation preserving diffeotopies of $(M,\mathcal F)$ satisfies the  parametric $h$-principle.\label{foliation_hprinciple}
\end{theorem}
The $h$-principle for foliated symplectic forms was derived as a corollary of the above theorem by observing that the associated differential relation is open and invariant under the action of foliation preserving diffeotopies.
\begin{theorem} (\cite{bertelson}) Let $(M,\mathcal{F})$ be an open foliated manifold. Then every non-degenerate foliated 2-form $\omega_0$ on $(M,\mathcal F)$ is homotopic through such 2-forms to a symplectic form $\omega_1$ on $\mathcal F$. \label{bertelson}\end{theorem}
Theroem~\ref{bertelson} can also be viewed as an $h$-principle of regular Poisson structures with prescribed characteristic foliation. The requirement of an additional condition on the foliation is better understood when the result is stated in the following form:\\

\emph{Let $\pi_0$ be a regular bivector field (on a manifold $M$) for which the distribution $\mathcal D=\text{Im\ }{\pi_0}^{\#}$ integrates to a foliation satisfying some `strong open-ness' condition. Then $\pi_0$ can be homotoped through regular bivector fields $\pi_t$ to a Poisson bivector field $\pi_1$ such that the underlying distributions $\text{Im\ }{\pi_t}^{\#}$ remains constant.} \\

\begin{remark} {\em A contact analogue of Theorem~\ref{bertelson} also follows from Theorem~\ref{foliation_hprinciple}. Suppose that $(M,\mathcal F)$ is an open foliated manifold, where dimension of $\mathcal F$ is $2n+1$. Let $(\alpha,\beta)$ be a section of $E= T^*{\mathcal F}\oplus\Lambda^2(T^*\mathcal F)$ which gives an almost contact structure on $\mathcal F$. The nowhere vanishing condition on $\alpha\wedge\beta^n$  is an open condition and hence defines an open subset $\mathcal R$ in the 1-jet space $E^{(1)}$. The non-vanishing condition is also invariant under the action of foliation preserving diffeotopies and hence the general theorem of Bertelson applies to this relation. Therefore, the pair $(\alpha,\beta)$ can be homotoped in the space of almost contact structures on $\mathcal F$ to $(\eta,d_{\mathcal F}\eta)$ for some foliated 1-form $\eta$ on $(M,\mathcal F)$, where $d_\mathcal F$ is the coboundary map of the foliated deRham complex. Note that $\eta$ is then a contact form on the foliation $\
mathcal F$.}\end{remark}

\section{Locally conformal symplectic foliations}

In this section we prove an $h$-principle for locally conformal symplectic foliations on open manifolds.

\begin{lemma} Let $M^{n}$ be a smooth manifold with a 1-form $\theta$. Then there exists an epimorphism $D_{\theta}:E^{(1)}=(T^*M)^{(1)}\rightarrow \wedge^{2}(T^*M)$ satisfying $D_{\theta}\circ j^{1}_\alpha=d_{\theta}\alpha$ so that the following diagram is commutative:
\[\begin{array}{rcl}
  E^{(1)} & \stackrel{D_\theta}{\longrightarrow} & \wedge^{2}(T^*M)\\
  \downarrow &  & \downarrow \\
  M & \stackrel{\text{id}_{M}}{\longrightarrow} & M
 \end{array}\]
In particular, given any 2-form $\omega$ there exists a section $F_\omega:M\to E^{(1)}$ such that $D_\theta\circ F_\omega=\omega$.\label{lemma_lcs}
\end{lemma}
\begin{proof} Let $\theta$ be as in the hypothesis. Define $D_\theta(j^1_\alpha(x_0))=d_\theta\alpha(x_0)$ for any local 1-form $\alpha$ on $M$. To prove that the right hand side is independent of the choice of a representative $\alpha$,
choose a local coordinate system $(x^1,...,x^n)$ around $x_{0}\in M$. We may then express $\alpha$ and $\theta$ as follows:
\[\alpha=\Sigma_{i=1}^n \alpha_idx^i, \ \ \ \theta=\Sigma_{i=1}^n \theta_idx^i\]
where $\alpha_i$ and $\theta_i$ are smooth (local) functions defined in a neighbourhood of $x_0$.
The 1-jet $j^1_\alpha(x_0)$ is completely determined by the ordered tuple $(a_i,a_{ij})\in\mathbb{R}^{(n+n^{2})}$, where
\[a_i=\alpha_i(x_0), \ \ \ a_{ij}=\frac{\partial \alpha_{i}}{\partial x^{j}}(x_{0}), \ \ \ i,j=1,2,\dots,n.\]
Now,
\[\begin{array}{rcl}d_\theta\alpha(x_0) & = &  d\alpha(x_0)+\theta(x_0)\wedge\alpha(x_0)\\
& = &  \Sigma_{i<j}[(a_{ji}-a_{ij})+(\theta_i(x_0)a_j-a_i\theta_j(x_0))]dx^i\wedge dx^j\end{array}\]
This shows that $d_\theta\alpha(x_0)$ depends only on the 1-jet $j^1_\alpha$ at $x_0$ and the value of $\theta(x_0)$. Since $\theta$ is fixed, $D_\theta$ is well-defined. Clearly, $D_{\theta}\circ j^1_\alpha = d_\theta(\alpha)$ for any 1-form $\alpha$.

It is easy to check that $D_\theta$ is a vector bundle epimorphism. Indeed, given a set of real numbers $b_{ij}, 1\leq i<j\leq n$, the following system of linear equations
\[(a_{ij}-a_{ji})+(a_{i}\theta_{j}(x_{0})-a_{j}\theta_{i}(x_{0}))=b_{ij}\]
has a solution, namely $a_i=0$, $a_{ij}=-a_{ji}=\frac{b_{ij}}{2}$.
Therefore, the fibres of $D_\theta$ are affine subspaces and hence contractible. This implies that $D_\theta$ has a right inverse. Hence every section $\omega :M \rightarrow \wedge^{2}M$ can be lifted to a section $F_{\omega}:M \rightarrow (T^*M)^{(1)}$ such that $D_{\theta} F_{\omega}=\omega$. Moreover, any two such lifts of $\omega$ are homotopic.\end{proof}

\begin{proposition}Let $M$ be an open manifold and $\mathcal F_0$ be a foliation on $M$. Let $\theta$ be a closed 1-form on $M$.
Then any $\mathcal F_0$-leafwise non-degenerate 2-form $\omega_0$ on $M$ can be homotoped through such forms to a 2-form $\omega_1$ which is $d_\theta$-exact on a neighbourhood $U$ of some core $K$ of $M$.\label{approx_lcs}
\end{proposition}
\begin{proof} Let $\mathcal S$ denote the set of all elements $\omega_x$ in $\wedge^2(T_x^*M)$, $x\in M$, such that the restriction of $\omega_x$ is non-degenerate on $T_x\mathcal F_0$. Since non-degeneracy is an open condition, $\mathcal S$ is an open subset of $\wedge^2(T^*M)$.
Let
\[{\mathcal R}_\theta = D_\theta^{-1}(\mathcal S)\subset (T^*M)^{(1)},\]
where $D_\theta:(T^*M)^{(1)}\to \wedge^2(T^*M)$ is defined as in Lemma~\ref{lemma_lcs}. Then ${\mathcal R}_\theta$ is an open relation. Let $\sigma_0$ be a section of ${\mathcal R}_\theta$ such that $D_\theta\circ \sigma_0=\omega_0$. By Proposition~\ref{h-principle}, there exists a homotopy of sections $\sigma_t:M\to \mathcal R_{\theta}$, such that $\sigma_1$ is holonomic on an open neighbourhood $U$ of $K$, where $K$ is a core of $M$. Therefore, $\sigma_1=j^1_\alpha$ for some 1-form $\alpha$ on $U$. The 2-forms $\omega_t=D_\theta\circ \sigma_t$, $t\in [0,1]$, are sections of $\wedge^2(T^*M)$ with values in $\mathcal S$. Hence,
\begin{enumerate}
\item $\omega_t$ is $\mathcal{F}_{0}$-leafwise non-degenerate for all $t\in [0,1]$, and
\item $\omega_1=d_\theta\alpha$ on $U$; in particular $\omega_1$ is $d_\theta$-closed on $U$.
\end{enumerate}
\end{proof}

\begin{theorem} Let $M^{2n+q}$ be an open manifold with a codimension $q$ foliation $\mathcal{F}_{0}$ and a 2-form $\omega_0$ on $M$ which is $\mathcal{F}_{0}$-leafwise non-degenerate. Let $\xi \in H_{deR}^{1}(M,\R)$ be a fixed de Rham cohomology class. Then there exists a homotopy $(\mathcal{F}_t, \omega_t)$ and a closed 1-form $\theta_0$ representing $\xi$ such that
\begin{enumerate}\item $\omega_t$ is $\mathcal{F}_t$-leafwise non-degenerate and
\item $\omega_1$ is $d_{\theta_0}$-closed, that is, $d\omega_1+\theta_0\wedge \omega_1=0$.
\end{enumerate} \label{lcs}\end{theorem}
\begin{proof}
To prove the result we proceed as in \cite{fernandes}.
Consider the canonical Grassmann bundle $G_{2n}(TM)\stackrel{\pi}{\longrightarrow}M$ for which the fibres $\pi^{-1}(x)$ over a point $x\in M$ is the Grassmannian of $2n$-planes in $T_x M$. The space $Dist_q(M)$ of codimension $q$ distributions on $M$ can be identified with the section space $\Gamma(G_{2n}(M))$. We topologize $Dist_q(M)$ by the $C^\infty$ compact open topology.
The space $Fol_{q}(M)$ consisting of codimension $q$ foliations can be viewed as a subspace of $Dist_q(M)$ if we identify a foliation with its tangent distribution.
Let $\Phi_q$ be the subspace of $Dist_q(M) \times \Omega^2(M)$ defined as follows:
\[\Phi_q=\{(\mathcal F,\omega)|\omega \text{ is } \mathcal F\text{-leafwise symplectic}\}\]
Fix a 1-form $\theta$ which represents the class $\xi$. By Proposition~\ref{approx_lcs}, there exists a homotopy of 2-forms, $\omega'_t$, $0\leq t\leq 1$, such that
\begin{enumerate}\item $\omega_0'=\omega_0$
\item $\omega'_t$ is $\mathcal{F}_0$-leafwise non-degenerate and
\item $\omega'_1$ is $d_\theta$-closed on some open set $U$ containing a core $K$ of $M$.
\end{enumerate}
Then $(\mathcal{F}_0,\omega'_t) \in \Phi_q$ for $0\leq t\leq 1$.
Since $M$ is an open manifold there exists an isotopy $g_t$, $0\leq t\leq 1$, with $g_{0}=id_{M}$ such that $g_1$ takes $M$ into $U$ (see Remark~\ref{core}). Now, we define $(\mathcal{F}''_t, \omega''_t) \in \Phi_{q}$ for  $t\in [0,1]$ by setting
\begin{center}$\mathcal F''_t=g_t^*\mathcal F_0, \ \ \ \omega''_t=g_t^* \omega'_1.$\end{center}
Then, $\omega''_t$ is $\mathcal F''_t$-leafwise non-degenerate. Further, it is easy to see that $\omega''_1$ is $d_{g_1^*\theta}$ closed: Indeed,
\[\begin{array}{rcl}
d_{g_1^*{\theta}}\omega''_1 & = & d_{g_1^*{\theta}}(g_1^* \omega'_1)\\
& = & dg_1^* \omega'_1 + g_1^*\theta \wedge g_1^* \omega'_1\\
&=& g_1^*[d \omega'_1+ \theta \wedge \omega'_1]\\
&=& g_1^*d_\theta \omega'_1=0
\end{array}\]
since $\omega'_1$ is $d_\theta$-closed on $U$ and $g_1$ maps $M$ into $U$. Since $g_1$ is homotopic to the identity map of $M$ the de Rham cohomology class $[g_1^*\theta]=[\theta]=\xi$. The desired homotopy is obtained by the concatenation of the two homotopies, namely $(\mathcal{F}_0,\omega'_t)$ and $(\mathcal{F}''_t,\omega''_t)$, and taking $\theta_0=g_1^*\theta$.\end{proof}
{\bf Remark} Theorem~\ref{T:Fernandes_Poisson} follows as a particular case of the above result by taking $\theta$ equal to zero.

\begin{theorem} Let $M$ be an open manifold and $\xi$ be any de Rham cohomology class in $H^1(M,\R)$. Then every almost symplectic foliation $(\mathcal F_0,\omega_0)$ is homotopic to a locally conformal symplectic foliation $(\mathcal F_1,\omega_1)$ with foliated Lee form $\theta$ such that the canonical morphism $H^2(M,\R)\to H^2(M,\mathcal F_1)$ maps $\xi$ onto the foliated de-Rham cohomology class of $\theta$.
\end{theorem}

\begin{proof} Let $\tilde{\omega}_0$ be a global 2-form on $M$ which extends $\omega_0$. By Theorem~\ref{lcs}, we get a homotopy $(\mathcal F_t, \tilde{\omega}_t)$ and a closed 1-form $\tilde{\theta}$ representing $\xi$ satisfying the following:
\begin{enumerate}
\item $\tilde{\omega}_t$ is $\mathcal F_t$-leafwise non-degenerate,
\item $d\tilde{\omega}_1+\tilde{\theta}\wedge\tilde{\omega}_1=0$.
\end{enumerate}
Let $\omega_t$ be a foliated 2-form obtained by restricting $\tilde{\omega}_t$ to $T\mathcal F_t$ and let $\theta$ be the restriction of $\tilde{\theta}$ to $T\mathcal F_1$. Note that, we have a commutative diagram as follows:
\[
 \xymatrix@=2pc@R=2pc{
 \Omega^k(M) \ar@{->}[r]^-{d}\ar@{->}[d]_-{r} &  \Omega^{k+1}(M)\ar@{->}[d]^-{r}\\
 \Gamma(\wedge^k(T^*\mathcal F_1))\ar@{->}[r]_-{d_{\mathcal{F}_1}} & \Gamma(\wedge^{k+1}(T^*\mathcal{F}_1))
 }
\]
where the vertical arrows are the restriction maps. Hence, relation (2) above implies that $d_{{\mathcal F}_1}\omega_1+\theta\wedge\omega_1=0$; thus, $(\mathcal F_1, \omega_1)$ is a locally conformal symplectic foliation on $M$ and $\theta$ is the foliated Lee class of $\omega_1$.
Further, the foliated de Rham cohomology class of $\theta$ in $H^1(M,\mathcal F_1)$ is the image of $\xi$ under the induced morphism $H_{deR}^1(M,\R)\to H^1(M,\mathcal F_1)$.
\end{proof}

\begin{corollary}Let $M$ be an open manifold and $\omega$ be a non-degenerate 2-form on $M$. Given any de Rham cohomology class $\xi\in H^1(M,\R)$,
$\omega$ can be homotoped through non-degenerate 2-forms to a locally conformal symplectic form $d_\theta \alpha$, where the deRham cohomology
class of $\theta$ is $\xi$.\label{conformal_symp}\end{corollary}

\begin{proof}This is a direct consequence of Theorem~\ref{lcs}. Indeed, the form $\omega_1$ in the theorem can be taken to be $d_\theta$-exact (see Proposition~\ref{approx_lcs}).\end{proof}

\section{Contact foliations}
In this section we prove an $h$-principle for contact foliations on open manifolds.
\begin{lemma} Let $M^n$ be a smooth manifold and $E=T^*M$ be the cotangent bundle of $M$. Then there exists a vector bundle epimorphism $\bar{D}$
\[\begin{array}{rcl}
  E^{(1)} & \stackrel{\bar{D}}{\longrightarrow} & T^*M\oplus \wedge^{2}(T^*M)\\
\downarrow &  & \downarrow \\
M & \stackrel{id_{M}}{\longrightarrow} & M
\end{array}\]
such that $\bar{D}\circ (j^{1}_\alpha)=(\alpha,d\alpha)$ for any 1-form $\alpha$ on $M$. Moreover, any section of $T^*M\oplus\wedge^2(T^*M)$ can be lifted to a section of $E^{(1)}$ through $\bar{D}$. \label{lemma_contact}\end{lemma}

\begin{proof} Define $\bar{D}$ by
\[\bar{D}(j^1_\alpha(x_0))=(\alpha(x_0),d\alpha(x_0))\]
for any local 1-form $\alpha$ defined near a point $x_0$. It follows from the proof of Lemma~\ref{lemma_lcs} that this map is well defined. Hence $\bar{D}\circ j^1_\alpha=(\alpha, d\alpha)$ for any 1-form $\alpha$.
Let $(x^{1},...,x^{n})$ be a local coordinate system around $x_{0}\in M$ and $\alpha=\Sigma_{i=1}^{n}\alpha_{i}dx^{i}$ be the representation of $\alpha$ with respect to these coordinates. Then $j^{1}_{\alpha}(x_{0})$ is uniquely determined by the ordered tuple $(a_{i},a_{ij})\in \mathbb{R}^{n+n^{2}}$ as in Lemma~\ref{lemma_lcs} and
\[\bar{D}(j^{1}_{\alpha}(x_{0}))= (\alpha(x_{0}),d\alpha(x_{0}))= (\Sigma_{i=1}^n a_i dx^i, \Sigma_{i<j}(a_{ij}-a_{ji})dx^i \wedge dx^j)\]
It is easy to see that the following system of equations
\begin{center}
$a_i=b_i$ \ \ and \ \ $a_{ij}-a_{ji}=b_{ij}$ \ \ for all $i\neq j$, $i,j=1,...,n$.
\end{center}
is solvable in $a_i$ and $a_{ij}$. Hence, $\bar{D}$ is an epimorphism, and so the fibres of $\bar{D}$ are affine subspaces.
Consequently, any section $(\theta,\omega):M\rightarrow T^*M\oplus \wedge^{2}(T^*M)$ can be lifted to a section $F_{(\theta,\omega)}:M\rightarrow E^{(1)}$ such that $\bar{D}\circ F_{(\theta,\omega)}=(\theta,\omega)$ and any two such lifts of a given $(\theta,\omega)$ are homotopic.\end{proof}

\begin{proposition}Let $M$ be an open manifold and $\mathcal F_0$ be a foliation on $M$. Let $(\theta_0,\omega_0)$ be a pair consisting of a 1-form $\theta_0$ and a 2-form $\omega_0$ on $M$ such that the restriction of $(\theta_0,\omega_0)$ to the leaves of $\mathcal F_0$ are almost contact structures. Then $(\theta_0,\omega_0)$ can be homotoped through such pairs to a pair $(\theta_1, \omega_1)$, where $\omega_1=d\theta_1$ on a neighbourhood $U$ of some core $K$ of $M$.\label{approx_contact}
\end{proposition}
\begin{proof} Let $\mathcal C$ denote the set of all pairs $(\theta_x, \omega_x)\in T_x^*M\times \wedge^2(T_x^*M)$, $x\in M$, such that $\iota_D^*\theta_x\wedge\iota_D^*\omega_x\neq 0$, where $D=T_x\mathcal F$. Then $\mathcal C$ is an open  subset of $T^*M\oplus \wedge^2(T^*M)$.
Let
\[{\mathcal R} = \bar{D}^{-1}(\mathcal C)\subset E^{(1)},\]
where $E=T^*M$ and $\bar{D}$ is as in Lemma~\ref{lemma_contact}. Then ${\mathcal R}$ is an open first order relation. Let $\sigma_0$ be a section of $\mathcal R$ such that  $\bar{D}\circ\sigma_0=(\theta_0,\omega_0)$. By Proposition~\ref{h-principle}, there exists a homotopy of sections $\sigma_t$ lying in ${\mathcal R}$ such that  $\sigma_1$ is holonomic on an open neighbourhood $U$ of some core $K$ of $M$. Thus, there exists $1$-form $\theta_1$ on $U$ such that $\sigma_1=j^1\theta_1$.
Evidently, the pairs $(\theta_t,\omega_t)=\bar{D}\circ \sigma_t$, $t\in [0,1]$ are sections of $T^*M\oplus \wedge^2(T^*M)$ with values in $\mathcal C$. Hence,
\begin{enumerate}
\item $(\theta_t,\omega_t)$ is a $\mathcal{F}_{0}$-leafwise almost contact structures and
\item $\omega_1=d\theta_1$ on $U$.
\end{enumerate}
This completes the proof.
\end{proof}

\begin{theorem} Let $M^{(2n+1)+q}$ be an open manifold and $\mathcal{F}_{0}$ a codimension $q$ foliation on $M$. Let $(\theta_{0},\omega_{0})\in \Omega^{1}(M)\times \Omega^{2}(M)$ be a $\mathcal F_0$-leafwise almost contact structure. Then there exists a homotopy $(\mathcal{F}_{t},\theta_{t},\omega_{t})$ such that
\begin{enumerate}
\item $(\theta_{t},\omega_{t})$ is a $\mathcal{F}_{t}$-leafwise almost contact structure and
\item $\omega_{1}=d\theta_{1}$.
\end{enumerate}
In particular, $\theta_1$ is leafwise contact form on $(M,\mathcal F_1)$. \label{contact}\end{theorem}
\begin{proof}
Let $Dist_q(M)$ denote the space of all codimension $q$ distribution on $M$, as in Theorem~\ref{lcs}. Define a subset $\Phi_q$ of $Dist_{q}(M)\times \Omega^1(M) \times \Omega^2(M)$ as follows:
\[\Phi_{q} =\{(\mathcal F,\alpha,\beta): (\alpha,\beta) \text{ restricts to an almost contact structure on }\mathcal F\}.\]
By the given hypothesis, $(\mathcal{F}_0,\theta_0,\omega_0)$ is in $\Phi_q$.
By Proposition~\ref{approx_contact} there exists a homotopy $(\theta'_t,\omega'_t)$ such that
\begin{enumerate}
\item $(\theta'_0,\omega'_0)=(\theta_0,\omega_0)$
\item $(\theta'_t,\omega'_t)$ is a $\mathcal{F}_{0}$-leafwise almost contact structures and
\item $d\theta'_1=\omega'_1$ on some open set $U$ containing a core of $M$.
\end{enumerate}
Then $(\mathcal{F}_0,\theta'_t,\omega'_t)$ belongs to $\Phi_{q}$ for $0\leq t\leq 1$.
Choose an isotopy $g_{t}:M\rightarrow M$ such that $g_{0}=id_{M}$ and $g_{1}(M)\subset U$. Now, we define $(\mathcal{F}''_t,\theta''_t,\omega''_t)\in \Phi_{q}$, $t\in[0,1]$ by setting
\begin{center}$\mathcal{F}''_t= g_t^*(\mathcal{F}_0),\ \ \
\theta''_t=g_t^*\theta'_1,\ \ \ \omega''_t=g_t^*\omega'_1.$\end{center}
Observe that,
\[d\theta''_1=dg_1^*\theta'_1= g_1^*d\theta'_1= g_1^*\omega'_1=\omega''_1,\]
since $g_1(M)\subset U$ and $d\theta_1=\omega_1$ on $U$. Therefore, $\theta''_1$ is a ${\mathcal F}''_1$-leafwise contact form. Concatenating the homotopies $(\mathcal F_0,\theta'_t,\omega'_t)$ and $(\mathcal F''_t,\theta''_t,\omega''_t)$ we obtain the desired homotopy.
\end{proof}

\begin{theorem} Let $M$ be an open manifold. Then every almost contact foliation  $(\mathcal F_0, \theta_0,\omega_0)$ is homotopic to a contact foliation. 
\label{T:contact}
\end{theorem}

\begin{proof} Choose global differential forms $\tilde{\theta}_0,\tilde{\omega}_0$ on $M$ which extend $\theta_0$ and $\omega_0$ respectively. By Theorem~\ref{contact}, we get a homotopy $(\mathcal F_t, \tilde{\theta}_t,\tilde{\omega}_t)$ satisfying the following:
\begin{enumerate}
\item $(\tilde{\theta}_t,\tilde{\omega}_t)$ restrict to an almost complex structure on $\mathcal F_t$,
\item $d\tilde{\theta}_1= \tilde{\omega}_1$.
\end{enumerate}
Let $\omega_t$ and $\theta_t$ be foliated forms obtained by restricting $\tilde{\omega}_t$ and $\tilde{\theta}_t$ to $T\mathcal F_t$.
Clearly, $(\mathcal F_t,\theta_t,\omega_t)$, $0\leq t\leq 1$, is an almost contact foliation on $M$. Also, by restricting both sides of relation (2) to $T\mathcal F_1$ we get $d_{{\mathcal F}_1}\theta_1=\omega_1$; thus, $(\mathcal F_1, \theta_1)$ is, in fact, a contact foliation on $M$.\end{proof}

\section{Regular Jacobi structures on open manifolds}

We now reformulate Theorems ~\ref{lcs} and ~\ref{contact} in terms of Jacobi structures. Let $\nu^k(M)$ \index{$\nu^k(M)$} denote the space of sections of the alternating bundle $\wedge^k(TM)$. We shall refer to these sections as $k$-\emph{multivector fields} on $M$. We may recall that every bivector field $\Lambda$ defines a bundle homomorphism $\Lambda^\#:T^*M\to TM$ by $\Lambda^\#(\alpha)=\Lambda(\alpha,\ \ )$, for all $\alpha\in T^*M$.
\begin{definition}{\em A bivector field $\Lambda$ is said to be \emph{regular} if rank\,$\Lambda^{\#}$ is constant. A pair $(\Lambda,E)\in\nu^2(M)\times\nu^1(M)$ will be called a regular pair if $\mathcal D=\Lambda^\#(T^*M)+\langle E\rangle$ is a regular distribution on $M$.}\end{definition}

If $\Lambda$ is a regular bivector field and $\text{Im\,}\Lambda^\#=\mathcal D$, then there exists a bundle homomorphism  $\phi:\mathcal D^*\to \mathcal D$ such that the following square is commutative:
\begin{equation}
\begin{array}{rcl}
T^*M & \stackrel{\Lambda^\#}{\longrightarrow} & TM\\
i^* {\downarrow} &  & {\uparrow} i\\
\mathcal D^*& \stackrel{\phi}{\longrightarrow} & \mathcal D
\end{array}\label{bivector lambda}
\end{equation}
where $i:\mathcal D\to TM$ is the inclusion map. In fact, we can define $\phi$ by $\phi(i^*\alpha)=\Lambda^\#\alpha$ for all $\alpha\in T^*M$. If $i^*\alpha=0$ then $\alpha|_{\mathcal D}=0$. The skew symmetry property of $\Lambda$ implies that $\alpha\in\ker\Lambda^\#$. Hence $\phi$ is well defined. Moreover, it is an isomorphism as $\text{Im\,}\phi=\text{Im\,}\Lambda^{\#}=\mathcal D$.
We can define a section $\omega$ of $\wedge^2\mathcal D^*$ by
\[\omega(\Lambda^{\#}\eta,\Lambda^{\#}\eta')=\Lambda(\eta,\eta'),\]
for any two 1-forms $\eta,\eta'$ on $M$. This is well-defined. Moreover $\omega$ is non-degenerate, since
$\omega(\Lambda^{\#}\eta,\Lambda^{\#}\eta')=0$ for all $\eta'$ implies that $\eta'(\Lambda^\#(\eta)=0$ for all $\eta'$ and therefore, $\Lambda^\#(\eta)=0$.
If $\tilde{\omega}:\mathcal D\to \mathcal D^*$ is given by by $\tilde{\omega}(X)=i_X\omega$ for all $X\in\Gamma\mathcal D$, then we have the relation $\tilde{\omega}\circ\Lambda^\#=-i^*$, and since $\tilde{\omega}$ is an isomorphism we have $\Lambda^\#=-\tilde{\omega}^{-1}\circ i^*$. Thus $\tilde{\omega}$ is the inverse of $-\phi$.
Conversely, any section $\omega$ of $\wedge^2(\mathcal D^*)$ which is fibrewise non-degenerate, defines a bivector field $\Lambda$ by the relation $\Lambda^\#=-\tilde{\omega}^{-1}\circ i^*$. Observe that the image of $\Lambda^\#=\mathcal D$.

In view of the above correspondence, we can interpret Theorem~\ref{lcs} as follows.

\begin{theorem} Let $M$ be an open manifold with a regular bivector field $\Lambda_0$ such that the distribution $\mathcal D_0=\text{Im\,}\Lambda_0^\#$ is integrable. Let $\xi$ be a fixed de Rham cohomology class in $H^1(M,\R)$. Then there is a homotopy $\Lambda_t$ of regular bivector fields and a vector field $E_1$ on $M$ such that
\begin{enumerate}\item $\mathcal D_t=\text{Im\,}\Lambda_t^\#$ is an integrable distribution for all $t\in [0,1]$,
\item $E_1$ is a section of $\mathcal D_1$ and
\item $(\Lambda_1,E_1)$ is a regular Jacobi pair.\end{enumerate}
Furthermore, we can choose $E_1$ such that the foliated de Rham cohomology class of $\phi_1^{-1} (E_1)$ in $H^1(M,\mathcal F_1)$ is equal to the image of $\xi$ under $i^*:H^1(M,\R)\to H^1(M,\mathcal F_1)$, where $\mathcal F_1$ is the characteristic foliation of the Jacobi pair $(\Lambda_1,E_1)$.
\label{even_jacobi}\end{theorem}
\begin{proof} Suppose that $\mathcal D_0=\text{Im\,}\Lambda_0^\#$ integrates to a foliation $\mathcal F_0$. It follows from the above discussion that the associated section $\omega_0\in \Gamma(\wedge^2(\mathcal D_0^*))$ is non-degenerate. By Theorem ~\ref{lcs}, there exists a homotopy $(\mathcal F_t,\omega_t)$ of $(\mathcal F_0,\omega_0)$
such that $(\mathcal F_1,\omega_1)$ is a locally conformal symplectic foliation. Let $\mathcal D_t=T\mathcal F_t$ and define $\Lambda_t$ by a diagram analogous to (\ref{bivector lambda}). If $\theta_1$ is the Lee form of $\omega_1$ then define $E_1$ by the relation $i_{E_1}\omega_1=\theta_1$. This proves that $(\Lambda_1,E_1)$ is a regular Jacobi pair (Theorem~\ref{T:jacobi_foliation}).
\end{proof}

\begin{theorem} Let $(\Lambda_{0},E_{0})\in \nu^2(M)\times\nu^1(M)$ be a regular pair on an open manifold $M$. Suppose that the distribution $\mathcal{D}_0:= \text{Im\,}\Lambda_0^\#+\langle E_0\rangle$ is odd-dimensional and integrable. Then there is a homotopy of regular pairs $(\Lambda_t,E_t)$ of $(\Lambda_0,E_0)$ such that
\begin{enumerate}\item $\mathcal{D}_t=\text{Im\,}\Lambda_t^\#+\langle E_t\rangle$, $t\in [0,1]$, are integrable distributions and
\item $(\Lambda_1,E_1)$ is a Jacobi pair.
\end{enumerate}\label{odd-jacobi}\end{theorem}
\begin{proof} Suppose that $(\Lambda,E)$ is a regular pair and the distribution $\mathcal D=\Lambda^\#(T^*M)+\langle E\rangle$ is odd dimensional, then we can define a section $\alpha$ of $\mathcal D^*$ by the relations
\begin{equation}\alpha(\text{Im}\,\Lambda^\#)=0 \ \ \text{ and }\ \ \ \alpha(E)=1.\label{alpha}\end{equation}
Also, we can define a section $\beta$ of $\wedge^2(\mathcal D^*)$ by
\begin{equation}i_E\beta=0,\ \ \ \ \ \beta(\Lambda^\#\eta,\Lambda^\#\eta')=\Lambda(\eta,\eta') \text{ for all }\eta,\eta'\in \Omega^1(M)\label{beta},\end{equation}
where $i_E$ denotes the interior multiplication by $E$. It can be shown easily that $\beta$ is non-degenerate on Im\,$\Lambda^\#=\ker\alpha$. Hence  $\alpha\wedge \beta^n$ is nowhere vanishing.

On the other hand, suppose that $\mathcal D$ is a $(2n+1)$-dimensional distribution. If $\alpha$ is a section of $\mathcal D^*$ and $\beta$ is a section of $\wedge^2(\mathcal D^*)$ such that $\alpha\wedge\beta^n$ is nowhere vanishing, then we can write $\mathcal D=\ker\alpha\oplus\ker \beta$. Define a vector field $E$ on $M$ satisfying the relations
\begin{equation}i_E\beta=0, \ \ \text{and}\ \ \alpha(E)=1\label{vector E}\end{equation}
Since $\beta$ is non-degenerate on $\ker\alpha$ by our hypothesis, $\tilde{\beta}:\ker\alpha \to (\ker\alpha)^*$ is an isomorphism. For any $\eta\in T^*(M)$ define $\Lambda^\#(\eta)$ to be the unique element in $\ker\alpha$ such that $\tilde{\beta}(\Lambda^{\#}\eta)=-\eta|_{\ker\alpha}$. In other words,
\begin{equation} \Lambda^\#=-\tilde{\beta}^{-1}\circ i^*.\label{jacobi_lambda}\end{equation}
This relation shows that the image of $\Lambda^\#$ is equal to $\ker \alpha$ and $\ker\beta$ is spanned by $E$. Hence $\mathcal D=\text{Im\,}\Lambda^\#\oplus \langle E\rangle$ which means that $(\Lambda,E)$ is a regular pair. Thus there is a one to one correspondence between regular pairs $(\Lambda,E)$ and the triples $(\mathcal D,\alpha,\beta)$ such that $\alpha\wedge\beta^n$ is nowhere vanishing. Further, the regular contact foliations correspond to regular Jacobi pairs with odd-dimensional characteristic distributions under this correspondence \cite{kirillov}.

The result now follows directly from Theorem~\ref{contact}.
Let $(\Lambda_0,E_0)$ be as in the hypothesis and $\mathcal F_0$ be the foliation such that $T\mathcal F_0=\mathcal D_0$. We can define $(\alpha_0,\beta_0)$ by the equations (\ref{alpha}) and (\ref{beta}) so that $\alpha_0\wedge \beta_0^n$ is non-vanishing on $\mathcal D_0$. By Theorem~\ref{contact}, we obtain a homotopy $(\mathcal F_t,\alpha_t,\beta_t)$ of $(\mathcal F_0, \alpha_0,\beta_0)$ such that $\alpha_t\wedge\beta_t^n$ is a nowhere vanishing form on $T\mathcal F_t$ and $\beta_1=d\alpha_1$ on $\mathcal F_1$, so that $(\mathcal F_1,\alpha_1)$ is a contact foliation. The desired homotopy $(\Lambda_t,E_t)$ is then obtained from  $(\alpha_t,\beta_t)$ by (\ref{vector E}) and (\ref{jacobi_lambda}).
\end{proof}
We conclude with the following remark.
\begin{remark} {\em The integrability condition on the initial distribution in Theorems~\ref{lcs} and \ref{contact} can be relaxed to the extent that we can take the initial distribution to be \emph{homotopic} to an integrable one.
We refer to  \cite{fernandes} for a detailed argument.}\end{remark}

\chapter{Contact foliations on open contact manifolds}
In this chapter we shall give a complete homotopy classification of contact foliations on open contact manifolds. On our way to the classification result, we study equidimensional contact immersions which plays a very significant role in the proof. We also prove a general $h$-principle for open relations on open contact manifolds which are invariant under an action of local contactomorphisms. This leads to an extension of Gromov-Phillips Theorem in the contact setting. We shall begin with a review of similar results in the context of symplectic manifolds.
\section{Backgrouond: Symplectic foliations on symplectic manifolds}
In \cite{datta-rabiul}, M. Datta and Md. R. Islam proved an extension of Theorem~\ref{open-invariant} which can be stated as follows.
\begin{theorem}
 \label{ST}
 Let $(M,\omega)$ be an open symplectic manifold and $\mathcal{R}\subset J^{r}(M,N)$ be an open relation which is invariant under the action of the pseudogroup of local symplectomrphisms of $(M,\omega)$. Then $\mathcal R$ satisfies the $h$-principle.
\end{theorem}
The symplectic diffeotopies have the sharply moving property (Definition~\ref{D:sharp_diffeotopy}, Example~\ref{ex:sharp_diffeotopy}); hence the relation satisfies the local $h$-principle near a core $K$ by Theorem~\ref{T:gromov-invariant}. The global $h$-principle follows with a consequence of Ginzburg's theorem (Theorem~\ref{T:equidimensional-symplectic-immersion}) which guarantees a deformation of $M$ through isosymplectic immersions into a neighbourhood of $K$ (\cite{datta-rabiul}). As a corollary of it the authors obtained the following result.
\begin{theorem}(\cite{datta-rabiul}) If $(M,\omega)$ is an open contact manifold then submersions $f:M\to N$ whose level sets are symplectic submanifolds of $(M,\omega)$ satisfy the $h$-principle.
\end{theorem}
In fact, we can obtain a generalisation of the above result for maps which are transversal to a foliation $\mathcal F_N$ on $N$. We denote by $\pi:TN\rightarrow \nu\mathcal{F}_N$ the projection onto the normal bundle of the foliation $\mathcal F_N$.
Let $Tr_{\omega}(M,\mathcal{F}_N)$ \index{$Tr_{\omega}(M,\mathcal{F}_N)$} be the set of all smooth maps $f:M\rightarrow N$ transversal to $\mathcal{F}_N$ such that $\ker (\pi\circ df)$ is a symplectic subbundle of $(TM,\omega)$. Let $\mathcal E_{\omega}(TM,\nu{\mathcal{F}_N})$ \index{$\mathcal E_{\omega}(TM,\nu{\mathcal{F}_N})$} be the set of all vector bundle morphisms $F:TM\rightarrow TN$ such that
\begin{enumerate}
\item $\pi\circ F$ is an epimorphism onto $\nu{\mathcal{F}_N}$ and
\item $\ker (\pi\circ F)$ is a symplectic subbundle of $(TM,\omega)$.
\end{enumerate}
These spaces, as before, will be equipped with the $C^\infty$ compact open topology and the $C^0$ compact open topology respectively. Then we have the following extension of Gromov-Phillips Theorem in the symplectic setting:
\begin{theorem}
 Let $(M^{2m},\omega)$ be an open symplectic manifold and $N$ be any manifold with a foliation $\mathcal{F}_N$ of codimension $2q$, where $m>q$. Then the map
  \[\begin{array}{rcl}\pi\circ d:Tr_{\omega}(M,\mathcal{F}_N) & \rightarrow & \mathcal E_{\omega}(TM,\nu{\mathcal{F}_N})\\
  f & \mapsto & \pi\circ df\end{array}\]
is a weak homotopy equivalence.\label{T:symplectic transverse}
\end{theorem}
The maps in $Tr_{\omega}(M,\mathcal{F}_N)$ are solutions of an open relation $\mathcal{R}$ which is invariant under the action of local symplectomorphisms.
Hence the result follows as a direct application of Theorem~\ref{ST}. We would like to observe that the relation in Theorem~\ref{ST}, in fact, satisfies the parametric $h$-principle. %Theorem~\ref{T:symplectic transverse} would play an important role in the homotopy classification of symplectic foliations on open symplectic manifold.

\begin{definition}
{\em A foliation $\mathcal F$ on a symplectic manifold $(M,\omega)$ will be called a \emph{symplectic foliation subordinate to $\omega$} if its leaves are symplectic submanifolds of $(M,\omega)$. We shall often mention these foliations simply as \emph{symplectic foliations on} $(M,\omega)$}\end{definition}
\begin{definition} {\em Two symplectic foliations $\mathcal F_0$ and $\mathcal F_1$ on a symplectic manifold $(M,\omega)$ are said to be \emph{integrably homotopic relative to $\omega$} if there exists a foliation $\tilde{\mathcal F}$ on $(M\times\I,\omega\oplus 0)$ transversal to the trivial foliation of $M\times\R$ by leaves $M\times\{t\}$ ($t\in [0,1]$) such that the following conditions are satisfied:
\begin{enumerate}\item
the induced foliation on $M\times\{t\}$ for each $t\in [0,1]$ is a symplectic foliation subordinate to $\omega$;
\item the induced foliations on  $M\times\{0\}$ and $M\times\{1\}$ coincide with $\mathcal F_0$ and $\mathcal F_1$ respectively,\end{enumerate}
where $\omega\oplus 0$ denotes the pull-back of $\omega$ by the projection map $p_1:M\times\R\to M$.}\end{definition}
Let $Fol^{2q}_{\omega}(M)$ be the space of all codimension $2q$ symplectic foliations on the symplectic manifold $(M,\omega)$ and let $\pi_0(Fol^{2q}_{\omega}(M))$ denote the integrable homotopy classes of symplectic foliations on $(M,\omega)$. The map $H'$ defined in Subsection~\ref{haefliger_map} induces a map
\[H_\omega:\pi_0(Fol^{2q}_{\omega}(M)) \longrightarrow \pi_0(\mathcal E_{\omega}(TM,\nu\Omega_{2q})),\] 
where $\Omega_{2q}$ is the universal $\Gamma_{2q}$-structure on $B\Gamma_{2q}$ (Subsection~\ref{classifying space}) and $\mathcal E_{\omega}(TM,\nu\Omega_{2q})$ is the space of all vector bundle epimorphisms from $F:TM\to \nu \Omega_{2q}$ such that $\ker F$ is a symplectic subbundle
of $(TM,\omega)$. Indeed, if $\mathcal{F}$ is a symplectic foliation on $M$ (subordinate to $\omega$), then the kernel of $H'(\mathcal F)$ is $T\mathcal F$ which is by hypothesis a symplectic subbundle of $TM$. Therefore, $H_\omega$ is well-defined. Proceeding as in \cite{haefliger} we can then obtain the following classification result.
\begin{theorem}
\label{T:symplectic foliation}
 The map $\pi_0(Fol^{2q}_{\omega}(M)) \stackrel{H_{\omega}}\longrightarrow \pi_0(\mathcal E_{\omega}(TM,\nu\Omega_{2q}))$ is bijective.
\end{theorem}
We have omitted the proofs of Theorem~\ref{T:symplectic transverse} and Theorem~\ref{T:symplectic foliation} here to avoid repetition of arguments. In the subsequent sections we shall deal with the classification problem of contact foliations on open contact manifolds in full details. The proofs of the above theorems will be very similar to Theorem~\ref{T:contact-transverse} and Theorem~\ref{haefliger_contact}.

\section{Equidimensional contact immersions}

In this section we get an analogue of Ginzburg's theorem (Theorem~\ref{T:equidimensional-symplectic-immersion}) in the contact setting. We begin with a simple observation.

\begin{observation}{\em
Let $(M,\alpha)$ be a contact manifold. The product manifold $M\times\R^2$ has a canonical contact form given by $\tilde{\alpha}=\alpha- y\,dx$, where $(x,y)$ are the coordinate functions on $\R^2$. We shall denote the contact structure associated with $\tilde{\alpha}$ by $\tilde{\xi}$.
Now suppose that $H:M\times\R\to \R$ is a smooth function which vanishes on some open set $U$. Define $\bar{H}:M\times\R\to M\times\R^2$ by $\bar{H}(u,t)=(u,t,H(u,t))$ for all $(u,t)\in M\times\R$. Since $\bar{H}(u,t)=(u,t,0)$ for all $(u,t)\in U$, the image of $d\bar{H}_{(u,t)}$ is $T_uM\times\R\times\{0\}$. On the other hand,  $\tilde{\xi}_{(u,t,0)}=\xi_u\times\R^2$. Therefore, $\bar{H}$ is transversal to $\tilde{\xi}$ on $U$.
}\end{observation}

\begin{proposition} Let $M$ be a contact manifold with contact form $\alpha$. Suppose that $H$ is a smooth real-valued function on $M\times(-\vare,\vare)$ with compact support
such that its graph $\Gamma$ in $M\times\R^2$ is transversal to the kernel of $\tilde{\alpha}=\alpha-y\,dx$. Then there is a diffeomorphism $\Psi:M\times(-\vare,\vare)\to \Gamma$ which pulls back $\tilde{\alpha}|_{\Gamma}$ onto $h(\alpha\oplus 0)$, where $h$ is a nowhere-vanishing smooth real-valued function on $M\times\R$.\label{characteristic}
\end{proposition}
\begin{proof} Since the graph $\Gamma$ of $H$ is transversal to $\tilde{\xi}$, the restriction of $\tilde{\alpha}$ to $\Gamma$ is a nowhere vanishing 1-form on it.
Define a function $\bar{H}:M\times(-\vare,\vare)\to M\times\R^2$ by $\bar{H}(u,t)=(u,t,H(u,t))$. The map $\bar{H}$ defines a diffeomorphism of $M\times(-\vare,\vare)$ onto $\Gamma$, which pulls back the form $\tilde{\alpha}|_{\Gamma}$ onto $\alpha-H\,dt$. It is therefore enough to obtain a diffeomorphism $F:M\times(-\vare,\vare)\to M\times(-\vare,\vare)$ which pulls back the 1-form $\alpha-H\,dt$ onto a multiple of $\alpha\oplus 0$.
%Observe that $\Gamma$ is globally defined as the zero set of the function $\Phi:M\times\R^2\to\R$ given by $\Phi(u,x,y)=H(u,x)-y$.
%Hence the relation $i_{\tilde{X}} d\tilde{\alpha}=d\Phi$ on $\ker\tilde{\alpha}$ uniquely defines the characteristic vector field on $\Gamma$.
For each $t$, define a smooth function $H^t$ on $M$ by $H^t(u)=H(u,t)$ for all $u\in M$. Let
$X_{H^t}$ denote the contact Hamiltonian vector field on $M$ associated with $H^t$. Consider the vector field
$\bar{X}$ on $M\times\R$ as follows: \[\bar{X}(u,t)=(X_{H^t}(u),1),\ \ (u,t)\in M\times(-\vare,\vare).\]
Let $\{\bar{\phi}_s\}$ denote a local flow of $\bar{X}$ on $M\times\R$. Then writing $\bar{\phi}_s(u,t)$ as
\begin{center}$\bar{\phi}_s(u,t)=(\phi_s(u,t),s+t)$ for all $u\in M$ and $s,t\in \R$,\end{center}
we get the following relation:
\[\frac{d\phi_s}{ds}(u,t)=X_{t+s}(\phi_s(u,t)),  \]
where $X_t$ stands for the vector field $X_{H^t}$ for all $t$. In particular, we have
\begin{equation}\frac{d\phi_t}{dt}(u,0)=X_t(\phi_t(u,0)),\label{flow_eqn}\end{equation}
Define a level preserving map $F:M\times(-\vare,\vare)\to M\times(-\vare,\vare)$ by
\[F(u,t)=\bar{\phi}_t(u,0)=(\phi_t(u,0),t).\]
Since the support of $H$ is contained in $K\times (-\vare,\vare)$ for some compact set $K$, the flow $\bar{\phi}_s$ starting at $(u,0)$ remains within $M\times (-\vare,\vare)$ for $s\in (-\vare,\vare)$.
Note that
\begin{center}$dF(\frac{\partial}{\partial t})=\frac{\partial}{\partial t}\bar{\phi}_t(u,0)=\bar{X}(\bar{\phi}_t(u,0))=\bar{X}(\phi_t(u,0),t)=(X_{H^t}(\phi_t(u,0)),1).$\end{center}
This implies that
\begin{center}$\begin{array}{rcl}F^*(\alpha\oplus 0) (\frac{\partial}{\partial t}|_{(u,t)}) &  = & (\alpha\oplus 0)(dF(\frac{\partial}{\partial t}|_{(u,t)}))\\
%& = & (\alpha\oplus 0)(\bar{X}(\bar{\phi}_t(u,0))\\
%& = & \alpha(\bar{X}(\phi_t(u,0),t))
& = & \alpha(X_{H^t}(\phi_t(u,0)))\\
& = &H^t(\phi_t(u,0))\ \ \ \ \ \text{by equation }(\ref{contact_hamiltonian1}) \\
& = & H(\bar{\phi}_t(u,0))=H(F(u,t)) \end{array}$\end{center}
Also,
\begin{center}$F^*(H\,dt)(\frac{\partial}{\partial t})=(H\circ F)\,dt(dF(\frac{\partial}{\partial t}))=H\circ F$\end{center}
Hence,
\begin{equation}F^*(\alpha-Hdt)(\frac{\partial}{\partial t})=0.\label{eq:F1}\end{equation}
On the other hand,
\begin{equation}F^*(\alpha-H\,dt)|_{M\times\{t\}}=F^*\alpha|_{M\times\{t\}}=\psi_t^*\alpha,\label{eq:F2}\end{equation}
where $\psi_t(u)=\phi_t(u,0)$, $\psi_0(u)=u$. Thus, $\{\psi_t\}$ are the integral curves of the time dependent vector field $\{X_t\}$ on $M$ (see (\ref{flow_eqn})), and we get
\[\begin{array}{rcl}\frac{d}{dt}\psi_t^*\alpha & = & \psi^*_t(i_{X_t}d\alpha+d(i_{X_t}\alpha))\\
& = & \psi^*_t(dH^t(R_\alpha)\alpha-dH^t+dH^t) \ \ \text{by equation }(\ref{contact_hamiltonian1})\\
& = & \psi^*_t(dH^t(R_\alpha)\alpha)\\
& = & \theta(t)\psi_t^*\alpha,\end{array}\]
where $\theta(t)=\psi_t^*(dH^t(R_\alpha))$. Hence $\psi_t^*\alpha=e^{\int_0^t\theta(s)ds}\psi_0^*\alpha=e^{\int_0^t\theta(s)ds}\alpha$. We conclude from equation (\ref{eq:F1}) and (\ref{eq:F2}) that $F^*(\alpha-H\,dt)=e^{\int_0^t\theta(s)ds}\alpha$. Finally, take $\Psi=\bar{H}\circ F$ which has the desired properties.
\end{proof}
\begin{remark}{\em
If there exists an open subset $\tilde{U}$ of $M$ such that $H$ vanishes on $\tilde{U}\times (-\vare,\vare)$ then the contact Hamiltonian vector fields $X_t$ defined above are identically zero on $\tilde{U}$ for all $t\in(-\vare,\vare)$. Since $\psi_t=\phi_t(\ \ ,0)$ are the integral curves of the time dependent vector fields $X_t=X_{H^t}$, $0\leq t\leq 1$, we must have $\psi_t(u)=u$ for all $u\in \tilde{U}$. Therefore, $F(u,t)=(u,t)$ and hence $\Psi(u,t)=(u,t,0)$ for all $u\in\tilde{U}$ and all $t\in(-\vare,\vare)$.}\label{R:characteristic}\end{remark}

\begin{remark}{\em
%We begin with the notion of characteristics on an hypersurface of a contact manifold.
If $\Gamma$ is a codimension 1 submanifold of a contact manifold $(N,\tilde{\alpha})$ such that the tangent planes of $\Gamma$ are transversal to $\tilde{\xi}=\ker\tilde{\alpha}$ then there is a codimension 1 distribution $D$ on $\Gamma$ given by the intersection of $\ker \tilde{\alpha}|_\Gamma$ and $T\Gamma$. Since $D=\ker\tilde{\alpha}|_\Gamma\cap T\Gamma$ is an odd dimensional distribution, $d\tilde{\alpha}|_D$ has a 1-dimensional kernel. If $\Gamma$ is locally defined by a function $\Phi$ then $d\Phi_x$ does not vanish identically on $\ker\tilde{\alpha}_x$, for $\ker d\Phi_x$ is transversal to $\ker\tilde{\alpha}_x$. Thus there is a unique non-zero vector $Y_x$ in $\ker\tilde{\alpha}_x$ satisfying the relation $i_{Y_x}d\alpha_x=d\Phi_x$. Clearly, $Y_x$ is tangent to $\Gamma$ at $x$ and it is defined uniquely only up to multiplication by a non-zero real number (as $\Phi$ is not unique). However, the 1-dimensional distribution on $\Gamma$ defined by $Y$ is uniquely defined by the contact form $\tilde{\
alpha}$. The integral curves of $Y$ are called \emph{characteristics} of $\Gamma$ (\cite{arnold}).

It can be checked in the proof of the above proposition, that the diffeomorphism $\Psi$ maps the lines in $M\times\R$ onto the characteristics on $\Gamma$.}
\end{remark}

The following lemma is a parametric version of a result proved in \cite{eliashberg}. As we shall see later, it is a key ingradient in the proof of equidimensional contact immersions for open manifolds.
\begin{lemma}
\label{EM2} Let $\alpha_{t}, t\in[0,1]$, be a continuous family of contact forms on a compact manifold $M$, possibly with non-empty boundary. Then for each $t\in [0,1]$, there exists a sequence of primitive 1-forms $\beta_t^l=r_t^l\,ds_t^l, l=1,..,N$ such that
\begin{enumerate}
\item $\alpha_t=\alpha_0+\sum_1^N \beta_t^l$ for all $t\in [0,1]$,
\item for each $j=0,..,N$ the form $\alpha^{(j)}_t=\alpha_{0}+\sum_{1}^{j}\beta_t^{l}$ is contact,
\item for each $j=1,..,N$ the functions $r_t^j$ and $s_t^j$ are compactly supported within a coordinate neighbourhood.
\end{enumerate}
Furthermore, the forms $\beta_t^l$ depends continuously on $t$.

If $\alpha_t=\alpha_0$ on $Op\,V_0$, where $V_0$ is a compact subset contained in the interior of $M$, then the functions $r^l_t$ and $s^l_t$ can be chosen to be equal to zero on an open neighbourhood of $V_0$.
\end{lemma}
\begin{proof} If $M$ is compact and with boundary, then we can embed it in a bigger manifold, say $\tilde{M}$, of the same dimension. We may assume that $\tilde{M}$ is obtained from  $M$ by attaching a collar along the boundary of $M$. 
Using the compactness property of $M$, one can cover $M$ by finitely many coordinate neighbourhoods $U^i, i=1,2,\dots,L$.
Choose a partition of unity $\{\rho^i\}$ subordinate to $\{U^i\}$.
\begin{enumerate}\item Since $M$ is compact, the set of all contact forms on $M$ is an open subspace of $\Omega^1(M)$ in the weak topology. Hence, there exists a $\delta>0$ such that $\alpha_t+s(\alpha_{t'}-\alpha_t)$ is contact for all $s\in[0,1]$, whenever $|t-t'|<\delta$.
\item Get an integer $n$ such that $1/n<\delta$.
Define for each $t$ a finite sequence of contact forms, namely $\alpha^j_t$, interpolating between $\alpha_0$ and $\alpha_t$ as follows:
\[\alpha^j_t=\alpha_{[nt]/n}+\sum_{i=1}^j\rho^i(\alpha_t-\alpha_{[nt]/n}),\]
where $[x]$ denotes the largest integer which is less than or equal to $x$ and $j$ takes values $1,2,\dots,L$. In particular, for $k/n\leq t\leq (k+1)/n$, we have
\[\alpha^j_t=\alpha_{k/n}+\sum_{i=1}^j\rho^i(\alpha_t-\alpha_{k/n}),\]
and $\alpha^L_t=\alpha_t$ for all $t$.
\item Let $\{x_j^i:j=1,\dots,m\}$ denote the coordinate functions on $U^i$, where $m$ is the dimension of $M$. There exists unique set of smooth functions $y_{t,k}^{ij}$ defined on $U^i$ satisfying the following relation:
\[\alpha_t-\alpha_{k/n}=\sum_{j=1}^m y_{t,k}^{ij} dx^i_j\ \ \text{on } U^i \text{ for }k/n\leq t\leq (k+1)/n\]
Further, note that $y_{t,k}^{ij}$ depends continuously on the parameter $t$ and $y_{t,k}^{ij}=0$ when $t=k/n$, $k=0,1,\dots,n$.
\item Let $\sigma^i$ be a smooth function such that $\sigma^i\equiv 1$ on a neighbourhood of $\text{supp\,}\rho^i$ and $\text{supp}\,\sigma^i\subset U^i$. Define functions $r^{ij}_{t,k}$ and $s^{ij}$, $j=1,\dots,m$, as follows:
        \[ r_{t,k}^{ij}=\rho^i y_t^{ij} \ \ \ s^{ij}=\sigma^i x^i_j.\]
These functions are compactly supported and supports are contained in $U^i$. It is easy to see that $r^{ij}_{t,k}=0$ when $t=k/n$ and
\[\rho^i(\alpha_t-\alpha_{k/n})=\sum_{j=1}^m r_{t,k}^{ij}\,ds^{ij} \ \text{ for }\ t\in [k/n,(k+1)/n].\]
\end{enumerate}
It follows from the above discussion that $\alpha_t-\alpha_{k/n}$ can be expressed as a sum of primitive forms which depends continuously on $t$ in the interval $[k/n,(k+1)/n]$.
We can now complete the proof by finite induction argument. Suppose that
$(\alpha_{t}-\alpha_0)=\sum_l\alpha_{t,k}^l$ for $t\in [0,k/n]$, where each $\alpha_{t,k}^l$ is a primitive 1-form. Define
\[\begin{array}{rcl}\tilde{\alpha}_{t,k}^l& = & \left\{
\begin{array}{cl} \alpha_{t,k}^l & \text{if } t\in [0,k/n]\\
\alpha_{k/n,k}^l & \text{if } t\in [k/n,(k+1)/n]\end{array}\right.\end{array}\]
Further define for $j=1,\dots,N$, $i=1,\dots,L$,
\[\begin{array}{rcl}\beta_{t,k}^{ij}& = & \left\{
\begin{array}{cl} 0 & \text{if } t\in [0,k/n]\\
r^{ij}_{t,k}\,ds^{ij} & \text{if } t\in [k/n,(k+1)/n]\end{array}\right.\end{array}\]
Finally note that for $t\in [0,(k+1)/n]$, we can write $\alpha_t-\alpha_0$ as the sum of all the above primitive forms. Indeed, if $k/n\leq t<(k+1)/n$, then
\begin{eqnarray*}\alpha_t-\alpha_0 & = & (\alpha_t-\alpha_{k/n})+(\alpha_{k/n}-\alpha_0)\\
& = & \sum_{i=1}^L\sum_{j=1}^m r^{ij}_{t,k}\,ds^{ij}+\sum_l \alpha^l_{k/n,k}\\
& = & \sum_{i,j}\beta^{ij}_{t,k}+\sum_l \tilde{\alpha}^l_{t,k}.\end{eqnarray*}
The same relation holds for $0\leq t\leq k/n$, since $\beta^{ij}_{t,k}$ vanish for all such $t$. This proves the first part of the lemma.

Now suppose that $\alpha_t=\alpha_0$ on an open neighbourhood $U$ of $V_0$. Choose two compact neighbourhoods of $V_0$, namely $K_0$ and $K_1$ such that $K_0\subset \text{Int\,}K_1$ and $K_1\subset U$. Since $M\setminus\text{Int\,}K_1$ is compact we can cover it by finitely many coordinate neighbourhoods $U^i$, $i=1,2,\dots,L$, such that $(\bigcup_{i=1}^L U^i)\cap K_0=\emptyset$. Proceeding as above we get a decomposition of $\alpha_t$ on $\bigcup_{i=1}^L U^i$ into primitive 1-forms $r^l_t\,ds^l_t$. Observe that  $\{U^i:i=1,\dots,L\}\cup\{U\}$ is an open covering of $M$ in this case. The functions $r^l_t$ and $s^l_t$ can be extended to all of $M$ without disturbing their supports. Hence, the functions $r^l_t$ and $s^l_t$ vanish on $K_0$. This completes the proof of the lemma.

\end{proof}

\begin{theorem}
\label{GRAY}
Let $\xi_{t}$, $t\in[0,1]$ be a family of contact structures defined by the contact forms $\alpha_t$ on a compact manifold $M$ with boundary. Let $(N,\tilde{\xi}=\ker\eta)$ be a contact manifold without boundary. Then every isocontact immersion $f_0:(M,\xi_0)\to (N,\tilde{\xi})$ admits a regular homotopy $\{f_t\}$ such that $f_t:(M,\xi_t)\to (N,\tilde{\xi})$ is an isocontact immersion for all $t\in[0,1]$.

In addition, if $M$ contains a compact submanifold $V_{0}$ in its interior and $\xi_{t}=\xi_{0}$ on $\it{Op}(V_{0})$ then $f_{t}$ can be chosen to be a constant homotopy on $Op\,(V_{0})$.\label{T:equidimensional_contact immersion}
\end{theorem}

\begin{proof}  In view of Lemma~\ref{EM2}, it is enough to assume that $\alpha_{t}=\alpha_{0}+r_tds_t$, $t\in [0,1]$, where $r_t,s_t$ are smooth real valued functions (compactly) supported in an open set $U$ of $M$. We shall first show that $f_{0}:(M,\xi_{0})\rightarrow (N,\tilde{\xi})$ can be homotoped to an immersion $f_{1}:M\to N$ such that $f_{1}^{*}\tilde{\xi}=\xi_{1}$. The stated result is a parametric version of this.

For simplicity of notation we write $(r,s)$ for $(r_1,s_1)$ and define a smooth embedding $\varphi:U\to U\times\R^2$ by
\[\varphi(u)=(u,s(u),-r(u)) \text{ \ for \ }u\in U.\]
Since $r,s$ are compactly supported $\varphi(u)=(u,0,0)$ for all $u\in Op\,(\partial U)$ and there exist positive constants $\vare_1$ and $\vare_2$ such that $Im\,f$ is contained in $U\times I_{\vare_1}\times I_{\vare_2}$, where $I_\vare$ denotes the open interval $(-\vare,\vare)$ for $\vare>0$. Clearly, $\varphi^*(\alpha_0-y\,dx)=\alpha_0+r\,ds$ and so
\begin{equation}\varphi:(U,\xi_{1})\rightarrow (U\times \R^2,\ker(\alpha_{0}- y\,dx)) \label{eq:equidimensional_1}\end{equation}
is an isocontact embedding.
The image of $\varphi$ is the graph of a smooth function $k=(s,-r):U\rightarrow I_{\varepsilon_{1}}\times I_{\varepsilon_{2}}$ which is compactly supported with support contained in the interior of $U$. Further note that $\pi(\varphi(U))$ is the graph of $s$ and hence a submanifold of $U\times I_{\varepsilon_1}$. Now let $\pi:U\times I_{\varepsilon_{1}}\times I_{\varepsilon_{2}}\rightarrow U\times I_{\varepsilon_{1}}$ be the projection onto the first two coordinates. Since Im\,$\varphi$ is the graph of $k$,  $\pi|_{\text{Im\,}\varphi}$ is an embedding onto the set $\pi(\varphi(U))$ which is the graph of $s$. Now observe that Im\,$\varphi$ can also be viewed as the graph of a smooth function, namely $h:\pi(\varphi(U))\rightarrow I_{\varepsilon_{2}}$ defined by $h(u,s(u))=-r(u)$. It is easy to see that $h$ is compactly supported.

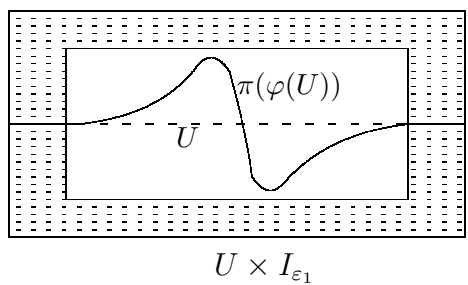
\begin{center}
\begin{picture}(300,140)(-100,5)\setlength{\unitlength}{1cm}
\linethickness{.075mm}

%vertical lines
\multiput(-1,1.5)(6,0){2}
{\line(0,1){3}}
\multiput(-.25,2)(4.5,0){2}
{\line(0,1){2}}

%horizontal lines
\multiput(-1,1.5)(0,3){2}
{\line(1,0){6}}
\multiput(-.25,2)(0,2){2}
{\line(1,0){4.5}}
\put(1.7,1){$U\times I_{\varepsilon_{1}}$}
\put(1.2,2.7){\small{$U$}}
\put(2,3.4){\small{$\pi(\varphi(U))$}}

\multiput(-.9,1.6)(.2,0){30}{\line(1,0){.05}}
\multiput(-.9,1.7)(.2,0){30}{\line(1,0){.05}}
\multiput(-.9,1.8)(.2,0){30}{\line(1,0){.05}}
\multiput(-.9,1.9)(.2,0){30}{\line(1,0){.05}}

\multiput(-.9,4.0)(0,-.1){21}{\line(1,0){.05}}
\multiput(-.7,4.0)(0,-.1){21}{\line(1,0){.05}}
\multiput(-.5,4.0)(0,-.1){21}{\line(1,0){.05}}
\multiput(-.3,4.0)(0,-.1){21}{\line(1,0){.05}}

\multiput(4.3,4.0)(0,-.1){21}{\line(1,0){.05}}
\multiput(4.5,4.0)(0,-.1){21}{\line(1,0){.05}}
\multiput(4.7,4.0)(0,-.1){21}{\line(1,0){.05}}
\multiput(4.9,4.0)(0,-.1){21}{\line(1,0){.05}}

\multiput(-.9,4.1)(.2,0){30}{\line(1,0){.05}}
\multiput(-.9,4.2)(.2,0){30}{\line(1,0){.05}}
\multiput(-.9,4.3)(.2,0){30}{\line(1,0){.05}}
\multiput(-.9,4.4)(.2,0){30}{\line(1,0){.05}}

\multiput(-1,3)(.3,0){20}{\line(1,0){.1}}

%the curve
\multiput(-1,3)(5.1,0){2}{\line(1,0){.9}}
\qbezier(-.1,3)(1,3.1)(1.5,3.8)
\qbezier(1.5,3.8)(1.7,4)(1.9,3.7)
\qbezier(1.9,3.7)(2.1,3)(2.2,2.3)
\qbezier(2.2,2.3)(2.4,2)(2.6,2.2)
\qbezier(2.6,2.2)(3.2,2.9)(4.3,3)

\end{picture}\end{center}
In the above figure, the bigger rectangle represents the set $U\times I_{\varepsilon_{1}}$ and the central dotted line represents $U\times 0$. The curve within the rectangle stands for the domain of $h$, which is also the graph of $s$. We can now extend $h$ to a compactly supported function $H:U\times I_{\varepsilon_{1}}\rightarrow I_{\varepsilon_{2}}$ (see \cite{whitney}) which vanishes on the shaded region and is such that its graph is transversal to $\ker(\alpha_0- y\,dx)$. Indeed, since $\varphi$ is an isocontact embedding it is transversal to $\ker(\alpha_0-y\,dx)$ and hence graph $H$ is transversal to $\ker(\alpha_0-y\,dx)$ on an open neighbourhood of $\pi(\varphi(U))$ for any extension $H$ of $h$. Since transversality is a generic property, we can assume (possibly after a small perturbation) that graph of $H$ is transversal to $\ker(\alpha_0- y\,dx)$.

Let $ \varGamma $ be the graph of $H$; then the image of $\varphi$ is contained in $\varGamma$. By Lemma~\ref{characteristic} there exists a diffeomorphism $\Phi:\Gamma\to U\times I_{\vare_1}$ with the property that
\begin{equation}\Phi^*(\ker(\alpha_{0}\oplus 0))=\ker((\alpha_{0}- y\,dx)|_\varGamma).\label{eq:equidimensional_2}\end{equation}
Next we use $f_0$ to define an immersion $F_{0}:U\times \mathbb R\rightarrow N\times \mathbb{R}$
as follows:
\begin{center} $F_0(u,x)=(f_0(u),x)$ for all $u\in U$ and $x\in \R$.\end{center}
It is straightforward to see that
\begin{itemize}
\item $F_{0}(u,0)\in N \times 0$ for all $u\in U$ and
\item $F_0^*(\eta \oplus 0)$ is a multiple of $\alpha_{0}\oplus 0$ by a nowhere vanishing function on $M\times\R$.
\end{itemize}
Therefore, the following composition is defined:
\[U\stackrel{\varphi}{\longrightarrow}  \Gamma\stackrel{\Phi}{\longrightarrow} U\times I_{\vare_1} \stackrel{F_0}{\longrightarrow} N\times\mathbb{R} \stackrel{\pi_{N}}{\longrightarrow} N, \]
where $\pi_{N}:N\times \mathbb{R}\rightarrow N$ is the projection onto $N$. Observe that $\pi_{N}^*\eta=\eta \oplus 0$ and therefore, it follows from equations (\ref{eq:equidimensional_1}) and (\ref{eq:equidimensional_2}) that the composition map $f_1=\pi_{N} F_{0}\Phi \varphi:(U,\xi_1)\rightarrow (N,\tilde{\xi})$ is isocontact. Such a map is necessarily an immersion.

Let $K=(\text{supp\,}r\cup\text{supp\,}s)$. Take a compact set $K_1$ in $U$ such that $K\subset \text{Int\,}K_1$, and let $\tilde{U}=U\setminus K_1$. If $u\in \tilde{U}$ then $\varphi(u)=(u,0,0)$. This gives $h(u,0)=0$ for all $u\in\tilde{U}$. We can choose $H$ such that $H(u,t)=0$ for all $(u,t)\in\tilde{U}\times I_{\vare_1}$. Then, by Remark~\ref{R:characteristic}, $\Phi(u,0,0)=(u,0)$ for all $u\in\tilde{U}$. Consequently,
\[f_1(u)=\pi_{N} F_{0}\Phi \varphi(u)=\pi_{N} F_{0}(u,0)=\pi_N(f_0(u),0)=f_0(u) \ \text{for all } u\in\tilde{U}.\]
In other words, $f_1$ coincides with $f_0$ outside an open neighbourhood of $K$.

Now observe that if we have a one parameter family of compactly supported functions $(r_t,s_t)$ which depend continuously on the parameter $t$, then  $\varphi$ and $\Phi$ can be made to vary  continuously with respect to the parameter $t$. Thus we get the desired homotopy $f_t$.
This completes the proof of the theorem.
\end{proof}

The above result may be viewed as an extension of Gray's Stability Theorem for open manifolds. We shall now prove the existence of isocontact immersions of an open manifold $M$ into itself which compress the manifold $M$ into an arbitrary small neighbourhoods of its core.
\begin{corollary}
\label{CO}
Let $(M,\xi=\ker\alpha)$ be an open contact manifold and let $K$ be a core of it. Then for a given neighbourhood $U$ of $K$ in $M$ there exists a homotopy of isocontact immersions $f_{t}:(M,\xi)\rightarrow (M,\xi), t\in[0,1]$, such that $f_{0}=id_{M}$ and $f_{1}(M)\subset U$.
\end{corollary}
\begin{proof}Since $K$ is a core of $M$ there is an isotopy $g_t$ such that $g_0=id_M$ and $g_1(M)\subset U$ (see Remark~\ref{core}).
Using $g_t$, we can express $M$ as $M=\bigcup_{0}^{\infty}V_{i}$, where $V_{0}$ is a compact neighbourhood of $K$ in $U$ and $V_{i+1}$ is diffeomorphic to $V_i\bigcup (\partial V_{i}\times [0,1])$ so that $\bar{V_i}\subset \text{Int}\,(V_{i+1})$ and $V_{i+1}$ deformation retracts onto $V_{i}$. If $M$ is a manifold with boundary then this sequence is finite. We shall inductively construct a homotopy of immersions $f^{i}_{t}:M\rightarrow M$ with the following properties:
\begin{enumerate}
\item $f^i_0=id_M$
\item $f^i_1(M)\subset U$
\item $f^i_t=f^{i-1}_t$ on $V_{i-1}$
\item $(f^i_t)^*\xi=\xi$ on $V_{i}$.
\end{enumerate}
Assuming the existence of $f^{i}_{t}$, let $\xi_{t}=(f^{i}_{t})^{*}(\xi)$ (so that $\xi_0=\xi$, and consider a 2-parameter family of contact structures defined by $\eta_{t,s}=\xi_{t(1-s)}$. Then for all $t,s\in\I$, we have:
\[\eta_{t,0}=\xi_t,\ \ \eta_{t,1}=\xi_0=\xi\ \text{ and }\ \eta_{0,s}=\xi.\]
The parametric version of Theorem~\ref{GRAY} gives a homotopy of immersions $\tilde{f}_{t,s}:V_{i+2}\rightarrow M$, $(t,s)\in \I\times\I$, satisfying the following conditions:
\begin{enumerate}
 \item $\tilde{f}_{t,0},\tilde{f}_{0,s}:V_{i+2}\hookrightarrow M$ are the inclusion maps
 \item $(\tilde{f}_{t,s})^*\xi_t=\eta_{t,s}$; in particular, $(\tilde{f}_{t,1})^*\xi_t=\xi$
 \item $\tilde{f}_{t,s}=id$ on $V_i$ since $\eta_{t,s}=\xi_0$ on $V_i$.
\end{enumerate}
We now extend the homotopy $\{\tilde{f}_{t,s}|_{V_{i+1}}\}$ to all of $M$ as immersions such that $\tilde{f}_{0,s}=id_M$ for all $s$. By an abuse of notation, we denote the extended homotopy by the same symbol. Define the next level homotopy as follows:
\[f^{i+1}_{t}=f^{i}_{t}\circ \tilde{f}_{t,1} \ \text{ for }\ t\in [0,1].\]
This completes the induction step since $(f^{i+1}_t)^*(\xi)=(\tilde{f}_{t,1})^*\xi_t=\xi$ on $V_{i+2}$ for all $t$, and $f^{i+1}_t|_{V_{i}}=f^i_t|_{V_{i}}$.
To start the induction we use the isotopy $g_t$ and let $\xi_t=g_t^*\xi$. Note that $\xi_t$ is a family of contact structures on $M$ defined by contact forms $g_t^*\alpha$.
For starting the induction we construct $f^{0}_{t}$ as above by setting $V_{-1}=\emptyset$.

Having constructed the family of homotopies $\{f^i_t\}$ as above we set $f_t=\lim_{i\to\infty}f^i_t$ which is the desired homotopy of isocontact immersions.

\end{proof}

\section{An $h$-principle for open relations on open contact manifolds}

In this section we prove an extension of Theorem~\ref{open-invariant} for some open relations on open contact manifolds. The main result of this section can be stated as follows:
\begin{theorem}
\label{CT}
Let $(M,\alpha)$ be an open contact manifold and $\mathcal{R}\subset J^{r}(M,N)$ be an open relation invariant under the action of the pseudogroup of local contactomorphisms of $(M,\alpha)$. Then parametric $h$-principle holds for $\mathcal{R}$.
\end{theorem}
\begin{proof}
Let $\mathcal{D}$ denote the pseudogroup of contact diffeomorphisms of $M$. We shall first show that $\mathcal{D}$ has the sharply moving property (see Definition~\ref{D:sharp_diffeotopy}). Let $M_0$ be a submanifold of $M$ of positive codimension. Take a closed hypersurface $S$ in $M_0$ and an open set $U\subset M$ containing $S$. We take a vector field $X$ along $S$ which is transversal to $M_{0}$. Let $H:M\rightarrow \mathbb{R}$ be a function such that
\[\alpha(X)=H,\ \ \ \  i_X d\alpha|_\xi=-dH|_\xi, \ \ \text{at points of } S.\]
(see equation~\ref{contact_hamiltonian1}). The contact-Hamiltonian vector field $X_H$ is clearly transversal to $M_0$ at points of $S$. As transversality is a stable property and $U$ is small, we can assume that $X_{H}\pitchfork U$. Now consider the initial value problem
\[\frac{d}{dt}\delta_{t}(x)=X_{H}(\delta_t(x)), \ \ \delta_0(x)=x\]
The solution to this problem exists for small time $t$, say for $t\in [0,\bar{\varepsilon}]$, for all $x$ lying in some small enough neighbourhood of $S$. Moreover, since $X_H$ is transversal to $S$, there would exist a positive real number $\vare$ such that the integral curves $\delta_t(x)$ for $x\in S$ do not meet $M_0$ during the time interval $(0,\vare)$. Let \[S_{\varepsilon}=\cup_{t\in[0,\varepsilon/2]}\delta_{t}(S).\]
Take a smooth function $\varphi$ which is identically equal to 1 on a small neighbourhood of $S_{\varepsilon}$ and supp\,$\varphi\subset \cup_{t\in[0,\varepsilon)}\delta_t(S)$. We then consider the initial value problem with $X_{H}$ replaced by $X_{\varphi H}$. Since $X_{\varphi H}$ is compactly supported the flow of $X_{\varphi H}$, say  $\bar{\delta_{t}}$, is defined for all time $t$. Because of the choice of $\varphi$, the integral curves $\bar{\delta}_t(x_0)$, $x_0\in M_0$, cannot come back to $M_0$ for $t>0$. Hence, we have the following:
\begin{itemize}
 \item $\bar{\delta}_0|_U=id_U$
 \item $\bar{\delta_t}=id$ outside a small neighbourhood of $S_{\varepsilon}$
 \item $dist(\bar{\delta}_1(x),M_0)>r$ for all $x\in S$ and for some $r>0$.
\end{itemize}
This proves that $\mathcal{D}$ sharply moves any submanifold of $M$ of positive codimension.

Since $M$ is open it has a core $K$ which is of positive codimension. Since the relation $\mathcal R$ is open and invariant under the action of $\mathcal D$, we can apply Theorem~\ref{T:gromov-invariant} to conclude that $\mathcal R$ satisfies the parametric $h$-principle near $K$. We now need to lift the $h$-principle from $Op\,K$ to all of $M$.

By the local $h$-principle near $K$, an arbitrary section $F_0$ of $\mathcal R$ admits a homotopy $F_{t}$ in $\Gamma(\mathcal{R}|_U)$ such that $F_1$ is
holonomic on $U$, where  $U$ is an open neighbourhood of $K$ in $M$.
Let $f_{t}=p^{(r)}\circ F_t$, where $p^{(r)}:J^{r}(M,N)\rightarrow N$ is the canonical projection map of the jet bundle. By Corollary~\ref{CO} above we get a homotopy of isocontact immersions $g_{t}:(M,\xi)\rightarrow (M,\xi)$ satisfying $g_{0}=id_{M}$ and $g_{1}(M)\subset U$, where $\xi=\ker\alpha$. The concatenation of the homotopies $g_t^*(F_0)$ and $g_1^*(F_t)$ gives the desired homotopy in $\Gamma(\mathcal R)$ between $F_0$ and the holonomic section $g_1^*(F_1)$. This proves that $\mathcal R$ satisfies the ordinary $h$-principle.

To prove the parametric $h$-principle, take a parametrized section $F_z\in \Gamma(\mathcal{R})$, $z\in \D^n$, such that $F_z$ is holonomic for all $z\in \mathbb S^{n-1}$. This implies that there is a family of smooth maps $f_z\in Sol(\mathcal R)$, parametrized by $z\in \mathbb S^{n-1}$, such that $F_z=j^r_f(z)$. We shall  homotope the parametrized family $F_z$ to a family of holonomic sections in $\mathcal R$ such that the homotopy remains constant on $\mathbb S^{n-1}$. By the parametric $h$-principle near $K$, there exists an open neighbourhood $U$ of $K$ and a homotopy $H:\D^n\times\I\to \Gamma(\mathcal R|_U)$, such that $H^0_z=F_z$ and $H_z^1$ is holonomic for all $z\in \D^n$; furthermore,  $H_z^t=j^r_f(z)$ on $U$ for all $z\in \mathbb S^{n-1}$.

Let $\delta:[0,1/2]\to [0,1]$ be the linear homeomorphism such that $\delta(0)=0$ and $\delta(1/2)=1$. Define a function $\mu$ as follows:
\begin{eqnarray*}\mu(z)   = &  \delta(\|z\|)z/\|z\| & \text{ if }\ \|z\|\leq 1/2.\end{eqnarray*}
 First deform $F_z$ to $\widetilde{F}_z$, where
%, where $\delta_s:[0,1]\to [0,1]$ is a homotopy starting from $id:[0,1]\to [0,1]$ and ending at $\delta$ defined as $\delta(\tau)=0,\ \tau\in[0,1/3]$, $\delta_1=\delta,\ on\ %[1/3,2/3]$ and $\delta_1(\tau)=1,\ \tau\in[2/3,1]$. So
\[\begin{array}{rcl}\widetilde{F}_z & = & \left\{
\begin{array}{ll}
F_{\mu(z)} & \text{if } \|z\|\leq 1/2\\
F_{z/\|z\|} & \text{if } 1/2\leq\|z\|\leq 1\end{array}\right.\end{array}\]
Let $\bar{\delta}:[1/2,1]\to [0,1]$ be the linear homeomorphism such that $\bar{\delta}(1/2)=1$ and $\bar{\delta}(1)=0$.
Define a homotopy $\widetilde{F}^s_z$ of $\tilde{F}_z$ as follows:
\[\begin{array}{rcl}\widetilde{F}_z^s & = & \left\{
\begin{array}{ll}
g_s^*(F_{\mu(z)}), &  \|z\|\leq 1/2\\
g^*_{s\bar{\delta}(\|z\|)}(F_{z/\|z\|}) & 1/2\leq\|z\|\leq 1\end{array}\right.\end{array}\]
Note that
\[\begin{array}{rcl}\widetilde{F}^1_z & = & \left\{
\begin{array}{ll}
g_1^*(F_{\mu(z)}), & \|z\|\leq 1/2\\
g^*_{\bar{\delta}(\|z\|)}(F_{z/\|z\|}) & 1/2\leq\|z\|\leq 1\end{array}\right.\end{array}\]
Finally we consider a parametrized homotopy given as follows:
\[\begin{array}{rcl}\widetilde{H}^s_z & = & \left\{
\begin{array}{ll}
g_1^*(H^s_{\mu(z)}), & \|z\|\leq 1/2\\
g^*_{\bar{\delta}(\|z\|)}(F_{z/\|z\|}) & 1/2\leq\|z\|\leq 1\end{array}\right.\end{array}\]
Note that $\widetilde{H}^1_z$ is holonomic for all $z\in\D^n$ and $\widetilde{H}^s_z=j^r_f(z)$ for all $z\in \mathbb S^{n-1}$. The concatenation of the three homotopies now give a homotopy between the parametrized sections $F_z$ and $\tilde{H}^1_z$ relative to $\mathbb S^{n-1}$. This proves the parametric $h$-principle for $\mathcal R$.
\end{proof}

\section{Gromov-Phillips Theorem on open contact manifolds}
Recall that a leaf $L$ of an arbitrary foliation on $M$ admits an injective immersion $i_L:L\to M$. We shall say that $L$ is a contact submanifold of $(M,\alpha)$ if the pullback form $i_L^*\alpha$ is a contact form on $L$.
\begin{definition} {\em Let $M$ be a smooth manifold with a contact form $\alpha$. A foliation $\mathcal F$ on $M$ will be called a \emph{contact foliation subordinate to} $\alpha$ or, a \emph{contact foliation on} $(M,\alpha)$ if the leaves of $\mathcal F$ are contact submanifolds of $(M,\alpha)$.\label{subordinate_contact_foliation}}
\end{definition}
\begin{remark}{\em In view of Lemma~\ref{L:contact_submanifold}, $\mathcal F$ is a contact foliation on $(M,\alpha)$ if and only if $T\mathcal F$ is transversal to the contact distribution $\ker\alpha$ and $T\mathcal F\cap \ker\alpha$ is a symplectic subbundle of $(\ker\alpha,d'\alpha)$.}\label{R:tangent_contact_foliation}\end{remark}

Let $(M,\alpha)$ be a contact manifold and $N$ a manifold with a smooth foliation $\mathcal F_N$ of even codimension. We denote by $Tr_\alpha(M,\mathcal F_N)$ \index{$Tr_\alpha(M,\mathcal F_N)$} the space of smooth maps $f:M\to N$ transversal to $\mathcal F_N$ for which the inverse foliations $f^*\mathcal F_N$ are contact foliations on $M$ subordinate to $\alpha$.
Let $\mathcal E_\alpha(TM,\nu\mathcal F_N)$ \index{$\mathcal E_\alpha(TM,\nu\mathcal F_N)$} be the space of all vector bundle morphisms $F:TM\rightarrow TN$ such that
\begin{enumerate}
\item $\pi\circ F:TM\to\nu\mathcal F_N$ is an epimorphism and
\item $\ker(\pi\circ F)\cap \ker\alpha$ is a symplectic subbundle of $(\ker\alpha,d'\alpha)$,\end{enumerate}
where $\pi:TN\to \nu\mathcal F_N$ is the quotient map.
We endow $Tr_\alpha(M,\mathcal F_N)$ and $\mathcal{E}_\alpha(TM,\nu\mathcal F_N)$ with $C^{\infty}$ compact open topology and $C^{0}$ compact open topology respectively. The main result of this section can now be stated as follows:
\begin{theorem}\label{T:contact-transverse}
Let $(M,\alpha)$ be an open contact manifold and $(N,\mathcal F_N)$ be any foliated manifold. Suppose that the codimension of $\mathcal F_N$ is even and is strictly less than the  dimension of $M$. Then
\[\pi\circ d:Tr_\alpha(M,\mathcal F_N)\to\mathcal{E}_\alpha(TM,\nu\mathcal F_N)\]
is a weak homotopy equivalence.\end{theorem}

Let $\mathcal R$ denote the first order differential relation consisting of all 1-jets represented by triples $(x,y,G)$, where $x\in M, y\in N$ and $G:T_{x}M\rightarrow T_{y}N$ is a linear map such that
\begin{enumerate}\item $\pi\circ G:T_xM\to \nu(\mathcal F_N)_y$ is an epimorphism
%\item $\ker(\pi\circ G)$ is transversal to $\ker\alpha_x$ and
\item $\ker(\pi\circ G)\cap \ker\alpha_x$ is a symplectic subspace of $(\ker\alpha_x,d'\alpha_x)$. \end{enumerate}
The space of sections of $\mathcal R$ can be identified with $\mathcal E_\alpha(TM,\nu(\mathcal F_N))$ defined above.

\begin{observation} {\em Theorem~\ref{T:contact-transverse} states that the relation $\mathcal R$ satisfies the parametric $h$-principle. Indeed, the solution space of $\mathcal R$ is the same as $Tr_\alpha(M,\mathcal F)$. To see this, it is sufficient to note (see Definition~\ref{contact_submanifold}) that the following two statements are equivalent:
\begin{enumerate}
\item[(S1)] $f:M\to N$ is transversal to $\mathcal F_N$ and the leaves of the inverse foliation $f^*\mathcal F_N$ are contact submanifolds (immersed) of $M$.
\item[(S2)] $\pi\circ df$ is an epimorphism and $\ker (\pi\circ df)\cap \ker\alpha$ is a symplectic subbundle of $(\ker\alpha,d'\alpha)$.
\end{enumerate}}\label{P:solution space}\end{observation}
We will now show that the relation $\mathcal R$ is open and invariant under the action of local contactomorphisms.

\begin{lemma}
\label{OR}
 The relation $\mathcal{R}$ defined above is an open relation.
\end{lemma}
\begin{proof} Let $V$ be a $(2m+1)$-dimensional vector space with a (linear) 1-form $\theta$ and a 2-form $\tau$ on it such that $\theta \wedge \tau^{m}\neq 0$. We shall call $(\theta,\tau)$ an almost contact structure on $V$. Note that the restriction of $\tau$ to $\ker\theta$ is then non-degenerate. A subspace $K$ of $V$ will be called an almost contact subspace if the restrictions of $\theta$ and $\tau$ to $K$ define an almost contact structure on $K$. In this case, $K$ must be transversal to $\ker\theta$ and $K\cap \ker\theta$ will be a symplectic subspace of $\ker\theta$.

Let $W$ be a vector space of even dimension and $Z$ a subspace of $W$ of codimension $2q$. Denote by $L_Z^\pitchfork(V,W)$ the set of all linear maps $L:V\to W$ which are transversal to $Z$. This is clearly an open subset in the space of all linear maps from $V$ to $W$. Define a subset $\mathcal L$ of $L_Z^\pitchfork(V,W)$ by
\[\mathcal L=\{L\in L_Z^\pitchfork(V,W)| \ker(\pi\circ L) \text{ is an almost contact subspace of }V\}\]
We shall prove that $\mathcal L$ is an open subset of $L_Z^\pitchfork(V,W)$.
Consider the map
\[E:L_Z^\pitchfork(V,W)\rightarrow Gr_{2(m-q)+1}(V)\]
\[L\mapsto \ker (\pi\circ L),\]
where $\pi:W\to W/Z$ is the quotient map.
Let $\mathcal U_c$ denote the subset of $G_{2(m-q)+1}(V)$ consisting of all almost contact subspaces $K$ of $V$. Observe that $\mathcal L=E^{-1}(\mathcal U_c)$. We shall now prove that
\begin{itemize}\item $E$ is a continuous map and
\item $\mathcal U_c$ is an open subset of $G_{2(m-q)+1}(V)$.
\end{itemize}
To prove that $E$ is continuous, take $L_0\in L_Z^\pitchfork(V,W)$ and let $K_0=\ker (\pi\circ L_0)$. Consider the subbasic open set $U_{K_0}$ consisting of all subspaces $Y$ of $V$ such that the canonical projection $p:K_0\oplus K_0^\perp\to K_0$ maps $Y$ isomorphically onto $K_0$. The inverse image of $U_{K_0}$ under $E$ consists of all $L:V\to W$ such that $p|_{\ker (\pi\circ L)}:\ker (\pi\circ L)\to K_0$ is onto.
It may be seen easily that if $L\in L_Z^\pitchfork(V,W)$ then
\begin{eqnarray*}  p \text{ maps } \ker (\pi\circ L) \text{ onto }K_0 &
  \Leftrightarrow &  \ker (\pi\circ L)\cap K_0^\perp=\{0\} \\
 &  \Leftrightarrow &  \pi\circ L|_{K_0^\perp}:K_0^\perp\to W/Z \text{ is an isomorphism}.\end{eqnarray*}
Now, the set of all $L$ such that $\pi\circ L|_{K_0^\perp}$ is an isomorphism is an open subset. Hence $E^{-1}(U_{K_0})$ is open and therefore $E$ is continuous.

To prove the openness of $\mathcal U_c$ take $K_0\in\mathcal U$. Recall that a subbasic open set $U_{K_0}$ containing $K_0$ can be identified with the space $L(K_0,K_0^\perp)$,  where $K_0^\perp$ denotes the orthogonal complement of $K$ with respect to some inner product on $V$ (\cite{milnor_stasheff}).
Let $\Theta$ denote the following composition of continuous maps:
\[\begin{array}{rcccl}U_{K_0}\cong L(K_0,K_0^{\perp}) & \stackrel{\Phi}{\longrightarrow} &  L(K_0,V)& \stackrel{\Psi}{\longrightarrow} & \Lambda^{2(m-q)+1}(K_0^*)\cong\R\end{array}\]
where $\Phi(L)=I+L$ and  $\Psi(L)=L^*(\theta\wedge\tau^{2(m-q)+1})$.
It may be noted that, if $K\in U_{K_0}$ is mapped onto some $T\in L(K_0,V)$ then the image of $T$ is $K$. Hence it follows that
\[{\mathcal U}_c\cap U_{K_0}=(\Psi\circ\Phi)^{-1}(\R\setminus 0)\]
which proves that ${\mathcal U}_c\cap U_{K_0}$ is open. Since $U_{K_0}$ is a subbasic open set in the topology of Grassmannian it proves the openness of $\mathcal U_c$.
Thus $\mathcal L$ is an open subset.

We now show that $\mathcal R$ is an open relation. First note that, each tangent space $T_xM$ has an almost contact structure given by $(\alpha_x,d\alpha_x)$.
Let $U$ be a trivializing neighbourhood of the tangent bundle $TM$. We can choose a trivializing neighbourhood $\tilde{U}$ for the tangent bundle $TN$ such that $T\mathcal F_N$ is isomorphic to $\tilde{U}\times Z$ for some codimension $2q$-vector space in $\R^{2n}$. This implies that $\mathcal R\cap J^1(U,\tilde{U})$ is diffeomorphic with  $U\times\tilde{U}\times\mathcal L$. Since the sets $J^1(U,\tilde{U})$ form a basis for the topology of the jet space, this completes the proof of the lemma.
\end{proof}

\begin{lemma}
\label{IV}
$\mathcal{R}$ is invariant under the action of the pseudogroup of local contactomorphisms of $(M,\alpha)$.
\end{lemma}
\begin{proof}
Let $\delta$ be a local diffeomorphism on an open neighbourhood of $x\in M$ such that $\delta^*\alpha=\lambda\alpha$, where $\lambda$ is a nowhere vanishing function on $Op\, x$. This implies that $d\delta_x(\xi_x)=\xi_{\delta(x)}$ and $d\delta_x$ preserves the conformal symplectic structure determined by $d\alpha$ on $\ker \xi$. If $f$ is a local solution of $\mathcal R$ at $\delta(x)$, then
\[d\delta_x(\ker d(f\circ\delta)_x\cap \xi_x)=\ker df_{\delta(x)}\cap\xi_{\delta(x)}.\]
Hence $f\circ\delta$ is also a local solution of $\mathcal R$ at $x$.
Since $\mathcal R$ is open every representative function of a jet in $\mathcal R$ is a local solution of $\mathcal R$. Thus
local contactomorphisms act on $\mathcal R$ by $\delta.j^1_f(\delta(x)) = j^1_{f\circ\delta}(x)$.
\end{proof}

\emph{Proof of Theorem~\ref{T:contact-transverse}}:
In view of Theorem~\ref{CT}, and Lemma~\ref{OR}, \ref{IV} it follows that the relation $\mathcal R$ satisfies the parametric $h$-principle. This completes the proof by Observation ~\ref{P:solution space}.\qed\\
\begin{definition} \emph{A smooth submersion $f:(M,\alpha)\to N$ is called a \emph{contact submersion} if the level sets of $f$ are contact submanifolds of $M$.}
\end{definition}
We shall denote the space of contact submersion $(M,\alpha)\to N$ by $\mathcal C_\alpha(M,N)$.
The space of epimorphisms $F:TM\to TN$ for which $\ker F\cap \ker\alpha$ is a symplectic subbundle of $(\ker\alpha,d'\alpha)$ will be denoted by $\mathcal E_\alpha(TM,TN)$.
If $\mathcal F_N$ in Theorem~\ref{T:contact-transverse} is the zero-dimensional foliation then we get the following result.
\begin{corollary} Let $(M,\alpha)$ be an open contact manifold. The derivative map
 \[d:\mathcal C_\alpha(M,N)\to \mathcal E_\alpha(TM,TN)\]
is a weak homotopy equivalence.\label{T:contact_submersion}
\end{corollary}
\begin{remark}{\em Suppose that $F_0\in \mathcal E_\alpha(TM,TN)$ and $D$ is the kernel of $F_0$. Then $(D,\alpha|_D,d\alpha|_D)$ is an almost contact distribution. Since $M$ is an open manifold, the bundle epimorphism $F_0:TM\to TN$ can be homotoped (in the space of bundle epimorphism) to the derivative of a submersion $f:M\to N$ (\cite{phillips}). Hence the distribution $\ker F_0$ is homotopic to an integrable distribution, namely the one given by the submersion $f$. It then follows from Theorem~\ref{T:contact} that $(D,\alpha|_D,d\alpha|_D)$ is homotopic to the distribution associated to a contact foliation $\mathcal F$ on $M$. Theorem~\ref{T:contact-transverse} further implies that it is possible to get a foliation $\mathcal F$ which is subordinate to $\alpha$ and is defined by a submersion.}\end{remark}

\section{Contact Submersions into Euclidean spaces}
In this section we interpret the homotopy classification of contact submersions of $M$ into $\R^{2n}$ in terms of certain $2n$-frames in $M$. We then apply this result to show the existence of contact foliations on some subsets of odd-dimensional $N$-spheres obtained by deleting lower dimensional spheres. Throughout this section $M$ is a contact manifold with a contact form $\alpha$ and $\xi$ is the contact distribution $\ker\alpha$.

Recall from Section 2 that the tangent bundle $TM$ of a contact manifold $(M,\alpha)$ splits as $\ker\alpha\oplus\ker \,d\alpha$. Let $P:TM\to\ker\alpha$ be the projection morphism onto $\ker\alpha$ relative to this splitting. We shall denote the projection of a vector field $X$ on $M$ under $P$  by $\bar{X}$. For any smooth function $h:M\to \R$, $X_h$ will denote the contact Hamiltonian vector field defined as in the prelimiaries (see equations (\ref{contact_hamiltonian1})).
\begin{lemma} Let $(M,\alpha)$ be a contact manifold and $f:M\to \R^{2n}$ be a submersion with coordinate functions $f_1,f_2,\dots,f_{2n}$. Then the following statements are equivalent:
\begin{enumerate}\item[(C1)] $f$ is a contact submersion.
\item[(C2)] The restriction of $d\alpha$ to the bundle spanned by $X_{f_1},\dots,X_{f_{2n}}$ defines a symplectic structure.
\item[(C3)] The vector fields $\bar{X}_{f_1},\dots,\bar{X}_{f_{2n}}$ span a symplectic subbundle of $(\xi,d'\alpha)$.
\end{enumerate}\end{lemma}
\begin{proof} If $f:(M,\alpha)\to\R^{2n}$ is a contact submersion then the following relation holds pointwise:
\begin{equation}\ker df\cap \ker\alpha=\langle \bar{X}_{f_1},...,\bar{X}_{f_{2n}}\rangle^{\perp_{d'\alpha}},\end{equation}
where the right hand side represents the symplectic complement of the subbundle spanned by $\bar{X}_{f_1},...,\bar{X}_{f_{2n}}$ with respect to $d'\alpha$. Indeed, for any $v\in \ker\alpha$,
\[  d'\alpha(\bar{X}_{f_i},v)=-df_i(v),\ \ \text{ for all }i=1,...,2n \]
Therefore, $v\in\ker\alpha\cap\ker df$ if and only if $d'\alpha(\bar{X}_{f_i},v)=0$ for all $i=1,\dots,2n$, that is $v\in \langle \bar{X}_{f_1},...,\bar{X}_{f_{2n}}\rangle^{\perp_{d'\alpha}}$. Thus, the equivalence of (C1) and (C3) is a consequence of the equivalence between (S1) and (S2). The equivalence of (C2) and (C3) follows from the relation  $d\alpha(X,Y)=d\alpha(\bar{X},\bar{Y})$, where $X,Y$ are any two vector fields on $M$. \end{proof}

An ordered set of vectors $e_{1}(x),...,e_{2n}(x)$ in $\xi_x$ will be called a \emph{symplectic $2n$-frame} \index{symplectic $2n$-frame} in $\xi_x$ if the subspace spanned by these vectors is a symplectic subspace of $\xi_x$ with respect to the symplectic form $d'\alpha_x$. Let $T_{2n}\xi$ be the bundle of symplectic $2n$-frames in $\xi$ and $\Gamma(T_{2n}\xi)$ denote the space of sections of $T_{2n}\xi$ with the $C^{0}$ compact open topology.

For any smooth submersion $f:(M,\alpha)\rightarrow \mathbb{R}^{2n}$, define the \emph{contact gradient} of $f$ by
\[\Xi f(x)=(\bar{X}_{f_{1}}(x),...,\bar{X}_{f_{2n}}(x)),\]
where $f_{i}$, $i=1,2,\dots,2n$, are the coordinate functions of $f$. If $f$ is a contact submersion then $\bar{X}_{f_{1}}(x),...,\bar{X}_{f_{2n}}(x))$ span a symplectic subspace of $\xi_x$ for all $x\in M$, and hence $\Xi f$ becomes a section of $T_{2n}\xi$.

\begin{theorem}
\label{ED}
Let $(M^{2m+1},\alpha)$ be an open contact manifold. Then the contact gradient map $\Xi:\mathcal{C}_\alpha(M,\mathbb{R}^{2n})\rightarrow \Gamma(T_{2n}\xi)$ is a weak homotopy equivalence.
\end{theorem}
\begin{proof} As $T\mathbb{R}^{2n}$ is a trivial vector bundle, the map
\[i_{*}:\mathcal{E}_\alpha(TM,\mathbb{R}^{2n})\rightarrow \mathcal{E}_\alpha(TM,T\mathbb{R}^{2n})\]
induced by the inclusion $i:0 \hookrightarrow \mathbb{R}^{2n}$ is a homotopy equivalence, where $\mathbb{R}^{2n}$ is regarded as the vector bundle over $0\in \mathbb{R}^{2n}$. The homotopy inverse $c$ is given by the following diagram. For any $F\in \mathcal E_\alpha(TM,T\R^{2n})$, $c(F)$ is defined by as $p_2\circ F$,
\[\begin{array}{ccccc}
TM & \stackrel{F}{\longrightarrow} & T\mathbb{R}^{2n}=\mathbb{R}^{2n}\times \mathbb{R}^{2n} & \stackrel{p_2}{\longrightarrow} & \mathbb{R}^{2n}\\
\downarrow & & \downarrow & & \downarrow\\
M & \longrightarrow & \mathbb{R}^{2n} & \longrightarrow & 0
\end{array}\]
where $p_2$ is the projection map onto the second factor.

Since $d'\alpha$ is non-degenerate, the contraction of $d'\alpha$ with a vector $X\in\ker\alpha$ defines an isomorphism
\[\phi:\ker\alpha \rightarrow (\ker\alpha)^*.\]
We define a map $\sigma:\oplus_{i=1}^{2n}T^*M\to \oplus_{i=1}^{2n}\xi$ by
\[\sigma(G_1,\dots,G_{2n})=-(\phi^{-1}(\bar{G}_1),...,\phi^{-1}(\bar{G}_{2n})),\]
where $\bar{G}_i=G_i|_{\ker\alpha}$. Then noting that
\[\ker(G_1,\dots,G_{2n})\cap \ker\alpha=\langle\phi^{-1}(\bar{G}_1),\dots,\phi^{-1}(\bar{G}_{2n})\rangle^{\perp_{d'\alpha}},\]
we get a map $\tilde{\sigma}$ by restricting $\sigma$ to $\mathcal E(TM,\R^{2n})$:
\[\tilde{\sigma}:{\mathcal E}(TM,\mathbb{R}^{2n})\longrightarrow \Gamma(M,T_{2n}\xi),\]
Moreover, the contact gradient map $\Xi$ factors as $\Xi= \tilde{\sigma} \circ c \circ d$:
\begin{equation}\mathcal{C}_\alpha(M,\mathbb{R}^{2n})\stackrel{d}\rightarrow \mathcal{E}_\alpha(TM,T\mathbb{R}^{2n})\stackrel{c}\rightarrow \mathcal{E}_\alpha(TM,\mathbb{R}^{2n})\stackrel{\tilde{\sigma}}\rightarrow \Gamma(T_{2n}\xi).\end{equation}
To see this take any $f:M\to \R^{2n}$. Then, $c(df)=(df_{1},...,df_{2n})$, and hence
\[ \tilde{\sigma} c (df)=(\phi^{-1}(df_1|_\xi),...,\phi^{-1}(df_{2n}|_\xi)) = (\bar{X}_{f_1},\dots,\bar{X}_{f_{2n}})=\Xi(f)\]
which gives $\tilde{\sigma} \circ c \circ d(f)=\Xi f$.

We claim that $\tilde{\sigma}: \mathcal{E}_\alpha(TM,\mathbb{R}^{2n})\to \Gamma(T_{2n}\xi)$ is a homotopy equivalence.
To prove this we define a map $\tau: \oplus_{i=1}^{2n}\xi \to  \oplus_{i=1}^{2n} T^*M$
by the formula \[\tau(X_1,\dots,X_{2n})=(i_{X_1}d\alpha,...,i_{X_{2n}} d\alpha)\]
which induces a map $\tilde{\tau}: \Gamma(T_{2n}\xi) \to  \mathcal{E}(TM,\mathbb{R}^{2n})$.
It is easy to verify that $\tilde{\sigma} \circ \tilde{\tau}=id$. In order to show that $\tilde{\tau}\circ\tilde{\sigma}$ is homotopic to the identity, take any $G\in  \mathcal E_\alpha(TM,\R^{2n})$ and let $\widehat{G}=(\tau\circ \sigma)(G)$. Then $\widehat{G}$ equals $G$ on $\ker\alpha$. Define a homotopy between $G$ and $\hat{G}$ by $G_t=(1-t)G+t\widehat{G}$. Then $G_t=G$ on $\ker\alpha$ and hence $\ker G_t\cap \ker\alpha=\ker G\cap \ker\alpha$. This also implies that each $G_t$ is an epimorphism. Thus, the homotopy $G_t$ lies in $\mathcal E_\alpha(TM,\R^{2n})$. This shows that $\tilde{\tau}\circ \tilde{\sigma}$ is homotopic to the identity map.

This completes the proof of the theorem since $d:\mathcal{C}(M,\mathbb{R}^{2n}) \rightarrow \mathcal{E}(TM,T\mathbb{R}^{2n})$ is a weak homotopy equivalence (Theorem~\ref{T:contact-transverse}) and $c$, $\tilde{\sigma}$ are homotopy equivalences.\end{proof}

\begin{example}
{\em Let $\mathbb{S}^{2N-1}$ denote the $2N-1$ sphere in $\R^{2N}$
\[\mathbb{S}^{2N-1}=\{(z_{1},...,z_{2N})\in \mathbb{R}^{2N}: \Sigma_{1}^{2N}|z_{i}|^{2}=1\}\]
This is a standard example of a contact manifold where the contact form $\eta$ is induced from the 1-form $\sum_{i=1}^{N} (x_i\,dy_i-y_i\,dx_i)$ on $\R^{2N}$. For $N>K$, we consider the open manifold $\mathcal S_{N,K}$ obtained from $\mathbb{S}^{2N-1}$ by deleting a $(2K-1)$-sphere:
\begin{center}$\mathcal{S}_{N,K}=\mathbb S^{2N-1}\setminus \mathbb{S}^{2K-1}$,\end{center} where
\[\mathbb{S}^{2K-1}=\{(z_{1},...,z_{2K},0,...,0)\in \mathbb{R}^{2N}: \Sigma_{1}^{2K}|z_{i}|^{2}=1\}\]
Then $\mathcal{S}_{N,K}$ is an contact submanifold of $\mathbb S^{2N-1}$. Let $\xi$ denote the contact structure associated to the contact form $\eta$ on $\mathcal S_{N,K}$. Since $\xi\to \mathcal S_{N,K}$ is a symplectic vector bundle, we can choose a complex structure $J$ on $\xi$ such that $d'\eta$ is $J$-invariant. Thus, $(\xi,J)$ becomes a complex vector bundle of rank $N-1$.

We define a homotopy $F_t:\mathcal S_{N,K}\to \mathcal S_{N,K}$, $t\in [0,1]$, as follows: For $(x,y)\in \mathbb{R}^{2k}\times \mathbb{R}^{2(N-k)}\cap \mathcal{S}_{N,K}$
\[F_t(x,y)=\frac{(1-t)(x,y)+t(0,y/\|y \|)}{\|(1-t)(x,y)+t(0,y/\| y \|) \|}\]
This is well defined since $y\neq 0$. It is easy to see that $F_0=id$, $F_1$ maps $\mathbb{S}^{2(N-K)-1}$ into $\mathcal S_{N,K}$ and  the homotopy fixes $\mathbb{S}^{2(N-K)-1}$ pointwise. Define $r:\mathcal S_{N,K}\rightarrow \{0\}\times \mathbb{R}^{2(N-k)}\cap \mathcal S_{N,K}\mathbb{S}^{2(N-K)-1}\simeq \mathbb{S}^{2(N-K)-1}$ by
\[r(x,y)= (0,y/\|y\|), \ \ \ (x,y)\in\R^{2K}\times\R^{2(N-K)}\cap \mathcal{S}_{N,K}\]
Then $F_1$ factors as $F_1=i\circ r$, where $i$ is the inclusion map, and we have the following diagram:
\[
\begin{array}{lcccl}
  r^*(i^*\xi)&\longrightarrow&i^*\xi&\longrightarrow&\xi\\
\downarrow&&\downarrow&&\downarrow\\
\mathcal{S}_{N,K}&\stackrel{r}{\longrightarrow}&\mathbb{S}^{2(N-K)-1}&\stackrel{i}{\longrightarrow}&\mathcal{S}_{N,K}\end{array}\]
Hence, $\xi=F_0^*\xi\cong F_1^*\xi=r^*(\xi|_{S^{(2N-2K)-1}})$ as complex vector bundles.
Since $\xi$ is a (complex) vector bundle of rank $N-1$, $\xi|_{\mathbb S^{2(N-K)-1}}$ will have a decomposition of the following form (\cite{husemoller}):
\[\xi|_{S^{(2N-2K)-1}}\cong \tau^{N-K-1}\oplus \theta^K,\]
where $\theta^K$ is a trivial complex vector bundle of rank $K$ and $\tau^{N-K-1}$ is a complementary subbundle. Hence $\xi$ must also have a trivial direct summand $\theta$ of rank $K$. Moreover, $\theta$ will be a symplectic subbundle of $\xi$ since the complex structure $J$ is compatible with the symplectic structure $d'\eta$ on $\xi$. Thus, $S_{N,K}$ admits a symplectic $2K$ frame spanning $\theta$. Hence, by Theorem~\ref{ED}, there exist contact submersions of $\mathcal S_{N,K}$ into $\R^{2K}$. Consequently, $\mathcal S_{N,K}$ admits contact foliations of codimension $2K$ for each $K<N$.
}\end{example}

\section{Classification of contact foliations on contact manifolds}

Throughout this section $M$ is a contact manifold with a contact form $\alpha$. As before $\xi$ will denote the associated contact structure $\ker\alpha$ and $d'\alpha=d\alpha|_{\xi}$. Let $Fol_\alpha^{2q}(M)$ denote the space of contact foliations on $M$ of codimension $2q$ subordinate to $\alpha$ (Definition~\ref{subordinate_contact_foliation}).
Recall the classifying space $B\Gamma_{2q}$ and the universal $\Gamma_{2q}$ structure $\Omega_{2q}$ on it (see Subsection~\ref{classifying space}).
Let $\mathcal E_{\alpha}(TM,\nu\Omega_{2q})$ be the space of all vector bundle epimorphisms $F:TM\to \nu \Omega_{2q}$ such that $\ker F$ is transversal to $\ker\alpha$ and $\ker\alpha\cap \ker F$ is a symplectic subbundle of $(\ker\alpha,d'\alpha)$. 

If $\mathcal F\in Fol^{2q}(M)$ and $f:M\to B\Gamma_{2q}$ is a classifying map of $\mathcal F$, then  $f^*\Omega_{2q}= \mathcal F$ as $\Gamma_{2q}$-structure. Recall that we can define a vector bundle epimorphisms $TM\to \nu\Omega_{2q}$ by the following diagram (see \cite{haefliger1})
\begin{equation}
 \xymatrix@=2pc@R=2pc{
TM \ar@{->}[r]^-{\pi_M}\ar@{->}[rd] & \nu \mathcal{F}\cong f^*(\nu \Omega_{2q}) \ar@{->}[r]^-{\bar{f}}\ar@{->}[d] & \nu \Omega_{2q} \ar@{->}[d]\\
& M \ar@{->}[r]_-{f} & B\Gamma_{2q}
}\label{F:H(foliation)}
\end{equation}
where $\pi_M:TM\to \nu(\mathcal F)$ is the quotient map and $(\bar{f},f)$ is a pull-back diagram. Note that the kernel of this morphism is $T\mathcal F$ and therefore, 
if $\mathcal F\in Fol^{2q}_\alpha(M)$, then  $\bar{f}\circ \pi_M \in \mathcal E_\alpha(TM,\nu\Omega_{2q})$ (see Remark~\ref{R:tangent_contact_foliation}).
However, the morphism $\bar{f}\circ \pi_M$ is defined uniquely only up to homotopy. Thus, there is a function
\[H'_\alpha:Fol^{2q}_\alpha(M)\to \pi_0(\mathcal E_\alpha(TM,\nu\Omega_{2q})).\]
\begin{definition} {\em Two contact foliations $\mathcal F_0$ and $\mathcal F_1$ on $(M,\alpha)$ are said to be \emph{integrably homotopic relative to $\alpha$} if there exists a foliation $\tilde{\mathcal F}$ on $(M\times\I,\alpha\oplus 0)$ such that the following conditions are satisfied:
\begin{enumerate}
\item $\tilde{\mathcal F}$ is transversal to the trivial foliation of $M\times\I$ by the leaves $M\times\{t\}$, $t\in \I$; 
\item the foliation $\mathcal F_t$ on $M$ induced by the canonical injective map $i_t:M\to M\times\I$ (given by $x\mapsto (x,t)$) is a contact foliation subordinate to $\alpha$ for each $t\in\I$;
\item the induced foliations on  $M\times\{0\}$ and $M\times\{1\}$ coincide with $\mathcal F_0$ and $\mathcal F_1$ respectively,\end{enumerate}
where $\alpha\oplus 0$ denotes the pull-back of $\alpha$ by the projection map $p_1:M\times\R\to M$.}\end{definition}

Let $\pi_0(Fol^{2q}_{\alpha}(M))$ denote the space of integrable homotopy classes of contact foliations on $(M,\alpha)$. Define 
\[H_\alpha:\pi_0(Fol^{2q}_\alpha(M))\to \pi_0(\mathcal E_\alpha(TM,\nu\Omega_{2q})).\]
by $H_{\alpha}([\mathcal{F}])=H_\alpha'(\mathcal F)$, where $[\mathcal F]$ denotes the integrable homotopy class of $\mathcal F$ relative to $\alpha$.
To see that $H_\alpha$ is well-defined, let $\tilde{\mathcal F}$ be an integrable homotopy relative to $\alpha$ between two contact foliations $\mathcal F_0$ and $\mathcal F_1$. Then the induced foliations $\mathcal F_t$ are contact foliations subordinate to $\alpha$. If $F:M\times\I\to B\Gamma_{2q}$ is a classifying map of $\widetilde{\mathcal F}$ then $F^*\Omega_{2q}=\widetilde{\mathcal F}$. Let $f_t:M\to B\Gamma_{2q}$ be defined by $f_t(x)=F(x,t)$, for all $x\in M$, $t\in \I$. Then it follows that $f_t^*\Omega_{2q}=\mathcal F_t$, for all $t\in \I$. Hence, $H'_\alpha(\mathcal F_0)=H'_\alpha(\mathcal F_1)$. This shows that $H_\alpha$ is well-defined. The classification of contact foliations may now be stated as follows:
\begin{theorem}
\label{haefliger_contact}
If $M$ is open then $H_\alpha:\pi_0(Fol^{2q}_{\alpha}(M)) \longrightarrow \pi_0(\mathcal E_{\alpha}(TM,\nu\Omega_{2q}))$ is bijective.
\end{theorem}
We first prove a lemma.
\begin{lemma}Let $N$ be a smooth manifold with a foliation $\mathcal F_N$ of codimension $2q$. If $g:N\to B\Gamma_{2q}$ classifies $\mathcal F_N$ then we have a commutative diagram as follows:
\begin{equation}
 \xymatrix@=2pc@R=2pc{
\pi_0(Tr_{\alpha}(M,\mathcal{F}_N))\ar@{->}[r]^-{P}\ar@{->}[d]_-{\cong}^-{\pi_0(\pi \circ d)} & \pi_0(Fol^{2q}_{\alpha}(M))\ar@{->}[d]^-{H_{\alpha}}\\
\pi_0(\mathcal E_{\alpha}(TM,\nu \mathcal{F}_N))\ar@{->}[r]_{G_*} & \pi_0(\mathcal E_{\alpha}(TM,\nu\Omega_{2q}))
}\label{Figure:Haefliger}
\end{equation}
where the left vertical arrow is the isomorphism defined by Theorem~\ref{T:contact-transverse}, $P$ is induced by a map which takes an $f\in Tr_\alpha(M,\mathcal F_N)$ onto the inverse foliation $f^*\mathcal F_N$ and $G_*$ is induced by the bundle homomorphism $G:\nu\mathcal F_N\to \nu\Omega_{2q}$ covering $g$.\label{L:haefliger}
\end{lemma}
\begin{proof} We shall first show that the horizontal arrows in (\ref{Figure:Haefliger}) are well defined.
If $f\in Tr_{\alpha}(M,\mathcal{F}_N)$ then the inverse foliation $f^*\mathcal F_N$ belongs to $Fol^{2q}_\alpha(M)$. Furthermore, if $f_t$ is a homotopy in $Tr_{\alpha}(M,\mathcal{F}_N)$, then the map $F:M\times\I\to N$ defined by $F(x,t)=f_t(x)$ is clearly transversal to $\mathcal F_N$ and so $\tilde{\mathcal F}=F^*\mathcal F_N$ is a foliation on $M\times\I$.
The restriction of $\tilde{\mathcal F}$ to $M\times\{t\}$ is the same as the foliation $f^*_t(\mathcal F_N)$, which is a contact foliation subordinate to $\alpha$. Hence, we get a map \[\pi_0(Tr_{\alpha}(M,\mathcal{F}_N))\stackrel{P}\longrightarrow \pi_0(Fol^{2q}_{\alpha}(M))\] defined by \[[f]\longmapsto [f^*\mathcal{F}_N]\]

On the other hand, since $g:N\to B\Gamma_{2q}$ classifies the foliation $\mathcal F_N$, there is a vector bundle homomorphism $G:\nu\mathcal F_N\to \nu\Omega_{2q}$ covering $g$. This induces a map
\[G_*: \mathcal E_\alpha(TM,\nu(\mathcal F_N))\to  \mathcal E_\alpha(TM,\nu\Omega_{2q})\] which takes an element $F\in \mathcal E_\alpha(TM,\nu(\mathcal F_N))$ onto $G\circ F$.
We now prove the commutativity of (\ref{Figure:Haefliger}). Note that if $f\in Tr_{\alpha}(M,\mathcal{F}_N))$ then $g\circ f:M\to B\Gamma_{2q}$ classifies the foliation $f^*\mathcal F_N$. Let $\widetilde{df}:\nu(f^*\mathcal F_N)\to \nu(\mathcal F_N)$ be the unique map making the following diagram commutative:
 \[
 \xymatrix@=2pc@R=2pc{
TM\ar@{->}[r]^-{df}\ar@{->}[d]_-{\pi_M} & TN\ar@{->}[d]^-{\pi_N}\\
\nu (f^*\mathcal{F}_N)\ar@{->}[r]_{\widetilde{df}} & \nu(\mathcal F_N)
}
\]
where $\pi_M:TM\to\nu(f^*\mathcal F_N)$ is the quotient map onto the normal bundle of $f^*\mathcal F_N$. 
Observe that $G\circ\widetilde{df}:\nu(f^*\mathcal F_N)\to \nu(\Omega_{2q})$ covers the map $g\circ f$ and $(G\circ \widetilde{df},g\circ f)$ is a pullback diagram. Therefore, we have
\[H_\alpha([f^*\mathcal F_N])=[(G\circ\widetilde{df})\circ \pi_M]=[G\circ(\pi\circ df)].\]
This proves the commutativity of (\ref{Figure:Haefliger}).\end{proof}

{\em Proof of Theorem ~\ref{haefliger_contact}}. The proof is exactly similar to that of Haefliger's classification theorem. We can reduce the classification to Theorem~\ref{T:contact-transverse} by using Theorem~\ref{HL} and Lemma~\ref{L:haefliger}. For the sake of completeness we reproduce the proof here following \cite{francis}.
For simplicity of notation we shall denote the universal $\Gamma_{2q}$ structure by $\Omega$ in place of $\Omega_{2q}$. To prove surjectivity of $H_\alpha$, take $(\hat{f},f)\in \mathcal E_{\alpha}(TM,\nu\Omega)$ which can be factored as follows:
\begin{equation}
 \xymatrix@=2pc@R=2pc{
TM \ar@{->}[r]^-{\bar{f}}\ar@{->}[rd] &  f^*(\nu \Omega)  \ar@{->}[r]\ar@{->}[d] & \nu \Omega  \ar@{->}[d]\\
& M \ar@{->}[r]_-{f} & B\Gamma_{2q}
}\label{Figure:Haefliger1}
\end{equation}
By Theorem~\ref{HL} there exists a manifold $N$ with a codimension-$2q$ foliation $\mathcal{F}_N$ and a closed embedding $M\stackrel{s}\hookrightarrow N$ such that $s^*\mathcal{F}_N=f^*\Omega$. Let $f':N\rightarrow B\Gamma_{2q}$ be a map classifying $\mathcal{F}_N$, i.e. $f'^*\Omega\cong\mathcal{F}_N$. Hence $(f'\circ s)^*\Omega\cong f^*\Omega$ and $(f'\circ s)^*\nu(\Omega)\cong f^*\nu(\Omega)$. Therefore $f'\circ s$ must also be covered by a bundle epimorphism
which splits as in the following diagram:
\begin{equation}
\xymatrix@=2pc@R=2pc{
TM \ar@{->}[r]^-{\bar{f}}\ar@{->}[rd] &
f^*(\nu\Omega)\ar@{->}[r]^-{}\ar@{->}[d] & \nu \mathcal{F}_N\ar@{->}[r]\ar@{->}[d] & \nu \Omega \ar@{->}[d]\\
& M\ar@{->}[r]_-{s} & N\ar@{->}[r]_-{f'} & B\Gamma_{2q}
}\label{Figure:Haefliger2}\end{equation}
Let $\hat{s}:TM\stackrel{\bar{f}}{\rightarrow} f^*(\nu\Omega)\cong s^*(\nu\mathcal{F}_N)\rightarrow \nu \mathcal{F}_N$. It is not difficult to see that $(\hat{s},s)$ is an element of $\mathcal E_\alpha(TM,\nu(\mathcal F_N)$. Lastly we show that $\bar{P}(\hat{s},s)$ is homotopic to $(\hat{f},f)$. Since  $f^*\Omega\cong (f'\circ s)^*\Omega$, by Theorem~\ref{CMT} there exists a homotopy
\[M\times \I\stackrel{G}\longrightarrow B\Gamma_{2q}\]
starting at $f'\circ s$ and ending at $f$. As $s$ is a cofibration the following diagram can be solved for some $F$ so that $F(\ ,0)=f'$ and $F(s(x),1)=f(x)$ for all $x\in M$.
\begin{equation}
 \xymatrix@=2pc@R=2pc{
 M\times \{0\}\ar@{->}[rr]^-{i_M}\ar@{->}[dd]_-{s\times id_0} & &  M\times \I \ar@{->}[dl]^-{G}\ar@{->}[dd]^-{s\times id_{\I}}\\
 & B\Gamma_{2q} & \\
 N\times \{0\}\ar@{->}[ru]^-{f'}\ar@{->}[rr]_{i_N} & & N\times \I \ar@{-->}[ul]^-{F}
 }\label{Figure:Haefliger3}
\end{equation}
If we set $f_t'(x)=F(x,t)$ for $x\in N$ and $t\in [0,1]$ then $f$ factors as $f_1'\circ s$. Since $f_t$ is a homotopy $f_t'^*\nu(\Omega)\cong f'^*\nu(\Omega) \cong\nu(\mathcal F_N)$. Thus we get the following homotopy of vector bundle morphism.
\begin{equation}
 \xymatrix@=2pc@R=2pc{
 TM\ar@{->}[r]^{\hat{s}} \ar@{->}[d] & \nu(\mathcal F_N) \ar@{->}[r]^-{a_t}\ar@{->} [d]  & \nu \Omega \ar@{->}[d]\\
 M\ar@{->}[r]_-{s} &   N\ar@{->}[r]_-{f'_t} & B\Gamma_{2q}
 }\label{Figure:Haefliger4}
\end{equation}
This homotopy starts at the morphism shown in diagram (\ref{Figure:Haefliger2}) and ends at the morphism shown at diagram (\ref{Figure:Haefliger1}). Now the left square of diagram (\ref{Figure:Haefliger2}) represents an element $(\hat{s},s)$ of $\mathcal E_{\alpha}(TM,\nu \mathcal{F}_N)$ whose homotopy class is mapped to $[(\hat{f},f)]$ by the bottom map of diagram (\ref{Figure:Haefliger}). So in diagram (\ref{Figure:Haefliger}) $P \circ(\pi_0(q\circ d))^{-1}[(\hat{s},s)]$ is the required preimage of $[(\hat{f},f)]$ under $H_{\alpha}$. So we have proved the surjectivity.

Now to prove injectivity, suppose that $\mathcal{F}_0,\mathcal{F}_1$ are two contact foliations on $M$ such that $H_{\alpha}(\mathcal{F}_0)$ is homotopic to $H_{\alpha}(\mathcal{F}_1)$. Let $H_{\alpha}(\mathcal{F}_0)=(\hat{f}_0,f_0)$ and $H_{\alpha}(\mathcal{F}_1)=(\hat{f}_1,f_1)$. If $\hat{f}:TM\times[0,1]\to \nu\Omega$ is a homotopy between $\hat{f}_0$ and $\hat{f}_1$ in the space $\mathcal E_{\alpha}(TM,\nu\Omega)$, then we have the following factorization of $\hat{f}$:
\begin{equation}
 \xymatrix@=2pc@R=2pc{
TM\times[0,1] \ar@{->}[r]^-{\bar{f}}\ar@{->}[rd] &  f^*(\nu \Omega)  \ar@{->}[r]\ar@{->}[d] & \nu \Omega  \ar@{->}[d]\\
& M \times [0,1]\ar@{->}[r]_-{f} & B\Gamma_{2q}
}\label{Figure:Haefliger5}
\end{equation}
Without loss of generality we can assume that $f_0^*\Omega=\mathcal{F}_0$ and $f_1^*\Omega=\mathcal{F}_1$. By Theorem~\ref{HL} there exists a manifold $N$ with a foliation $\mathcal{F}_N$ and a closed embedding
\[M\times \I \stackrel{s}\longrightarrow N \]
such that $s^*\mathcal{F}_N=f^*\Omega$. As $s_0^*\mathcal{F}_N=f_0^*\Omega=\mathcal{F}_0$ and $s_1^*\mathcal{F}_N=f_1^*\Omega=\mathcal{F}_1$, so $s_0,s_1 \in Tr_{\alpha}(M,\mathcal{F}_N) $. We shall show that $ds_0$ and $ds_1$ are homotopic in $\mathcal F_{\alpha}(TM,\nu \mathcal{F}_N)$. Proceeding as in the first half of the proof, we can define a path between $ds_0$ and $ds_1$ by the following diagram:
\[
\xymatrix@=2pc@R=2pc{
TM \times \I\ar@{->}[r]^-{\bar{f}}\ar@{->}[rd] &
f^*(\nu\Omega)\cong s^*\nu(\mathcal F_N)\ar@{->}[r]^-{}\ar@{->}[d] & \nu \mathcal{F}_N \ar@{->}[d] \\ %\ar@{->}[r]& \nu \Omega \ar@{->}[d]\\
& M\times \I\ar@{->}[r]_-{s} & N %\ar@{->}[r]_-{f'}& B\Gamma_{2q}
}\]
Since the left vertical arrow in diagram (\ref{Figure:Haefliger}) is an isomorphism this proves that $s_0,s_1$ are homotopic in $Tr_{\alpha}(M,\mathcal{F}_N)$. This implies that $\mathcal F_0$ is integrably homotopic to $\mathcal F_1$. This completes the proof of injectivity.\qed

\begin{theorem}Let $(M,\alpha)$ be an open contact manifold and let $\tau:M\to BU(n)$ be a map classifying the symplectic vector bundle $\xi=\ker\alpha$. Then there is a bijection between the elements of $\pi_0(\mathcal E_{\alpha}(TM,\nu\Omega))$ and the homotopy classes of triples $(f,f_0,f_1)$, where $f_0:M\to BU(q)$, $f_1:M\to BU(n-q)$ and $f:M\to B\Gamma_{2q}$ such that
\begin{enumerate}\item $(f_0,f_1)$ is homotopic to $\tau$ in $BU(n)$ and
\item $Bd\circ f$ is homotopic to $Bi\circ f_0$ in $BGL_{2q}$.\end{enumerate}
In other words the following diagrams are homotopy commutative:\\
\[\begin{array}{ccc}
\xymatrix@=2pc@R=2pc{
& &\ \ B\Gamma(2q)\ar@{->}[d]^{Bd}\\
M \ar@{->}[r]_-{f_0}\ar@{-->}[urr]^{f} & BU(q)\ar@{->}[r]_{Bi} & BGL(2q)
}
& \hspace{1cm}&
\xymatrix@=2pc@R=2pc{
&\ \ BU(q)\times BU(n-q)\ar@{->}[d]^{\oplus}\\
M \ar@{->}[r]_-{\tau}\ar@{-->}[ur]^{(f_0,f_1)}& BU(n)
}\end{array}\]
\end{theorem}

\begin{proof}
An element $(F,f)\in \mathcal E_{\alpha}(TM,\nu\Omega)$ defines a (symplectic) splitting of the bundle $\xi$ as
\[\xi \cong (\ker F\cap \xi)\oplus (\ker F\cap \xi)^{d'\alpha}\]
since $\ker F\cap \xi$ is a symplectic subbundle of $\xi$. Let $F'$ denote the restriction of $F$ to $(\ker F\cap \xi)^{d'\alpha}$. It is easy to see that $(F',f):(\ker F\cap \xi)^{d'\alpha}\to \nu(\Omega)$ is a vector bundle map which is fibrewise isomorphism. If $f_0:M\to BU(q)$ and $f_1:M\to BU(n-q)$ are continuous maps classifying the vector bundles $\ker F\cap \xi$ and $(\ker F\cap \xi)^{d'\alpha}$ respectively, then the classifying map $\tau$ of $\xi$ must be homotopic to $(f_0,f_1):M\to BU(q)\times BU(n-q)$ in $BU(n)$ (Recall that the isomorphism classes of Symplectic vector bundles are classified by homotopy classes of continuous maps into $BU$ \cite{husemoller}). Furthermore, note that $(\ker F\cap \xi)^{d'\alpha}\cong f^*(\nu\Omega)=f^*(Bd^*EGL_{2q}(\R))$; therefore, $Bd\circ f$ is homotopic to $f_0$ in $BGL(2q)$.

Conversely, take a triple $(f,f_0,f_1)$ such that
\[Bd\circ f\sim Bi\circ f_0 \text{ and } (f_0,f_1)\sim \tau.\]
Then $\xi$ has a symplectic splitting given by $f_0^*EU(q)\oplus f_1^*EU(n-q)$. Further, since $Bd\circ f\sim Bi\circ f_0$, we have $f_0^*EU(q)\cong f^*\nu(\Omega)$. Hence there is an epimorphism $F:\xi\stackrel{p_2}{\longrightarrow} f_0^*EU(q) \cong f^*\nu(\Omega)$ whose kernel $f_1^*EU(n-q)$ is a symplectic subbundle of $\xi$. Finally, $F$ can be extended to an element of $\mathcal E_\alpha(TM,\nu\Omega)$ by defining its value on $R_\alpha$ equal to zero.\end{proof}

\begin{definition}{\em Let $N$ be a contact submanifold of $(M,\alpha)$ such that $T_xN$ is transversal to $\xi_x$ for all $x\in N$. Then $TN\cap \xi|_N$ is a symplectic subbundle of $\xi$. The symplectic complement of $TN\cap \xi|_N$ with respect to $d'\alpha$ will be called \emph{the normal bundle of the contact submanifold $N$}.}
\end{definition}
The following result is a direct consequence of the above classification theorem.
\begin{corollary} Let $B$ be a symplectic subbundle of $\xi$ with a classifying map $g:M\to BU(q)$. The integrable homotopy classes of contact foliations on $M$ with their normal bundles isomorphic to $B$ are in one-one correspondence with the homotopy classes of lifts of $Bi\circ g$ in $B\Gamma_{2q}$.
\end{corollary}

We end this article with an example to show that a contact foliation on a contact manifold need not be transversally symplectic, even if its normal bundle is a symplectic vector bundle.
\begin{definition}{\em (\cite{haefliger1}) A codimension ${2q}$-foliation $\mathcal F$ on a manifold $M$ is said to be \emph{transverse symplectic} if $\mathcal F$ can be represented by Haefliger cocycles which take values in the groupoid of local symplectomorphisms of $(\R^{2q},\omega_0)$.}
\end{definition}
Thus the normal bundle of a transverse symplectic foliation has a symplectic structure. It can be shown that if $\mathcal F$ is transverse symplectic then there exists a closed 2-form $\omega$ on $M$ such that $\omega^q$ is nowhere vanishing and $\ker\omega=T\mathcal F$.

\begin{example}
{\em
Let us consider a closed almost-symplectic manifold $V^{2n}$ which is not symplectic (e.g., we may take $V$ to be $\mathbb S^6$) and let $\omega_V$ be a non-degenerate 2-form on $V$ defining the almost symplectic structure. Set $M=V\times\mathbb{R}^3$ and let $\mathcal{F}$ be the foliation on $M$ defined by the fibres of the projection map $\pi:M\to V$. Thus the leaves are $\{x\}\times\mathbb{R}^3,\ x\in V$. Consider the standard contact form $\alpha=dz+x dy$ on the Euclidean space $\R^3$ and let $\tilde{\alpha}$ denote the pull-back of $\alpha$ by the projection map $p_2:M\to\R^3$. The 2-form $\beta=\omega_V\oplus d\alpha$ on $M$ is of maximum rank and it is easy to see that $\beta$ restricted to $\ker\tilde{\alpha}$ is non-degenerate. Therefore $(\tilde{\alpha},\beta)$ is an almost contact structure on $M$. Moreover, $\tilde{\alpha}\wedge \beta|_{T\mathcal{F}}$ is nowhere vanishing.

We claim that there exists a contact form $\eta$ on $M$ such that its restrictions to the leaves of $\mathcal F$ are contact.
Recall that there exists a surjective map \[(T^*M)^{(1)}\stackrel{D}{\rightarrow}\wedge^1T^*M \oplus \wedge^2T^*M\] such that $D\circ j^1(\alpha)=(\alpha,d\alpha)$ for any 1-form $\alpha$ on $M$. Let
\[r:\wedge^1T^*M \oplus \wedge^2T^*M\rightarrow \wedge^1T^*\mathcal{F} \oplus \wedge^2T^*\mathcal{F}\] be the restriction map defined by the pull-back of forms and let $A\subset \Gamma(\wedge^1T^*M \oplus \wedge^2T^*M)$ be the set of all pairs $(\eta,\Omega)$ such that $\eta \wedge \Omega^{n+1}$ is nowhere vanishing and let $B\subset \Gamma(\wedge^1T^*\mathcal{F} \oplus\wedge^2T^*\mathcal{F})$ be the set of all pairs whose restriction on $T\mathcal{F}$ is nowhere vanishing. Now set $\mathcal{R}\subset (T^*M)^{(1)}$ as
\[\mathcal{R}=D^{-1}(A)\cap (r\circ D)^{-1}(B).\] Since both $A$ and $B$ are open so is $\mathcal{R}$. Now if we consider the fibration $M\stackrel{\pi}{\rightarrow}V$ then it is easy to see that the diffeotopies of $M$ preserving the fibers of $\pi$ sharply moves $V\times 0$ and $\mathcal{R}$ is invariant under the action of such
diffeotopies. So by Theorem~\ref{T:gromov-invariant} there exists a contact form $\eta$ on $Op(V\times 0)=V\times\mathbb{D}^3_{\varepsilon}$ for some $\varepsilon>0$, and $\eta$ restricted to each leaf of the foliation $\mathcal F$ is also contact. Now take a diffeomorphism $g:\mathbb{R}^3\rightarrow \mathbb{D}^3_{\varepsilon}$. Then  $\eta'=(id_V\times g)^*\eta$ is a contact form on $M$. Further, $\mathcal{F}$ is a contact foliation relative to $\eta'$ since $id_V\times g$ is foliation preserving.

But $\mathcal{F}$ can not be transversal symplectic because then there would exist a closed 2-form $\beta$ whose restriction to $\nu \mathcal{F}=\pi^*(TV)$ would be non-degenerate. This would imply that $V$ is a symplectic manifold contradicting our hypothesis.}

\end{example}

\newpage
\printindex

\end{document}